\documentclass[a4paper,11pt]{article}
\usepackage[T1]{fontenc}
\usepackage[latin1]{inputenc}   
\usepackage[english]{babel}  
\usepackage{amsthm,amssymb,amsfonts,amsmath,color,graphicx,float,caption,hyperref}
\usepackage[authoryear,round,sort]{natbib}
\usepackage{geometry}
\newtheorem{Lemma}{Lemma}
\newtheorem{Prop}{Proposition}

\newtheorem{Theo}{Theorem}

\newtheorem{Remark}{Remark}

\newcommand{\R}{\mathbb{R}}

\newcommand{\indic}[1]{ \mathbf{1}_{#1}}

\newcommand{\N}{\mathbb{N}}
\newcommand{\E}{\mathbb{E}}

\newcommand{\X}{\mathbb{X}}  
\newcommand{\dhaus}{\operatorname{Haus}}

\newcommand{\TV}{\operatorname{TV}}

\title{Rates of convergence for robust geometric inference}

\author{F. Chazal, P. Massart and B. Michel}

\date{\today}

\begin{document}

\maketitle

\begin{abstract}
Distances to compact sets are widely used in the field of Topological Data Analysis for inferring geometric and topological features from point clouds. In this context, the distance to a probability measure (DTM) has been introduced by \cite{chazal2011geometric} as a robust alternative to the distance a compact set. In practice, the DTM can be estimated by its empirical counterpart, that is the distance to the empirical measure (DTEM). In this paper we give a tight control of the deviation of the DTEM. Our analysis relies on a local analysis of empirical processes. In particular, we show that the rate of convergence of the DTEM directly depends on the regularity at zero of a particular quantile function which contains some local information about the geometry of the support.  This quantile function  is the relevant quantity to describe precisely how difficult is a geometric inference problem. Several numerical experiments illustrate the convergence of the DTEM and also confirm that our bounds are  tight.
\end{abstract}

\section{Introduction and motivation}
The last decades have seen  an explosion in the amount of available data in almost all domains of science, industry, economy and even
everyday life. These data, often coming as point clouds embedded in Euclidean spaces, usually lie close to some lower dimensional geometric structures (e.g. manifold, stratified space,...) reflecting properties of the system from which they have been generated. Inferring the topological and geometric features of such multivariate data has recently attracted a lot of interest in both statistical and computational topology communities. 

Considering point cloud data as independent observations of some common probability distribution $P$ in $\R^d$, many statistical methods have been proposed to infer the geometric features of the support of $P$ such as principal curves and surfaces \cite{HastieStuetzle89}, multiscale geometric analysis \cite{adh-amdfs-06}, density-based approaches \cite{gpvw-pdgf-09} or  support estimation, to name a few.  Although they come with statistical guarantees these methods usually do not provide geometric guarantees on the estimated  features.

On another hand, with the emergence of Topological Data Analysis \citep{carlsson2009topology}, purely geometric methods have been proposed to infer the geometry of compact subsets of $\R^d$. These methods aims at recovering precise 
geometric information of a given shape -- see, e.g. \cite{smrnnuagCL,fhshcrsNSW,ncaosiCCL,chazal2008scm}.  Although these methods come with strong topological and geometric guarantees they usually rely on sampling assumptions that do not apply in statistical settings. In particular, these methods can be very sensitive to outliers. Indeed, they generally rely on the study of the sublevel sets of distance functions to compact sets. In practice only a sample drawn on, or close, to a geometric shape is known and thus only a distance to the data can be computed. The sup norm between the distance to the data and the distance to the underlying shape being exactly the Hausdorff distance between the data and the shape, we see that the statistical analysis of standards TDA methods boils down to the problem of support estimation in Hausdorff metric. This last problem  has been the subject of much study in statistics  \citep[see for instance ][]{devroye1980detection,cuevas2004boundary,singh2009adaptive}. Being strongly dependent of the estimation of the  support  in Hausdorff metric, it is now clear why  standard TDA methods may be very sensitive to outliers.

%

To  provide a more robust approach of TDA,
a notion of distance function to a measure (DTM) in $\R^d$ has been introduced by \cite{chazal2011geometric} as a robust alternative to the classical distance to compact sets.  Given a probability distribution $P$ in $\R^d$ and a real parameter $0 \leq u \leq 1$, \cite{chazal2011geometric} generalize the notion of distance to the support of $P$ by the function 
\begin{equation}
\label{eq:DTMPrelimdef}
\delta_{P, u} : x \in \R^d  \mapsto \inf\{ t > 0 \, ; \,  P(\bar B (x,t) ) \geq  u  \} 
\end{equation}
where $\bar B (x,t)$ is the closed Euclidean ball of center $x$ and radius $t$. For $u=0$, this function coincides with the usual distance function to the support of $P$. For higher values of $u$,  it is larger than the usual distance function since a  portion of mass $u$ has to be included in the ball centered on $x$.  To avoid issues due to discontinuities of the map $P \to \delta_{P, u}$, the distance to measure (DTM) function with parameter $m \in [0,1]$ and power $r \geq 1$ is defined by 
\begin{equation}
\label{eq:DTMdef}
d_{P,m,r}(x)  : x \in \mathcal \R^d  \mapsto  \left(  \frac 1 {m} \int _0 ^{m} \delta_{P,u}^r(x)  d u \right)^{1/r}.
\end{equation}

It was shown in \cite{chazal2011geometric}  that the DTM shares many properties with classical distance functions that make it well-adapted for geometric inference purposes (see Theorem~\ref{Theo:StabWp} in Appendix~\ref{sec:RatesStability}). First, it is stable with respect to perturbations of $P$ in the Wasserstein metric .  
This property implies that the DTM associated to close distributions in the Wasserstein metric have close sublevel sets. Moreover, when $r=2$, the function $d_{P,m,2}^2$ is semiconcave ensuring strong regularity properties on the geometry of its sublevel sets.  Using these properties, \cite{chazal2011geometric}
 show that, under general assumptions, if $\tilde P$ is a probability distribution approximating $P$, then the sublevel sets of $d_{\tilde P,m,2}$ provide a topologically correct approximation of the support of $P$.
The introduction of DTM has motivated further works and applications in various directions such as topological data analysis \cite{buchet2015socg}, GPS traces analysis \cite{ccgjs-ddts-11}, density estimation \cite{biau2011weighted}, deconvolution \cite{CCDM11} or clustering \cite{chazal2013persistence} just to name a few. Approximations, generalizations and variants of the DTM have also been recently considered in \cite{guibas2013witnessed,buchet2015efficient,phillips2014goemetric}. However no strong statistical analysis of the DTM has not been proposed so far.


In practice, the measure $P$ is usually only known through a finite set of observations  $\X_n = \{ X_1,\dots, X_n\} $ sampled from $P$, raising the question of the approximation of the DTM.
A natural idea to estimate the DTM from $\X_n$ is to plug the empirical measure $P_n$ instead of $P$ in the definition of the DTM. This ``plug-in strategy" corresponds to computing the distance to the empirical measure (DTEM). It can be applied with other estimators of the measure $P$, for instance in~\cite{CCDM11} it was proposed to plug  a deconvolved measure  into the DTM. 

For $m = \frac k n $, the DTEM satisfies
  $$ d^r_{P_n,k/n,r}(x)  := \frac 1 k \sum _{j=1} ^k  {\|x -  \X_n \|}^r_{(j)} \,  ,$$
where $\|x -  \X_n \|_{(j)}$ denotes the distance between $x$ and its $j$-th neighbor in $\{X_1,\dots, X_n\}$.  This
quantity can be easily  computed in practice since it only requires the distances between $x$ and the sample points. 

Let us introduce 
\begin{eqnarray}
\Delta_{n,m,r}(x) &:=& d^r_{P_n,m,r}(x) - d^r_{  P,m,r}(x) \label{eq:defDelta}
\end{eqnarray}
and
\begin{equation*}
\label{eq:defDeltatilde}
\tilde \Delta_{n,m,r}(x) := d_{P_n,r, m}(x) - d_{  P,m,r}(x).
\end{equation*}
The aim of this paper is to study the deviations and the rate of convergence of  $\Delta_{n,m,r}(x)$. The functional convergence of the DTEM has  been studied recently in \cite{chazal2014robust} where it is shown that the parametric  convergence rate in $1/\sqrt n$ is achieved under reasonable assumptions. 
In this paper we address the question of the convergence in probability and the rate of convergence in expectation of $\Delta_{n,m,r}(x)$, both from an asymptotic and non asymptotic point perspective. 

The stability properties of DTM with respect to Wasserstein metrics 
suggests that this problem could be addressed using known results 
about the convergence of empirical measure $P_n$ to $P$ under Wasserstein metrics. This last problem has been the subject of many works in the past \citep{RR98,BGM99,BGU05} and it is still an active field of research \citep{FG,DSS13}. Contrary to the context of TDA with the standard distance function, where stability result provide optimal rates of convergence (see \citet{chazalconvergenceJMLR}), we show in the paper that Wasserstein stability does not lead to optimal results for the DTM. Moreover, such a basic approach does not provide a correct understanding of the influence of the parameter $m$ (see Appendix~\ref{sec:RatesStability}).

We adopt an alternative approach based on the observation that the DTM  only depends on a push forward measure of $P$ on the real line. Indeed, the DTM can be rewritten as follows:
\begin{equation}     
\label{eq:DTMQuantile}  
 d^r_{P,m,r}(x) = \frac 1m \int_0^m    F_{x,r}^{-1}(u) d u ,
\end{equation}
where $F_{x,r}^{-1}$ is the quantile function of the push forward  probability measure of $P$ by the function $\|x-\cdot \| ^r$ (see appendix~\ref{subsec:proofStab} for a rigorous proof). Then we have
\begin{eqnarray}
\Delta_{n,m,r}(x) &:=& \frac 1 m \int_0^{m} \left\{   F_{x,r,n}^{-1}(u)   -    F_{x,r} ^{-1}(u) 
  \right\} du \label{eq:DeltaW1} ,
\end{eqnarray}
where $F_{x,r,n}$ is the empirical distribution function of the observed distances (to the power $r$): $  {\|x -  X_1\|}^r$, $\dots$, ${\|x -  X_n\|}^r$. 
We study the convergence of $\Delta_{n,m,r}(x)$  to zero  with both an asymptotic and non asymptotic points of view. An asymptotic approach means that we take $k = k_n := m n $ for some
fixed $m$ and we study the mean rate of convergence to zero of $\Delta_{n,\frac{k_n}{n},r}(x) $. 
A non asymptotic approach means that $n$ is fixed and then the problem is to get a tight expectation  bound on  $\Delta_{n,\frac{k}{n},r}(x)
$. In particular,  we  are particularly interested in the situation where $\frac kn$ is chosen very  close to zero.  This situation is of  primary interest since it 
corresponds to the realistic situation where we use the  DTM to clean the support from a small proportion of outliers.

Our results rely on a local analysis of the empirical process to compute tight deviation bounds of $\Delta_{n,\frac{k}{n},r}(x)$. More precisely, we use a  sharp control of a supremum defined on the uniform empirical process. Such local analysis  has been successfully applied in the literature about non asymptotic statistics, for instance  \cite{mammen1999smooth} obtain fast rates of convergence in classification. For a more general presentation of these ideas in model selection, see~\cite{Massart:07} and in particular Section~1.2 in the Introduction of this monograph.

We show that the rate of convergence of $\Delta_{n,\frac{k}{n},r}(x)$ directly depends on the regularity at zero of $F_{x,r}^{-1}$.  This quantile function appears to be the relevant quantity to describe precisely how difficult is a geometric inference problem.  The second contribution of this paper is relating the regularity of the quantile function $F_{x,r} ^{-1}$ to the geometry of the support, establishing a link between the complexity of the  geometric problem and a purely probabilistic quantity.

Our main results, the deviations bounds and the rate of convergence of $\Delta_{n,\frac{k}{n},r}(x)$ derived from  the local analysis, are given in Section~\ref{sec:mainresults}. These results are given in terms of the regularity of the quantile function $F_{x,r}^{-1}$. Generally speaking, it is not easy to determine  what is the regularity of the quantile function $F_{x,r}^{-1}$ given a distribution $P$ and an observation point $x \in \R^d$. Indeed, it depends on the shape of the support of $P$, on the way the measure $P$ is distributed on its support and on the position of $x$ with regards to the support of $P$. This is why, in the results given in Section~\ref{sec:mainresults}, the assumptions are made directly on the quantile functions $F_x^{-1}$. Section~\ref{sec:AppliRd-discu} is then devoted to the geometric interpretation of these results and their assumptions. In Section~\ref{subs:expe}, several numerical experiments illustrate the convergence of the DTEM and also confirm that our bounds are sharp. Rates of convergence derived from stability results of the DTM are presented in Appendix~\ref{sec:RatesStability}. Proofs and background about empirical processes and quantiles can be found in the appendices also.

\medskip

{\it Notation.} Let  $a \wedge  b$ and $a \vee b$ denotes the minimum and the maximum between two real numbers $a$ and $b$. The Euclidean norm on $\R^d$ is $\|  \cdot \|$. The open Euclidean ball of center $x$ and radius $t$ is denoted by $B (x,t)$. For some point $x$ and a compact set $K$ in $\R^d$, the distance between $x$ and $K$ is defined by  $\| K - x \|  :=  \inf_{y \in K} \|y -x \| $. The Hausdorff distance between two compact sets  $K$ and $K'$ is denoted by $\dhaus \left( K,K' \right)$. A probability distribution on $\R$ defined by a distribution function $F$ is denoted by $dF$. The quantile function $F^{-1}$ of $dF$ is defined by 
$$ F^{-1}(u)  :=  \inf \{ t  \in \R \, , \,  F(t) \geq u  \},  \quad   0  < u  < 1 .$$
By monotonicity, the quantile function $ F^{-1}$   can be extended in $0$ and at $1$ by setting $ F^{-1} (0) =  \inf \{ t  \in \R \, , \,  F(t)  >0   \}, $ and $ F^{-1} (1) =  \sup \{ t  \in \R \,, \,  F(t)  < 1  \} .$ Finally, for two positive sequences $(a_n)$ and $(b_n)$, we use the standard notation $a_n \lesssim b_n$ if there exists a positive constant $C$ such that $a_n \leq C b_n$.

 \section{Main results } \label{sec:mainresults}

We fix $r\geq 1$ and we henceforth write  $F_x $ for $F_{x,r} $  to facilitate the reading. In the same way we will
use the
notation $F_x^{-1} $, $\tilde \Delta_{P,m}$, $d_{P,m}$ since there is no ambiguity on the power term $r$.

Given an observation point $x \in  \R^d$, we introduce the modulus of continuity  $\tilde \omega_{x}$ of
$F_x^{-1}$ (possibly infinite) which is defined for any $v   \in (0,1]$ by
\begin{equation*}
\tilde \omega_{x} (v) := \sup_{(u,u') \in [0,1]^2 , \, \| u-u' \| \leq v} |F^{-1}_{x} (u) -  F^{-1}_{x} (u')|   .
\end{equation*}
Note that  the fact that $\tilde \omega_x $ is finite  is equivalent to the fact that  the support of $P$ is  bounded. 
An extensive discussion about the relation between the  measure $P$ and the modulus of continuity of $F_x^{-1}$
is proposed in Section~\ref{sec:AppliRd-discu}. The function $\tilde \omega_{x}$ being non decreasing and non negative,
it has a non negative limit $\tilde \omega_{x}(0^+)$  at zero.  In particular we do not assume here that $\tilde
\omega_{x}(0^+) =0$. In other terms we do not assume that $F_x^{-1}$ is continuous. We extend $\tilde \omega_{x}$ at zero
by taking $\tilde \omega_{x}(0) =\tilde \omega_{x}(0^+)$. 

In the following, it will be sufficient in our results  to consider upper bounds on the modulus of continuity, that is a
non negative function $  \omega_x$ on  $[0,1]$ such that $  \omega_x(v) \geq \tilde \omega(v)$ for any $v \in
[0,1]$. A modulus of continuity being a non decreasing function, we will assume  that such an upper bound $\omega_x$ 
is non decreasing  on $[0,1]$. For technical reasons and without loss of generality, we will also assume that $ \omega_x
$ is a continuous  function, which  takes its values in $[\omega(0),\omega(1)] \subset  \bar \R^+$.  For such a function
 $\omega_x$  we  also introduce its inverse function $\omega_x^{-1}$ which is defined on $[\omega(0),\omega(1)]$. We
extend this function to $\R^+$ by taking $\omega_x^{-1} (t) = 0 $ for any $ t \in [0,\omega(0)]$ and $\omega_x^{-1} (t)
= 1 $ for any $ t
\geq \omega(1)$. In particular, $\omega_x^{-1} \left(\omega_x(u)\right) = u$ for any $u \in [0,1]$. 

\medskip 

In this section, we show that the rate of convergence of  $ \Delta_{n,\frac
kn}(x) $ is of the order of  $\frac {\omega_x \left( \frac kn\right)}{\sqrt k}  $.

\subsection{Local analysis of the distance to the empirical measure in the bounded case}

We first consider the behavior of the distance to the empirical measure when the observations  $X_1,\dots,X_n$  are sampled from a distribution $P$  with compact support in $\R^d$. Let $F_x^{-1}$ be the quantile function of  $\|x - X_1 \|^r$ and let $\Delta_{n,\frac kn}$  be defined by \eqref{eq:defDelta}.
\begin{Theo} \label{Theo:ExpoDTMBounded}
Let $x$ be a fixed observation point in  $\R^d$. 
Assume that $\omega_x : [0,1] \rightarrow \R^+$ is an upper bound on the modulus of continuity of $F_x^{-1}$. Assume moreover that $\omega_x$ is an increasing and continuous function on $[0,1]$.  
\begin{enumerate}
\item For any $\lambda >0$, if $k <  \frac n 2$ then
\begin{multline}
\label{deviationMethodEmpirProc}
\frac{P\left( | \Delta_{n,\frac kn}(x) | \geq  \lambda \right) }{2}   \leq 
\exp \left( - \frac 1 {64 }   \frac { k \lambda ^2 } {\left[ F_x^{-1}\left(\frac k n\right)
 - F_x^{-1}(0)\right]^{2}}    
\right) 
+ \exp \left(  - \frac 3{16} \frac  {k \lambda} { F_x^{-1}\left(\frac k n\right)
 - F_x^{-1}(0)}       \right) \\ 
+ \exp \left(  - \frac{   n^2}{4 k }  \left\{ 
\omega_x^{-1}   \left(  k ^{1/4}    \sqrt{ \frac \lambda 8  \omega_x \left( \frac{2\sqrt{k}} n
\right)}     \right) 
\right\} ^2  \right) 
+   \exp \left(  -  \frac{   3 n}8   \omega_x^{-1}   \left(  k^{1/4}    \sqrt{ \frac \lambda 2
\omega_x\left( \frac{2\sqrt{k}} n \right)}     \right) \right)     \\
 +   \exp \left(  -  \frac{\sqrt k}8    \frac  \lambda  { \omega_x
\left(\frac{2\sqrt{k}} n \right)}      \right)  
+   \exp \left(  -  \frac {3   k^{3/4}} 4     \sqrt{ \frac  \lambda  {
2 \omega_x \left(\frac{2\sqrt{k}} n \right)}  } \right) := \square(\lambda) , 
\end{multline}
otherwise
\begin{multline*}
\frac { P\left( | \Delta_{n,\frac kn}(x) | \geq  \lambda \right)}{2}    \leq 
 \exp \left( - 2  n  \lambda^2  \left[ \frac{\frac k n}{  F^{-1}\left(\frac k n\right) - F^{-1}(0)} \right]^{2}   
\right) 
  \\  
  + \exp \left(  - 2 n \left\{  \omega_x^{-1} \left( \sqrt{    \frac k {\sqrt n}  \omega_x  \left( \frac 1{\sqrt{n}}
\right)      \frac \lambda 2 }\right)  \right\} ^2  \right)  
+     \exp \left(  -   \frac{k}  {\sqrt n   \omega_x  \left( \frac 1{\sqrt{n}} \right)  } \lambda      \right)    .
\end{multline*}
Furthermore, in all cases we have $ P\left(  | \Delta_{n,\frac kn}(x) | \geq \lambda \right)   =0$ 
for any $\lambda > \omega_x(1)$.
\item  Assume moreover that $\omega_x (u) / u $ is a non increasing function, then for any $k \in \{1,\dots, n\}$:
\begin{eqnarray}
 \E\left( | \Delta_{n,\frac kn}(x) | \right)     
 & \leq&  \frac {C}{\sqrt k}    \left\{  \left[ F_x^{-1}\left(\frac k n\right)
 - F_x^{-1}(0)\right]   +     \omega_x    \left(  \frac{\sqrt{k}}{n}  \right)  \right\}  \label{eq:InitilalBound}  \\
& \leq&    \frac {2C}{\sqrt k}    \omega_x \left( \frac kn\right)  \label{eq:FinalBound} ,
\end{eqnarray}
where $C$ is an absolute constant. 
\end{enumerate}
\end{Theo}
The proof of the Theorem is based on a particular decomposition of $\Delta_{n,\frac kn}(x)$, see
Lemma~\ref{DecomposDeltaJerome} in Appendix~\ref{subsec:proofStab}. This decomposition allows us to consider the deviations of the empirical process rather than the deviations of the quantile process. The proof is given in Appendix~\ref{app:proofs}.
\medskip

Let us now comment on the final bound on expectation \eqref{eq:FinalBound}. This bound can be rewritten as follows:
\begin{equation} \label{eq:mainupperbound}
\E\left|  \Delta_{n,\frac kn}(x) \right|  \lesssim
\frac nk  \frac 1 {\sqrt n} \sqrt{  \frac k n } \omega_x \left( \frac
kn\right) .
\end{equation}
The term $  \frac nk$ comes from the definition of the DTM, it is the renormalization by the mass proportion $  \frac kn$. The term $\frac 1 {\sqrt n}$ corresponds to a classical parametric rate of convergence. The term $\sqrt{ \frac k n}$ is obtained thanks to a local analysis of the empirical process. More precisely, it derives from a sharp control of the variance of the supremum over the uniform empirical process. The term $\omega_x \left( \frac kn\right)$ corresponds to the statistical complexity of the problem, expressed in terms of the regularity of the quantile function $F_x^{-1}$.

Theorem~\ref{Theo:ExpoDTMBounded} can be interpreted with either an asymptotic or a non asymptotic point of view. Taking a non asymptotic approach, we consider  $n$ as fixed. A first result here is that we obtain sharp upper bounds for small values of $\frac kn$. In the most favorable case where $\tilde \omega_x (u) \sim u$, we see  in \eqref{eq:FinalBound}
that an upper bound of the order of $\frac 1n$ is reached.  This is direct consequence of the local analysis we use to control the empirical process in the neighborhood of the origin. As mentioned before, assuming that $\frac kn$ is very small corresponds to the realistic situation where we use the DTM to clean the support from a small proportion of outliers.

Now, taking an asymptotic approach, a second result of Theorem~\ref{Theo:ExpoDTMBounded} is that it  allows us to
consider the asymptotic behavior of $\Delta_{n,\frac kn}(x)$ under all possible regimes, that is for all sequences
$(k_n)_{n \in \N}$. For instance, with the classical approach where $k_n$ is such that  $k_n/ n =m$  for some fixed
value $m \in (0, 1)$, we then
obtain the parametric rate of convergence $1/ \sqrt n$, as in the asymptotic functional results given in
\cite{chazal2014robust}. 

Another key fact about Theorem~\ref{Theo:ExpoDTMBounded} is that the upper bound \eqref{eq:InitilalBound}  
depends on the regularity of $F_x^{-1}$ through the function   
$$ \Psi_x : m  \in (0,1) \mapsto \frac{  \omega_x(m)}{ \sqrt m} .$$
Moreover,  if $\omega(0^+) = 0$, we see that the upper bound \eqref{eq:InitilalBound} 
depends on the
regularity of $F_x^{-1}$ only at $0$ for $n$ large enough. For instance, if $k_n$ is such that  $k_n/ n =m$  for some fixed value $m \in (0, 1)$ such that  $F_x^{-1}(m) > F_x^{-1}(0)$, coming back to
\eqref{eq:InitilalBound}, we find that for $n$ large enough:
$$ \omega_x    \left(  \frac{\sqrt{k_n }}{n} \right)   =  \omega_x    \left( \frac{ 1}{\sqrt n} m \right) < F_x^{-1}\left(m \right)
- F_x^{-1}(0).$$
In this context, the right hand term of Inequality \eqref{eq:InitilalBound} is  of the order
of $ \frac {\widetilde \Psi_x(\frac {k_n}n)}{\sqrt {k_n}} $ where 
$$ \widetilde \Psi_x  : m  \in (0,1) \mapsto  \frac{F_x^{-1}(m) - F_x^{-1}(0)}{ \sqrt{m}}  .$$

\medskip

We now give additional remarks   about  Theorem~\ref{Theo:ExpoDTMBounded}.
\begin{Remark} \label{rqHolder}  If the quantile function $F_{x}^{-1}$ is $\eta$-H\"{o}lder, then 
$\omega_x (u) = A u ^{\eta} $ for some constant $A  \geq 0$ and thus we have
$$
 \E\left( | \Delta_{n,\frac kn}(x) | \right)     
\lesssim    \frac 1{\sqrt n} \left[ \frac k n \right] ^  {\eta - 1/2}  .
$$
Remember that H\"{o}lder functions with power $\eta > 1$ are constants, we can thus assume that $\eta \leq 1$.
\end{Remark}
\begin{Remark}
Assuming that $\omega_x(u) / u$ is a non increasing function roughly means that $\omega_x$ is a concave function. Our
result is thus satisfied if we can find an concave function which is an upper bound on the modulus of continuity of the
quantile function. We show in Section~\ref{sub:abstandard} that it is satisfied for a large  class of measures.
\end{Remark}
\begin{Remark}
For values of $\frac kn$ not close to zero, the rate is consistent with  the upper
bound~\eqref{upperW1} deduced from the approach based on the stability results (see Appendix~\ref{sec:RatesStability}). However,
Theorem~\ref{Theo:ExpoDTMBounded} is  more satisfactory since it  describes the statistical
complexity of
the problem through the  regularity of the quantile function. 
\end{Remark}
\begin{Remark} The application $u \mapsto u^{1/r}$ is $1/r$- H\"{o}lder on $\R^+$ with H\"{o}lder constant $1$ since  $1/r < 1$.
It yields: 
\begin{equation} \label{DeltaDeltatilde}
 | \tilde \Delta_{n,\frac kn,r}(x) | \leq  | \Delta_{n,\frac kn,r}(x) | ^{1/r}.
\end{equation}
where $\tilde \Delta_{n,\frac kn,r}(x)$ is defined by \eqref{eq:defDeltatilde}.
We deduce an expectation bound on $\tilde  \Delta_{n,\frac kn,r}(x) $ from 
Jensen's Inequality and Inequality~\eqref{DeltaDeltatilde}:
$$ \E\left( | \tilde \Delta_{n,\frac kn,r}(x) | \right)   \lesssim   \left[  \frac 1 {\sqrt n} \left[ \frac k n
\right] ^  {- 1/2}   \omega_x \left( \frac kn\right)  \right]^{1/r}  .$$ 
\end{Remark}
\begin{Remark} As already mentioned before, to prove Theorem~\ref{Theo:ExpoDTMBounded}, we consider the deviations of the empirical process rather than the deviations of the quantile process. Indeed, the more direct approach that consists in directly controlling the deviations  of the quantile process gives slower rates. More precisely, using Proposition~\ref{prop:Quantileunif} given in Appendix~\ref{sec:Gn} borrowed from
\cite{shorack2009empirical}, it can be shown
that
$$
 \E\left( | \Delta_{n,\frac kn}(x) | \right)  \lesssim    \omega_x    \left(  \frac{\sqrt{ k}}{n}  \right) .
$$
For instance, if $\omega_x (u) = A u ^{\eta} $, we obtain $
 \E\left( | \Delta_{n,\frac kn,r}(x) | \right)       \lesssim   \left(  \frac{\sqrt k} n  \right)^{\eta}  
$ which is slower than the rate given in Remark~\ref{rqHolder}.
\end{Remark}

\medskip

To complete the results of Theorem~\ref{Theo:ExpoDTMBounded},  we give below a lower bound using Le Cam's lemma (see Lemma \ref{Lem:Lecam} in Appendix~\ref{sec:Gn}). Let $\omega$ be a
continuous and increasing function on $[0,1]$ and let $x \in\R^d$.
We introduce that class of probability measures:
 $$ \mathcal P_{\omega} := \left\{P \: \textrm{is a probability measure on }\R^d \textrm{ such that } \omega(u) \geq
\tilde \omega_x (u)  \textrm{ for any  } u \in [0,1]   \right\}. $$ 
In the previous definition, the function $ \tilde \omega$ is as before the modulus of continuity of  the quantile
function of the    distribution of the push-forward measure of $P$ by the function $y\mapsto \| y - x \| ^r$.
\begin{Prop}  \label{prop:lowerboundDTM}
Assume that there exists $P \in  \mathcal P_{\omega}$,  $c >0$ and $\bar u \in (0,1)$,  such that 
\begin{equation} \label{OmegaEstUnQuantile}
c \left[ F_x^{-1} (u) - F_x^{-1} (0)\right] \geq \omega(u)  \quad \textrm{ for any }u \in (0,\bar u] .
\end{equation}
 Then,  there exits  a constant $C$ which only depends on $c$, such that  for any $k \leq \bar u n$.
\begin{eqnarray*}
\sup _{ P \in  \mathcal P_{\omega}} \E\left( |  \Delta_{n,\frac kn,r}(x) | \right)  & \geq  & 
 \inf _{\hat d_n(x)} \sup _{ P \in  \mathcal P_{\omega}}   \E\left( \left| \hat d_n^r (x)-  d_{P,m,r}^r(x) 
\right|\right) \\
   &\geq  &  C \frac n k   \frac 1n  \omega \left(\frac {k-1} n \right),
\end{eqnarray*}
where the infimum is than over all the estimator $\hat d_n(x)$ of $d_{P,m,r}(x)$ defined from a sample $X_1,\dots, X_n$ 
of  distribution $P$.
\end{Prop}
The Assumption \eqref{OmegaEstUnQuantile} is not very strong. It  means that $\omega$ is not a too large upper bound on
the modulii of continuity of the quantile functions. More precisely, it says that  there exists a distribution $P \in
\mathcal P_{\omega}$ for which  $\omega$ can be comparable to the modulus of continuity of the quantile functions
$F_x^{-1}$ in the neighborhood of the origin.

Note that this lower bound matches with the upper bound of Theorem~\ref{Theo:ExpoDTMBounded}
 when $k$ is very small since it is of the order of  $\omega \left(\frac {k} n \right)$.  Providing the correct
lower bound for all values of $k$ is not obvious. As far as we know  there is no standard method in the literature for computing
lower bounds for this kind of functional and we consider that this issue is beyond the scope of this paper.

\subsection{Local analysis of the distance to the empirical measure in the unbounded case}

The previous results  provide a  description of the fluctuations and mean rates of convergence of the empirical distance to measure. However, when the support of $P$ is not bounded, the quantile function $F_x^{-1}$ tends to infinity
at $1$ and the modulus of continuity of $F_x^{-1}$ is not finite. In such a situation, Theorem~\ref{Theo:ExpoDTMBounded} can not be applied. We now propose a second result about the fluctuations of the DTEM, under weaker
assumptions on the regularity of  $F_x^{-1}$. The following result shows that under a weak moment assumption,  the rate of convergence is the same as for the bounded case, up to a term decreasing exponentially  fast to zero. 
\begin{Theo} \label{Theo:ExpoDTMUnBounded}
Let  $\bar m \in (0,1)$  and some observation point $x \in \mathcal \R^d$. Assume that $ \omega_{x,\bar m}  $ is
an upper bound of the modulus of continuity of $F_x^{-1}$ on $(0,\bar m]$: for any $u,u' \in [0,\bar m]^2$,
\begin{equation}  
\label{eq:OmegaPonc} |F^{-1}_{x} (u) -  F^{-1}_{x} (u')| \leq  \omega_{x,\bar m} ( |u- u'|) .
\end{equation}
 Assume moreover that $\omega_{x,\bar m}$ is increasing and continuous function on
$[0,\bar m]$. 
Then, for any  $k < n\left(\frac 12 \wedge \bar m\right)$ and any $\lambda>0$:
\begin{multline}
\label{deviationUnbounded}
\frac{P\left( | \Delta_{n,\frac kn}(x) | \geq  \lambda \right) }{2}   \leq 
\square(\lambda) + 
  \exp \left(  -  \frac{ \sqrt k }{ 8 }     \exp \left[ \frac{n^2}{4k} \left( \bar m - \frac kn \right)^2   \right] 
   \lambda \right)  \\  +    
     \exp \left(  -  \frac 38   \exp \left[ \frac{n^2}{8k} \left( \bar m - \frac kn \right)^2   \right]  
    k^{\frac 38}  \sqrt{ \frac \lambda 2} \right) \\
+  \left\{2 \exp \left(- \frac{n^2}{2k} \left\{   \bar m - \frac kn \right\}^2   \right) \right\} \wedge \left\{ 
\frac{ {n \choose k-1}}2  \left[1 - F  \left( \sqrt{ \frac{ \lambda}{ 6}} \frac kn \right)   \right] ^{n-k+1} \right\}
\end{multline}
where $\square(\lambda) $ is the upper bound given in Theorem~\ref{Theo:ExpoDTMBounded}, with $\omega_{x}$
replaced by $\omega_{x,\bar m}$. 
Assume moreover that $\omega_{x,\bar m} (u) / u $ is a non increasing function and that $P$ has a moment of order $r$.
Then
\begin{equation*} \label{BorneEsp}
\E\left|  \Delta_{n,\frac kn}(x) \right|  \leq
 \frac C{\sqrt n} \left[ \frac k n \right] ^  {- 1/2}    \left\{  \left[ F_x^{-1}\left(\frac k n\right)
 - F_x^{-1}(0)\right]
 +        \omega_{x,\bar m} \left( \frac{\sqrt{k}}{n}  \right)  \right\}
 + C_{x,r,\bar m} \sqrt k \exp \left[- \frac{n^2}{4k} \left( \bar m - \frac kn \right)^2   \right]. 
\end{equation*}
where $C$ is an absolute constant and $C_{x,r,\bar m}$ only depends on the quantity $\E {\|X - x \|}^r $ and on $\bar
m$.
\end{Theo}
As for the bounded case, if $\omega(0^+) = 0$ and if $F_x^{-1}(m) > F_x^{-1}(0)$, then the rate of convergence  is 
still of the order of $ \frac {\widetilde \Psi_x(m)}{\sqrt n} $. Note that this result is interesting even when the measure $P$ is supported on a compact set. Indeed, assume that the
quantile function $F_x^{-1}$ is not continuous, then $\tilde \omega_x^{-1}(0) >0 $. However, if  $F_x^{-1}$ is smooth
in the neighborhood of zero, for $\bar m$ small enough the assumption \eqref{eq:OmegaPonc} may be satisfied with a
function $\omega_{x,\bar m}$ which can be very small in the neighborhood of zero. Theorem~\ref{Theo:ExpoDTMUnBounded} may  provide better
bounds in this context than those given by Theorem~\ref{Theo:ExpoDTMBounded}. This fact also
confirms that the deviations of the DTEM mainly relies on the local regularity of the quantile function  $F_{x}^{-1}$ at
the origin rather then on its global regularity.

\subsection{Convergence of the distance to the empirical measure for the sup norm}

The previous results address the pointwise fluctuations of the DTEM. We now consider the same problem for the sup norm metric on a compact domain $\mathcal D$ of $\R^d $. Let $N(\mathcal D, t)$ be the covering number of $\cal D$, that is the smallest number of balls
$B(x_i,t)$ with $x_i \in \cal D$, such that $\bigcup_{i} B(x_i,t) \supset \cal D$. Since the domain $\cal D$ is compact,
 there exists two positive constants $c$ and $\nu \leq d$ such that
for any $t >0$ :
  $$N(\mathcal D, t) \leq c t^{-\nu} \vee1.$$
We assume that there exists a function $\omega_{\mathcal D} : (0,1] \rightarrow \R^+$ which uniformly upper bounds
the modulus of continuity of the quantile functions $(F_x^{-1})_{x \in \mathcal D}$: for any $u,u' \in (0,1]^2$ and for
any $x \in \mathcal D$: 
\begin{equation*}
\label{eq:OmegaUnif} |F^{-1}_{x} (u) -  F^{-1}_{x} (u')| \leq  \omega_{\cal D}( |u- u'|) .
\end{equation*}
We also assume as before  that $\omega_{\mathcal D}$ is an increasing and continuous function on  $[0,1]$. 
\begin{Theo} \label{Theo:ExpoDTMBoundedUniform}
Under the previous assumptions, for any $k \leq \frac n2$,
\begin{eqnarray*}
 \E\left( \sup_{x \in \mathcal D}| \Delta_{n,\frac kn}(x) | \right)     
 & \leq&  \frac C{\sqrt n} \left[ \frac k n \right] ^  {- 1/2}     \left[ F_x^{-1}\left(\frac k n\right)
 - F_x^{-1}(0)\right]   \log^+  \left(\left[  
 \frac k{\left[ F_x^{-1}\left(\frac k n\right)
 - F_x^{-1}(0)\right]^{2}}  \right]^{\nu +5 }  \right)       \\
& & \hskip 1cm +  \:  \frac  C{\sqrt n} \left[ \frac k n \right] ^  {- 1/2}    \omega_{\mathcal D}    \left( 
\frac{\sqrt{k}}{n}  \right) \log^+ \left( \left[     \frac{ \sqrt k} { \omega_{\mathcal D} \left(  \frac{ \sqrt {k} } n
\right)}   \right] ^ { \nu-1}   \right)   \\
& \leq & \frac C{\sqrt n} \left[ \frac k n \right] ^  {- 1/2}    \omega_{\mathcal D}    \left(  \frac{\sqrt{k}}{n} 
\right)    \log^+  \left( \frac { k^{\nu +5 }}{  \left[ F_x^{-1}\left(\frac {\sqrt k}n\right)  - F_x^{-1}(0) \right]
^{2 \nu + 5 } \wedge \left[  \omega_{\mathcal D} \left(  \frac{ \sqrt {k} } n \right) \right]^{\nu -1} }  \right)      
\end{eqnarray*}
where $\log^+(u)  =( \log u) \vee 1$ for any $u \in \R^+$ . The constant $C$ is an absolute constant if $r=1$ otherwise
it depends on $r$ and on the  Hausdorff distance between $\cal D$ and the support of $P$.
\end{Theo}
This bound is deduced from a deviation bound on $\sup_{x \in  \mathcal D } | \Delta_{n,\frac kn} (x)|$ which is  given in the proof. Up to a logarithm term, the rate is the same as for the pointwise convergence. As for the pointwise convergence, this result could be easily extended to the case of non compactly supported measures.

\section{The geometric information carried by  the quantile function $\mathbf{F_x^{-1}}$}
\label{sec:AppliRd-discu}

The upper bounds we obtain in the previous section directly depend on  the regularity of $F_x^{-1}$. We now give some insights about how the geometry of the support of the  measure in $\R^d$ impacts the quantile function $ F_x^{-1}$.

\subsection{Compact support and modulus of continuity of the quantile function}

\begin{figure}
\begin{center}
	\includegraphics[scale=0.5]{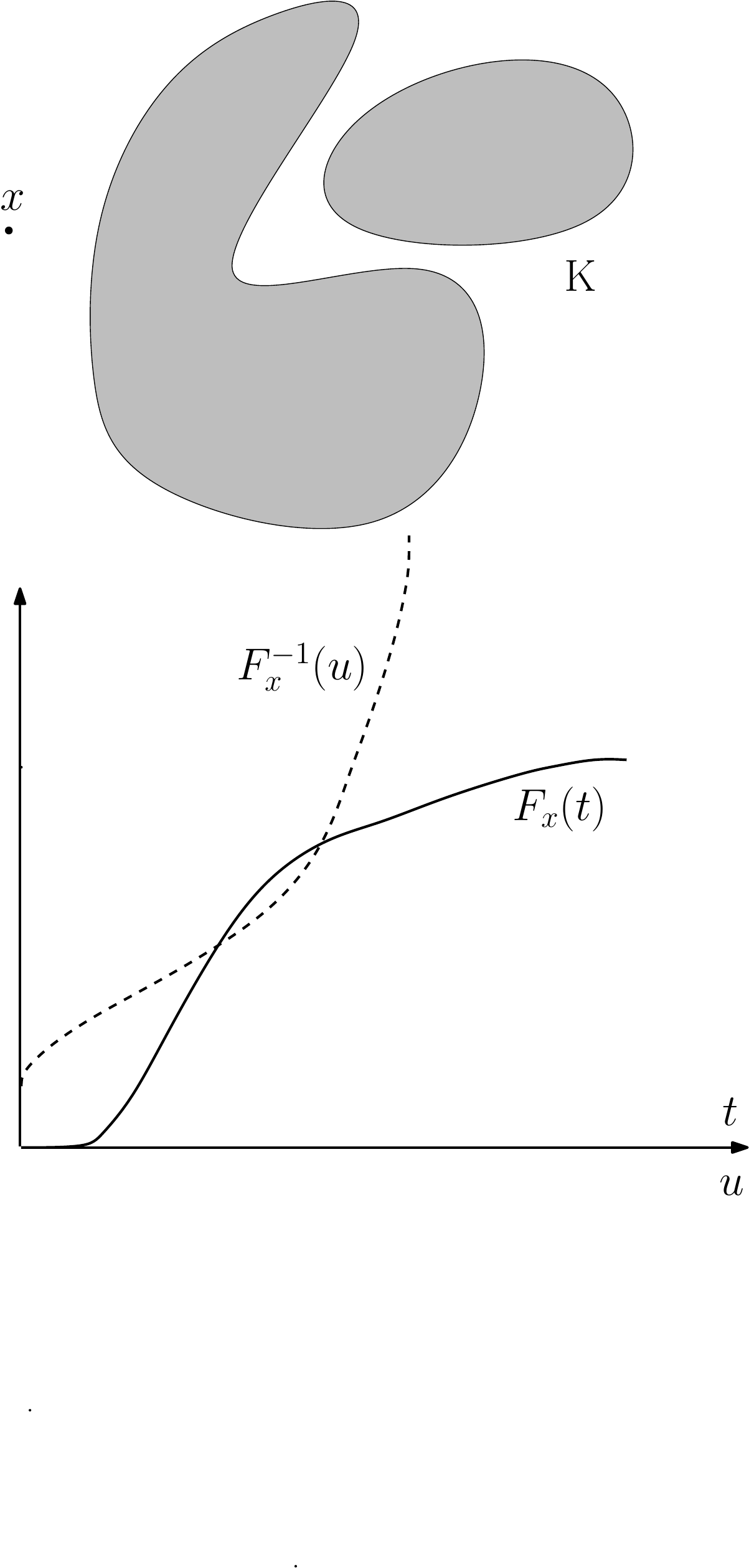}
	\hskip 1cm
	\includegraphics[scale=0.5]{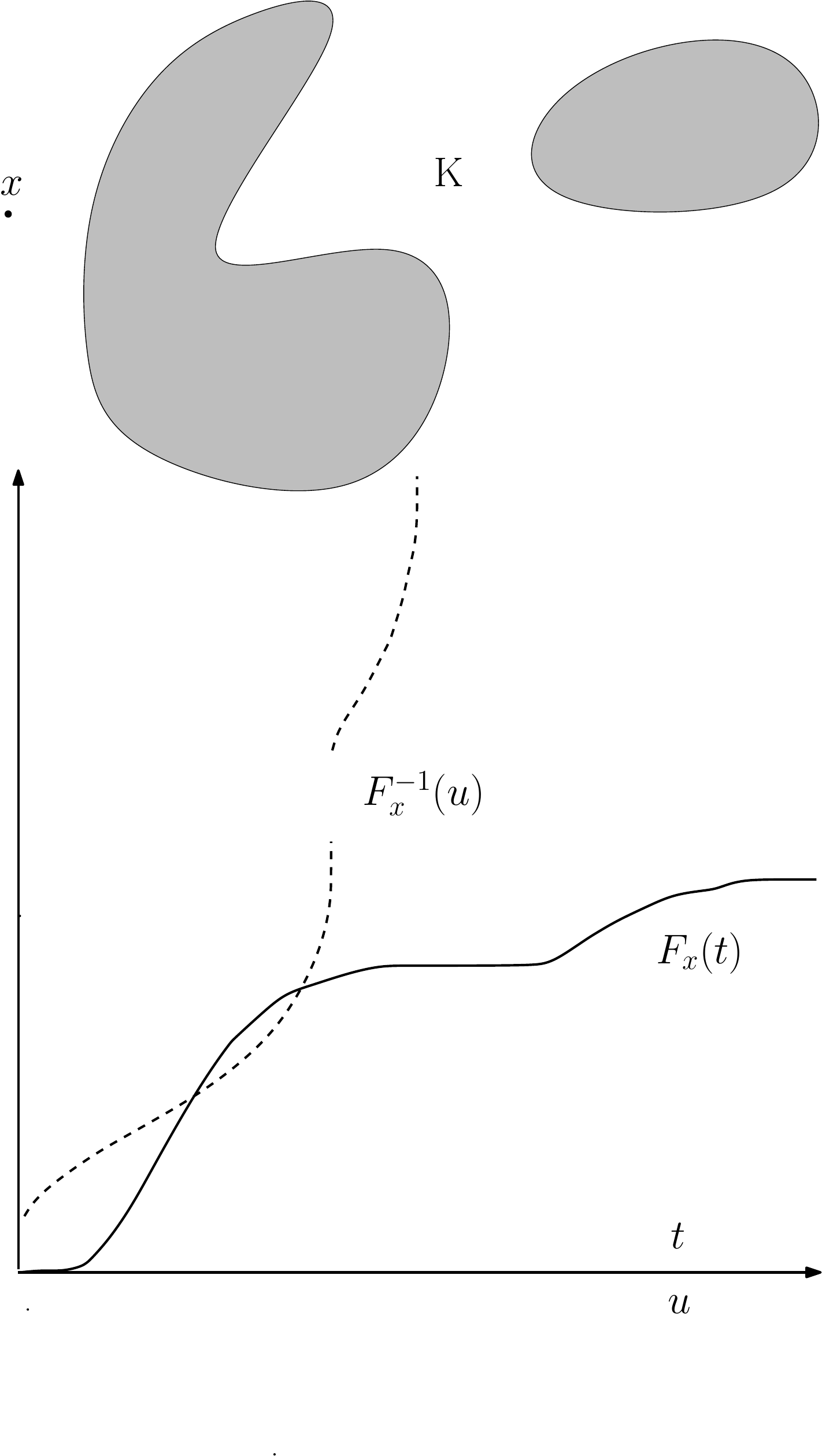}	
\caption{Left: one situation where the support of $P$ is not a connected set whereas the support of $dF_x$ is (for $r =1$). The
quantile function $F_x^{-1}$ is continuous. Right: one situation where the support of $dF_x$ is is not a connected set ;
the quantile function $F_x^{-1}$  is not continuous.}
\label{fig:Pnotconnex}
\end{center}
\end{figure}

A geometric characterization of the existence of $\tilde \omega_x$ on $[0,1]$ can be given in terms of the support of the
 measure $P$. The following Lemma is borrowed and adapted from Proposition A.12 in~\cite{bobkov2014one}:
 \begin{Lemma}
 Given a measure $P $  in $\R^d$ and  an observation point $x \in \R^d$, the following properties are
equivalent:
\begin{enumerate}
\item  the modulus of continuity  of the quantile function $F^{-1}_{x}$ satisfies  $\tilde \omega_x(u)  <  \infty$ for
any $u\leq 1$ ;
\item the push-forward distribution of $P$ by the  function $\| x- \cdot \| ^r$ is compactly supported ;
\item $P$  is compactly supported.
\end{enumerate}
\end{Lemma}
In particular, if $P$ is compactly supported,  we can always take as an upper bound on $\tilde \omega_x $  the constant function  $\omega_x = \dhaus \left( \{x\}, K\right)$. Of course this is not a very relevant choice  to describe the  rate of convergence of the DTEM.

\subsection{Connexity of the support and modulus of continuity of the quantile function}

While discontinuity of the distribution function corresponds to atoms, discontinuity points of the quantile function corresponds to area with empty mass in $\R^d$ (see the right picture of Figure \ref{fig:Pnotconnex}). The fact that $\tilde{\omega}_x(0^+) = 0  $  is directly related to the connectedness of the support of the distribution $d F_{x}$. Indeed,  it  is equivalent to assuming that the support of $d F_{x}$ is a closed interval in $\R^+$, see for instance Proposition~A.7 in~\cite{bobkov2014one}. 

 In the most favorable situations where  the support of $P$  is a connected set, then $\tilde{\omega}_x(0^+) = 0  $ and the  faster $\tilde \omega_x$ tends to 0 at 0, the better the rate we obtain. However, for some point $x \in \R^d$, it is also possible for the support of $d F_{x}$ to be an interval even when the support of $P$ is not a connected set of $\R^d$ (see the left picture of  Figure~\ref{fig:Pnotconnex}).  In the other case, when the support of  $d F_{x}$ is not a connected set, the term $\tilde
\omega_x (0) $ roughly corresponds to the maximum distance between two consecutive intervals of the support of  $d
F_{x}$ (see the right picture of Figure \ref{fig:Pnotconnex}).  Our results  can   still be applied in these situations
but the upper bounds we obtain in this case are larger because  $\omega_x (\frac kn)$ can not be smaller  than $\tilde \omega_x (0) $.

\subsection{Uniform modulus of continuity of $F_{x,r}^{-1}$ versus local continuity of $F_{x,r}^{-1}$ at the origin}

Though  stronger than continuity, a natural regularity assumption on $F_{x,r}^{-1}$ is assuming that this function is also concave:
\begin{Lemma}
If $F_x^{-1}$ is concave  then  we can take $\omega_x = F_x^{-1} - F_x^{-1}(0)$. In particular, if $x$ is in the support
of $P$ then we can take $\omega_x = F_x^{-1}$.
\end{Lemma}

If we take $r = 1$, in many simple situations we note that the cumulative distribution function $F_{x,1}$ roughly
behaves as a power function $t^\ell$, where $\ell$ is the dimension of the support. In this context, the quantile
function $F_{x,1}^{-1}$ roughly behaves as a power function in $u^{1/\ell}$.  We then have that $F_{x,r} ^{-1}(u) = 
\left[  F_{x,1} ^{-1}(u) \right]^r $ behaves as $u^{\frac r \ell}$. This is for instance the case for $(a,b)$ standard
measures, as shown in the next section. These considerations suggest that if $r/ \ell  < 1$, in many situations the
quantile function is concave and then $\omega_x$ is of the order of $ F_x^{-1} - F_x^{-1}(0)$.   This means that the
upper bound on $\E | \Delta_{P,n,\frac kn}| $ is of the order of  $\frac 1{\sqrt n}  \widetilde \Psi_x(\frac kn)$. 

More generally, as noticed in the comments following Theorem~\ref{Theo:ExpoDTMBounded}, the term  $F_x^{-1}(\frac kn) -
F_x^{-1}(0)$ is the dominating term in the upper bound \eqref{eq:InitilalBound}. We may check with the numerical
experiments of Section~\ref{subs:expe} that the function $ \widetilde \Psi_x$  yet captures the correct monotonicity of $\E | \Delta_{P,n,\frac kn}| $ as a function of $\frac kn$.

\subsection{The case of $\mathbf{(a,b)}$ standard measures} \label{sub:abstandard}

The intrinsic dimensionality of a given measure in $\R^d$ can be quantified by the so-called 
$(a,b)$- {\it standard} assumption which assumes that there exists $a'>0$, $\rho_0>0$ and $b >0$  such that 
\begin{equation*} \label{ref:SdtAssump1}
\forall x \in K, \ \forall r \in (0,\rho_0), \  P(B(x,\rho)) \geq a' \rho^b,
\end{equation*}
where $K$ is the support of $P$.  This assumption is popular in the literature about set estimation  \citep[see for instance][]{Cuevas09,cuevas2004boundary}. More recently, it has also been in used in \cite{chazalconvergenceJMLR,fasy2014confidence,chazal2014subsampling} for statistical analysis inTopological Data Analysis. 

Since  $K$ is compact, by reducing the constant $a'$ to a smaller constant $a$ if necessary, we easily check that this assumption (\ref{ref:SdtAssump1}) is equivalent to 
\begin{equation*} \label{ref:SdtAssump2}
\forall x \in K, \  P(B(x,\rho)) \geq 1\wedge a \rho^b .
\end{equation*}

We now give control on the two key terms $\omega_x$ and $F_x ^{-1}(u) -  F_x ^{-1}(0)  $  which are involved in the bounds on expectations of Section~\ref{sec:mainresults}.
\begin{Lemma}
Let $P$ be a probability measure on $\R^d$ which is $(a,b)$ standard on its support $K$. Then, for any $u
\in [0,1]$,
$$  F_x ^{-1}(u) -  F_x ^{-1}(0)  \leq  r  \left(  \frac{u}{a} \right) ^{1/b}  \left[  \left(  \frac{u}{a}
\right) ^{1/b} + \| K - x \|   \right] ^{r-1} ,$$
where $r$ is the power parameter in the definition~\eqref{eq:DTMdef} of the DTM.
Assume moreover that $K$ is a connected set of $\R^d$. Then, for any $h \in (0,1)$ we have
$$ 
\tilde \omega_x(h)   \leq r   \left(  \frac{h}{a} \right) ^{1/b}  \dhaus \left( \{x\},K \right) ^{r-1} .$$
\end{Lemma}
\begin{figure}
\begin{center}
	\includegraphics[scale=0.65]{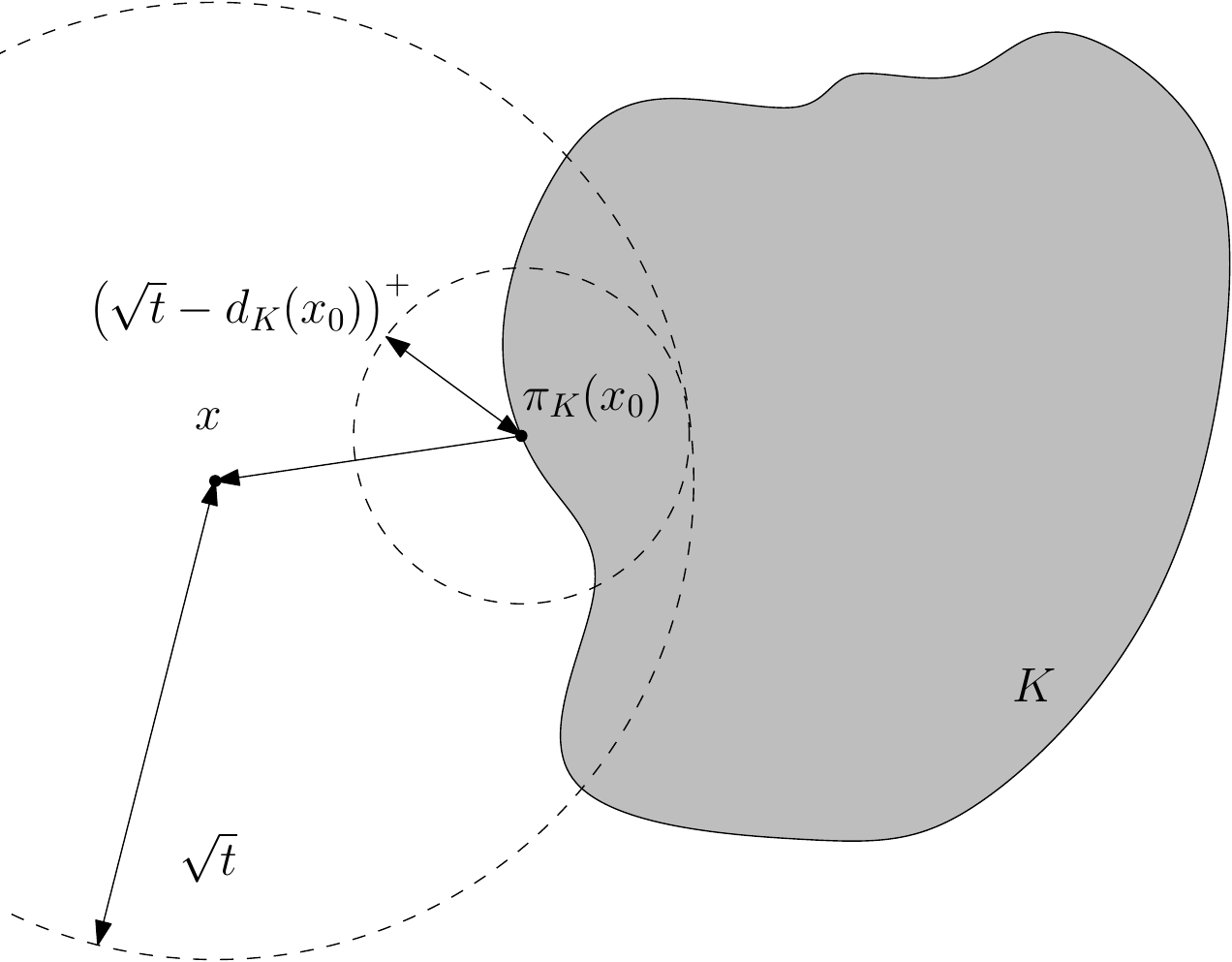}
	\includegraphics[scale=0.65]{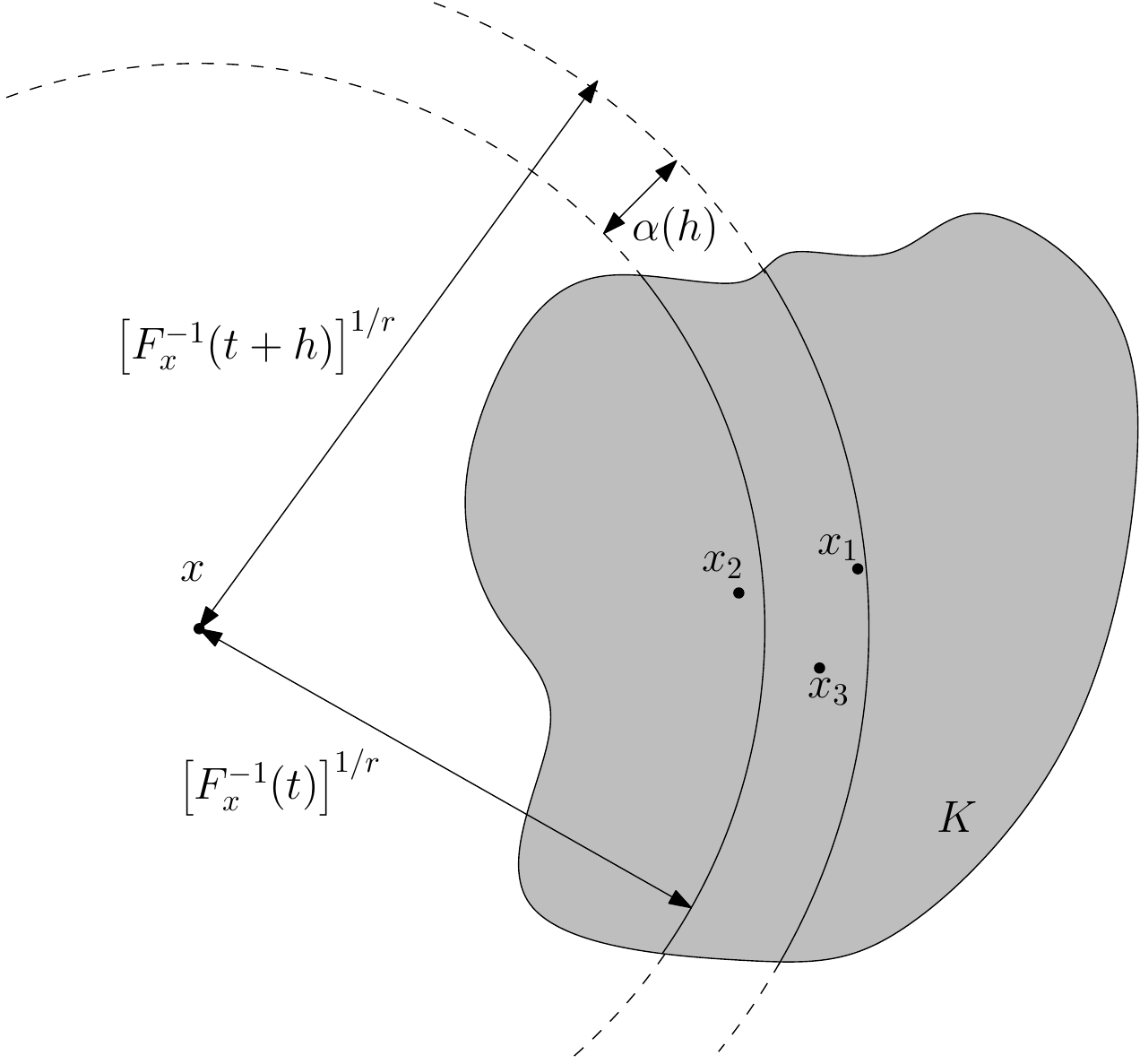}
\caption{About the modulus of continuity of the quantile function $F_x^{-1}$ in the case of $(a,b)$- standard measures in $\R^d$.}
\label{fig:Deltah}
\end{center}
\end{figure}
\begin{proof}
We have (see the left picture of Figure \ref{fig:Deltah})
\begin{align*}
F_x (t)  &= P\left( B\left(x,  t^{1/r} \right) \right) \\
& \geq  P\left(   B \left(\pi_K(x) , \left(t^{1/r} - \| K - x \| \right) ^+  \right) \right) 
 \end{align*}
where $\pi _K(x)$ is a point of $\R^d$ which satisfies  $\| K - x \| = \| \pi _K(x) - x\|$. Then 
$F_x (t)  \geq  a \left[  \left(t^{1/r} - \| K - x \| \right) ^+   \right]^b $ and  we find that
 $ F_x ^{-1}(u)    \leq   \left[  \left(  \frac{u}{a} \right) ^{1/b} + \| K - x \|   \right] ^{r} $. Next, we have 
$ F_x ^{-1}(0) =   \| K - x \|^r$ and the first point derives by upper bounding the derivatives of $v \mapsto   \left[   v
+ \| K - x \|  \right] ^{r} $.

We now assume that $K$ is a connected set. Let  $(u,h) \in (0,1)^2$ such that $u + h \leq 1$ and $F_x ^{-1}(u)  > F_x ^{-1}(0) $.   We can also assume that $F_x^{-1}(u+h) > F_x^{-1}(u)$. Let 
 $\alpha(h) = [F_x^{-1}(u+h)]^{1/r} -  [ F_x^{-1}(u)]^{1/r}$ (see the right picture of Figure \ref{fig:Deltah}). By
definition of a quantile, there exists a
point $x_1 \in  K  \cap \left\{B \left( x, [F_x^{-1}(u+h)]^{1/r} \right) \setminus  B \left( x, [F_x^{-1}(u)]^{1/r}
\right)
\right\}   $. If $F_x^{-1}(u) >0$  then for the same reason there exists a point $x_2 \in K \cap  B \left( x, [F_x^{-1}(u)]^{1/r} \right)$. If $F_x^{-1}(u) =0$ then  $x  \in K$ and we take $x_2 = x$. Next, since  $K$ is a connected set, there exists a point $x_3 \in K \cap B \left( x, [F_x^{-1}(t)]^{1/r} +
\frac{\alpha}2 \right)$. The measure $P$ being $(a,b)$-standard, we find that
\begin{align*}
  h &\geq  P \left(  B \left(  x_3, \frac{\alpha(h)}2   \right) \right) \\
  &\geq  a \left(  \frac{\alpha(h)} {2}\right)^{b} .
 \end{align*}
Then,
$$  [F_x^{-1}(t+h)]^{1/r} - [F_x^{-1}(t)]^{1/r} \leq \frac 2 {a ^{1/b}} h ^{1/b},$$
and thus
$$ 
F_x^{-1}(t+h) -  F_x^{-1}(t)  \leq  r  a ^{-1/b}   \left( F_x^{-1}(t+h) \right) ^{r-1} h ^{1/b},
$$
which proves the Lemma.
\end{proof}

\section{Numerical experiments} \label{subs:expe}

\begin{figure}
   \begin{minipage}[b]{0.5\linewidth}
	\includegraphics[scale=0.5]{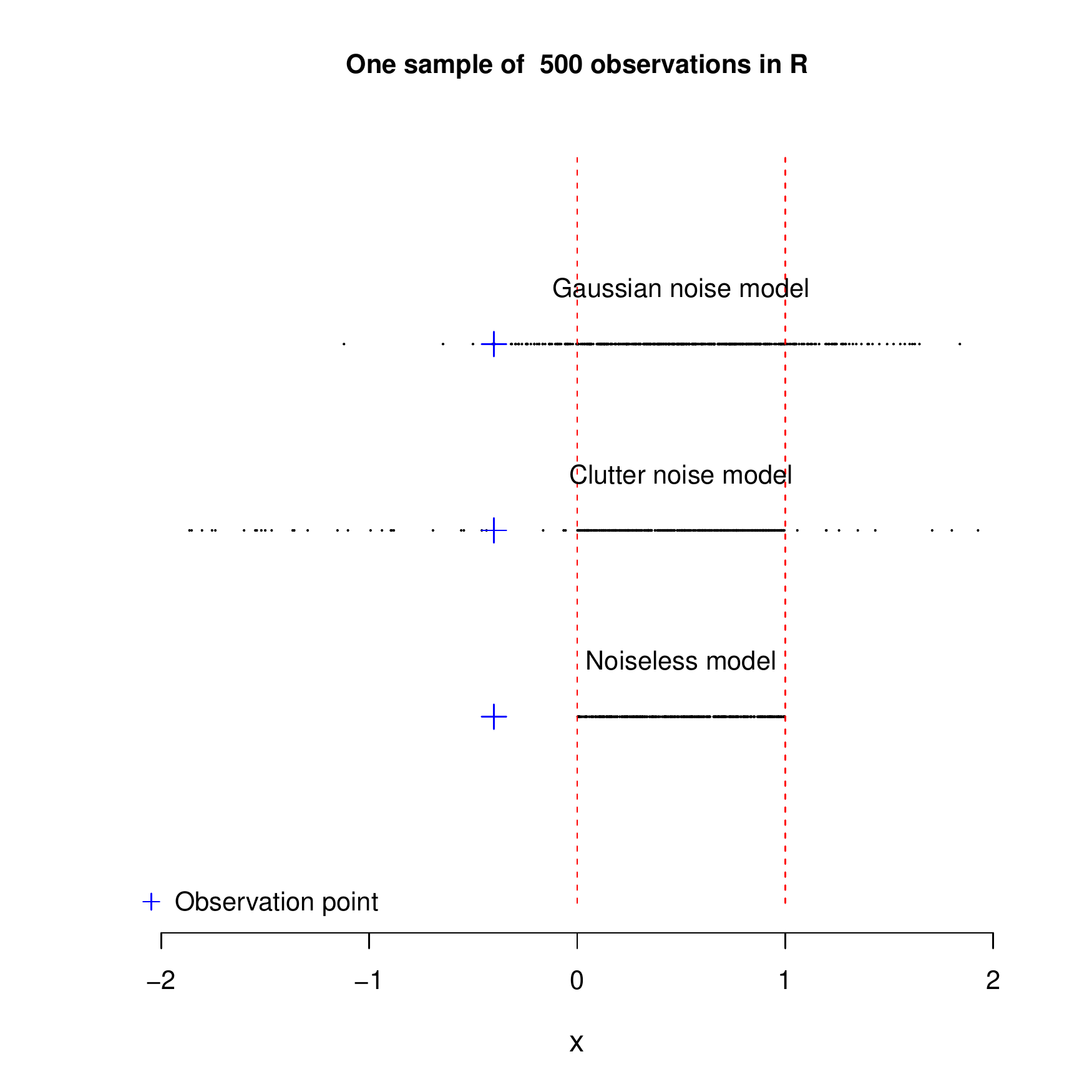}
   \end{minipage}\hfill
   \begin{minipage}[b]{0.5\linewidth}   
	\includegraphics[scale=0.5]{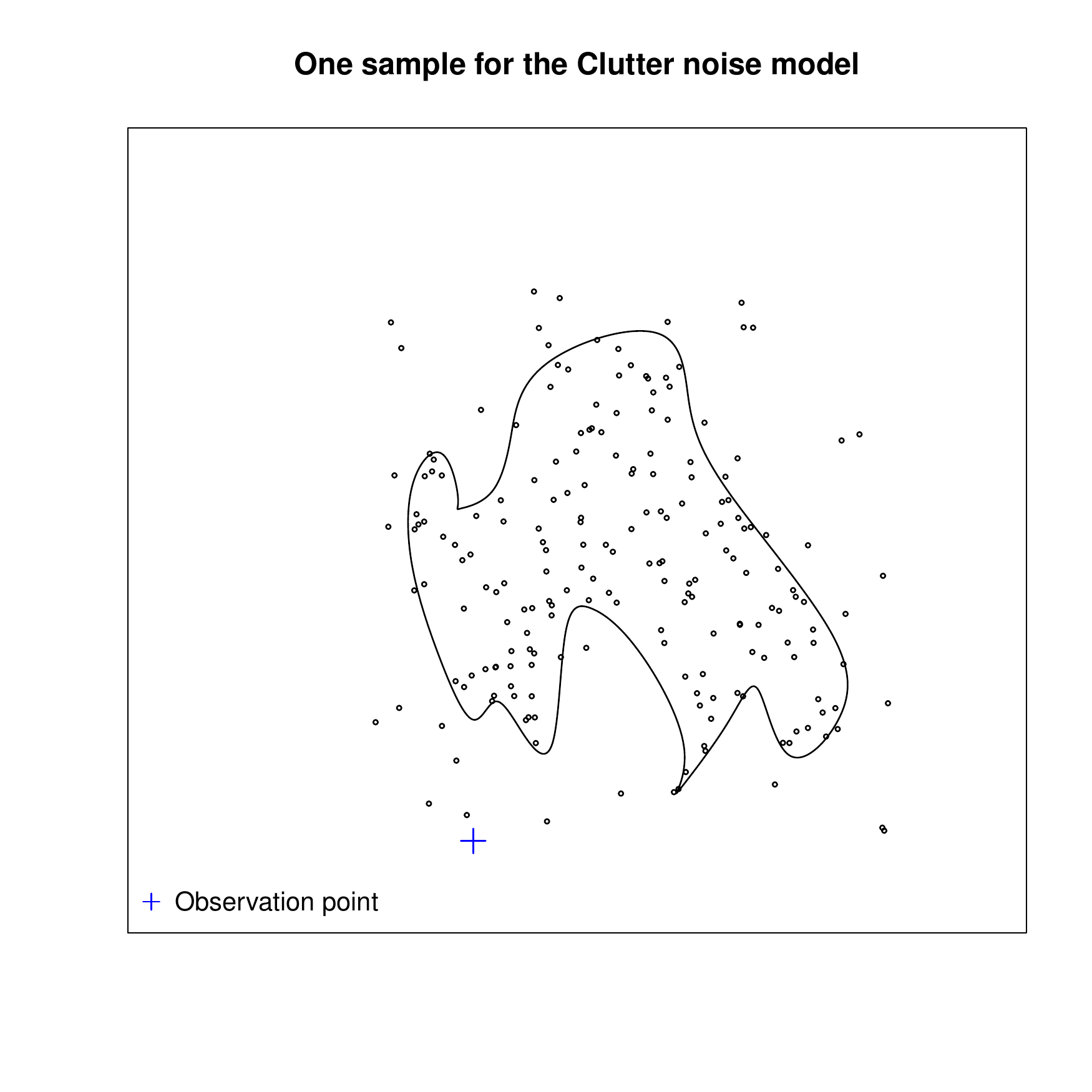}
   \end{minipage}
\caption{Left: samples drawn for each generative model for the Segment Experiment. Right: one sample drawn from the clutter noise
model for  the 2-d Shape Experiment. The observation point is represented by a blue cross.}
\label{fig:SimusSegmentAnd2dShape}
\end{figure}
\begin{figure}
   \begin{minipage}[b]{0.5\linewidth}  
   	\includegraphics[scale=0.3]{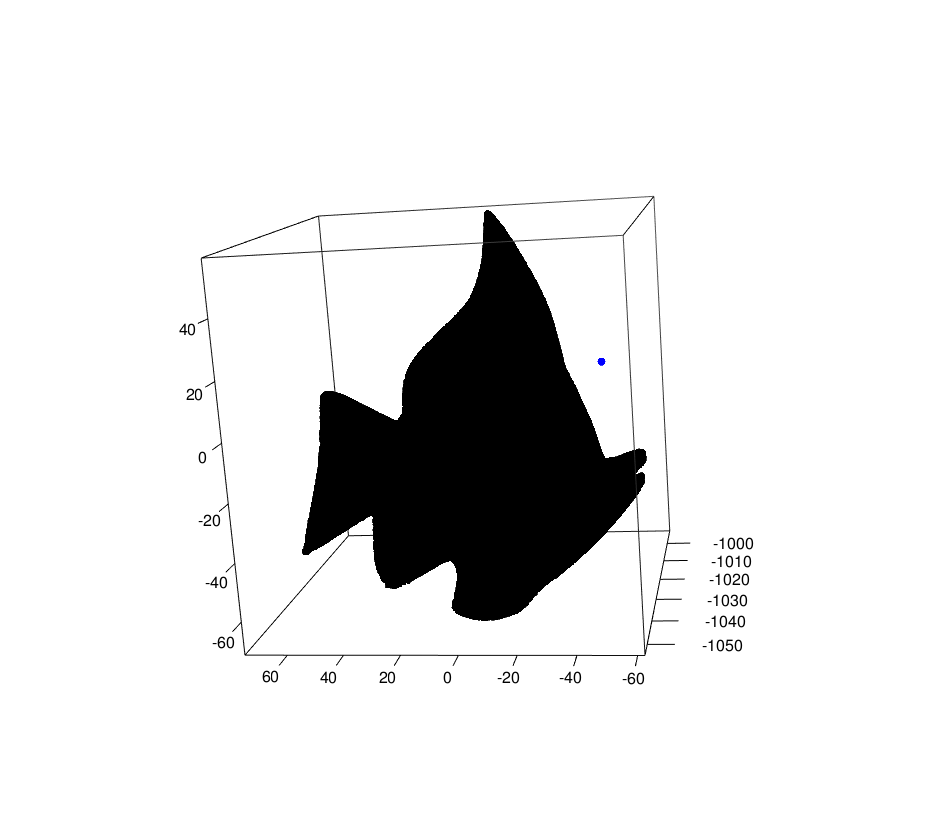}
   \end{minipage}\hfill 
   \begin{minipage}[b]{0.5\linewidth}
	\includegraphics[scale=0.3]{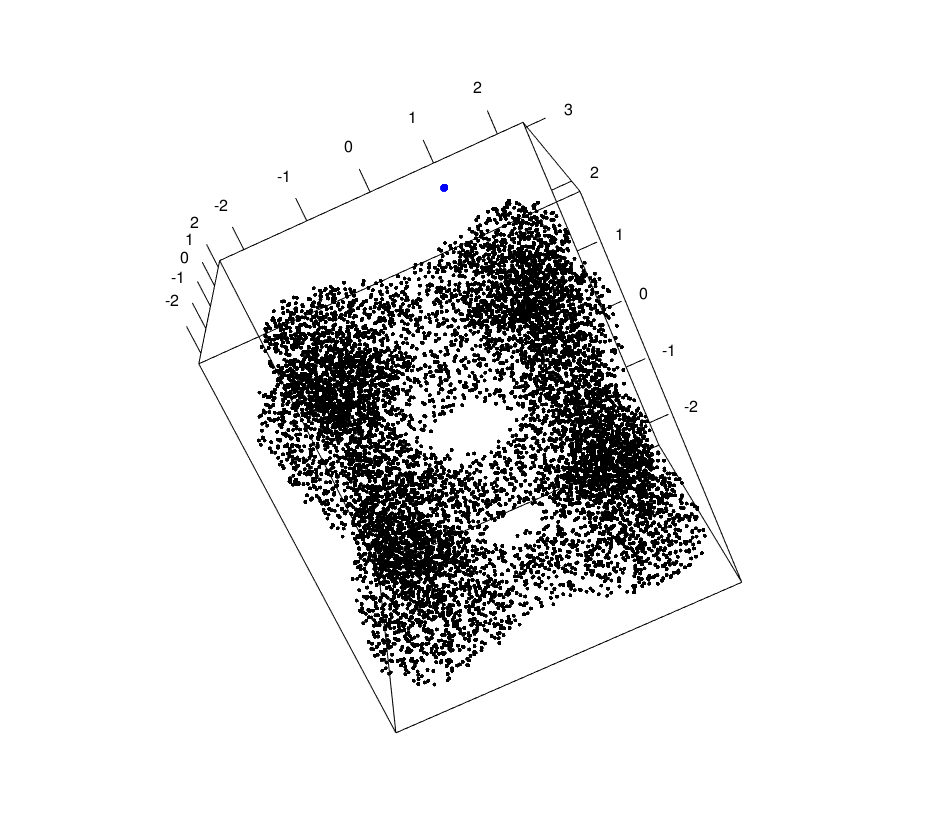}
   \end{minipage}
\caption{Left: 3-d plot of the shape for the Fish Experiment. Right: a 3-d plot of a sample drawn for the uniform measure on the Tangle Cube. The observation point is represented by the blue point outside of the shape.}
\label{fig:SimusFishTangle}
\end{figure}

In this section, we illustrate with numerical experiments that the expectation bounds given on  $\Delta_{n,\frac kn}$ in Section~\ref{sec:mainresults} are sharp. In particular, we check that the function $\widetilde\Psi_x$  has the same monotonicity as the function  $ m \mapsto  \E | \Delta_{n,\frac kn}(x) | $. 

We consider four different geometric shapes  in $\R$, $\R^2$ and $\R^3$, for which a visualization is possible: see
Figures~\ref{fig:SimusSegmentAnd2dShape} and \ref{fig:SimusFishTangle}.
\begin{itemize}
\item {\bf Segment Experiment in $\R$.} The shape $K$ is the segment $[0,1]$  in $\R$.
\item{\bf 2-d shape Experiment in $\R^2$.}  A  closed curve has been drawn at hand in  $\R^2$. It has been next approximated 
by a polygonal curve with a high precision. The shape $K$ is the  compact set  delimited by the polygon curve.
\item {\bf Fish Experiment: a 2-d surface in $\R^3$.} The shape  $K$ is the discrete set defined by a point cloud of 216979 points approximating a 2-d surface representing a fish. This dataset is provided courtesy of CNR-IMATI  by the AIM@SHAPE-VISIONAIR Shape Repository. 
\item{\bf Tangle Cube Experiment in $\R^3$.}   The shape $K$ is the tangle cube, that is the 3-d manifold defined as the
set of points $(x_1,x_2,x_3)  \in \R^3$ such that  $  x_1^4 -  5 x_2^2 + x_2^4 - 5  x_2^2  +x_3^4 - 5 x_3^2 + 10 \leq 0$. 
\end{itemize}

For each shape, we consider three generative models. These models are standard in support estimation and geometric
inference, see \cite{genovese2012manifold} for instance. 
\begin{itemize}
 \item {\bf Noiseless model:} $X_1, \dots X_n$ are sampled   from the uniform probability distribution $P_{uni}$  on
$K$.
 \item {\bf Clutter noise model:} $X_1, \dots X_n$ are sampled  from the mixture distribution $P_{cl} =  \pi U + (1-\pi)
P$ where $U$ is the uniform measure on a box $B$ which contains $K$ and where $\pi $ is a proportion parameter.
 \item {\bf Gaussian convolution model:} $X_1, \dots X_n$ are sampled   from the distribution $P_g = P  \star
\Phi(0,\sigma I_d) $ where  $\Phi(0,\sigma) $ is the centered isotropic multivariate Gaussian distribution on $\R^d$
with covariance matrix $\sigma I_d$. We take $\sigma = 0.5$ in all the experiments.
\end{itemize}
We use the same notation $P_{\mbox{\scriptsize $\square$}}$ for any  of the probability distributions  $P_{uni}$,  $P_{cl}$ or $P_{g}$.  An observation point $x$ is fixed for each experiment. For each experiment and each generative model, from a very large sample drawn from $P_{\square}$ we  compute very accurate estimations of the quantile functions $F_{x,r}^{-1}$ and of the DTM $d_{P_{\square},m,t}(x)$. Next, we simulate $n$-samples  from $P_{\square}$  and we compute  the DTEM for each sample. We take  $n=500$ for the two first experiments  and $n=2000$ for the two others.  The trials are all repeated 100 times and finally  we compute some approximations of  the error $\E \Delta_{n,\frac kn,r}(x)$ with a standard Monte-Carlo procedure, for all the measures $P_{\square}$. The  DTMs and the DTEMs are computed for the powers $r=1$,  $r=2$ , and also for
$r=3$  for the Tangle Cube Experiment. We also compute the function $m \mapsto
\tilde \Psi (m)$. The simulations have been performed using  R software \citep{Rcite} and we have 
used the packages \textsc{FNN}, \textsc{rgl}, \textsc{grImport} and \textsc{sp}. 

\subsection*{Results}
The figures~\ref{fig:ErrorSegment} to~\ref{fig:ErrorTangle} give the results of the four experiments with the three generative models. The top graphics of Figures~\ref{fig:ErrorSegment} to~\ref{fig:ErrorTangle} represent the quantiles functions $F_{x,r}^{-1}$ in each case. For the noiseless models, the behavior of  $F_{x,r}^{-1}$  at the origin is directly related to
the power $r$ and to the intrinsic dimension of the shape. For $r  =1$ , the quantile is linear for the the segment, it is roughly in $\sqrt m$ for the 2-d shape and for the Fish Experiment. It is of order of $m ^{1/3}$ for the Tangle Cube. We observe that $F_{x,r}^{-1}$ is roughly linear with $r=2$ for  the 2-d shape and the Fish shape,  and with $r=3$ for the Tangle Cube. 

The quantile functions of the noise models in the four cases start from
zero since the observation is always taken inside the supports of $P_{cl}$ and $P_{g}$. A regularity break for the quantile function of the clutter noise model  can be observed in the neighborhood of $ m = P( B(x , \|K-x\|^r))  $. The quantile functions for the Gaussian noise is always smoother. 

The main point of these experiments is that, in all cases,  the function $ m \mapsto \tilde \Psi(m) $ shows the same
monotonicity as the expected error studied in the paper $ : m \mapsto | \E \Delta_{n, m,r}(x)| $. These results confirm that the
function $ \tilde \Psi$ provides a correct description of $\E \Delta_{n,m,r}$. 

We also observe that the function $: m \mapsto \E | \Delta_{n, m,r}(x)| $ does not have one typical shape : it can be an increasing curve, a decreasing curve or even an U-shape curve. Indeed, the monotonicity  depend on many factors including the intrinsic dimension of the shape, its geometry, the presence of noise and the power coefficient $r$. 

\begin{figure}[pH]
   \begin{minipage}[b]{0.5\linewidth}
      \centering \includegraphics[scale=0.42]{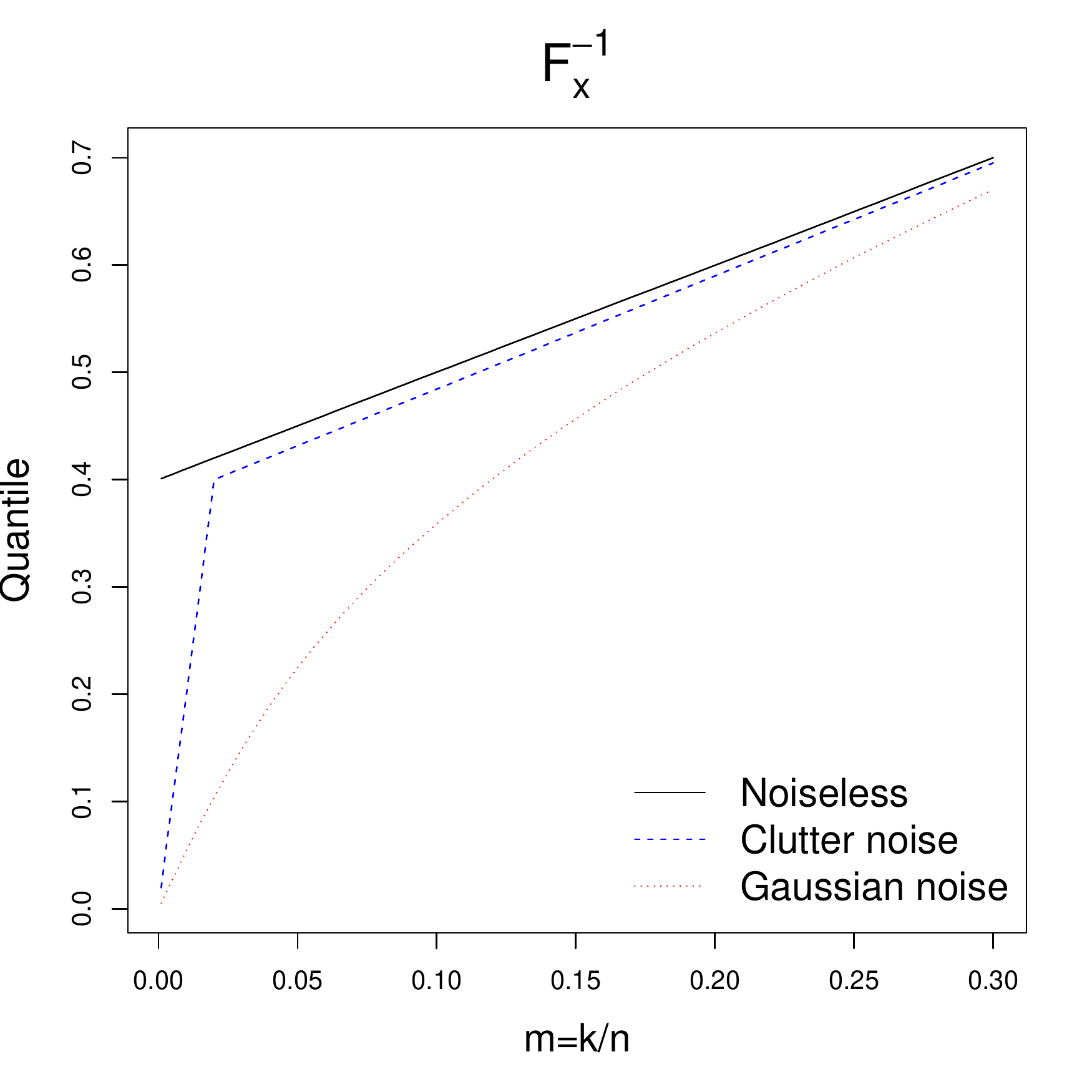} 
   \end{minipage}\hfill
   \begin{minipage}[b]{0.5\linewidth}   
      \centering \includegraphics[scale=0.42]{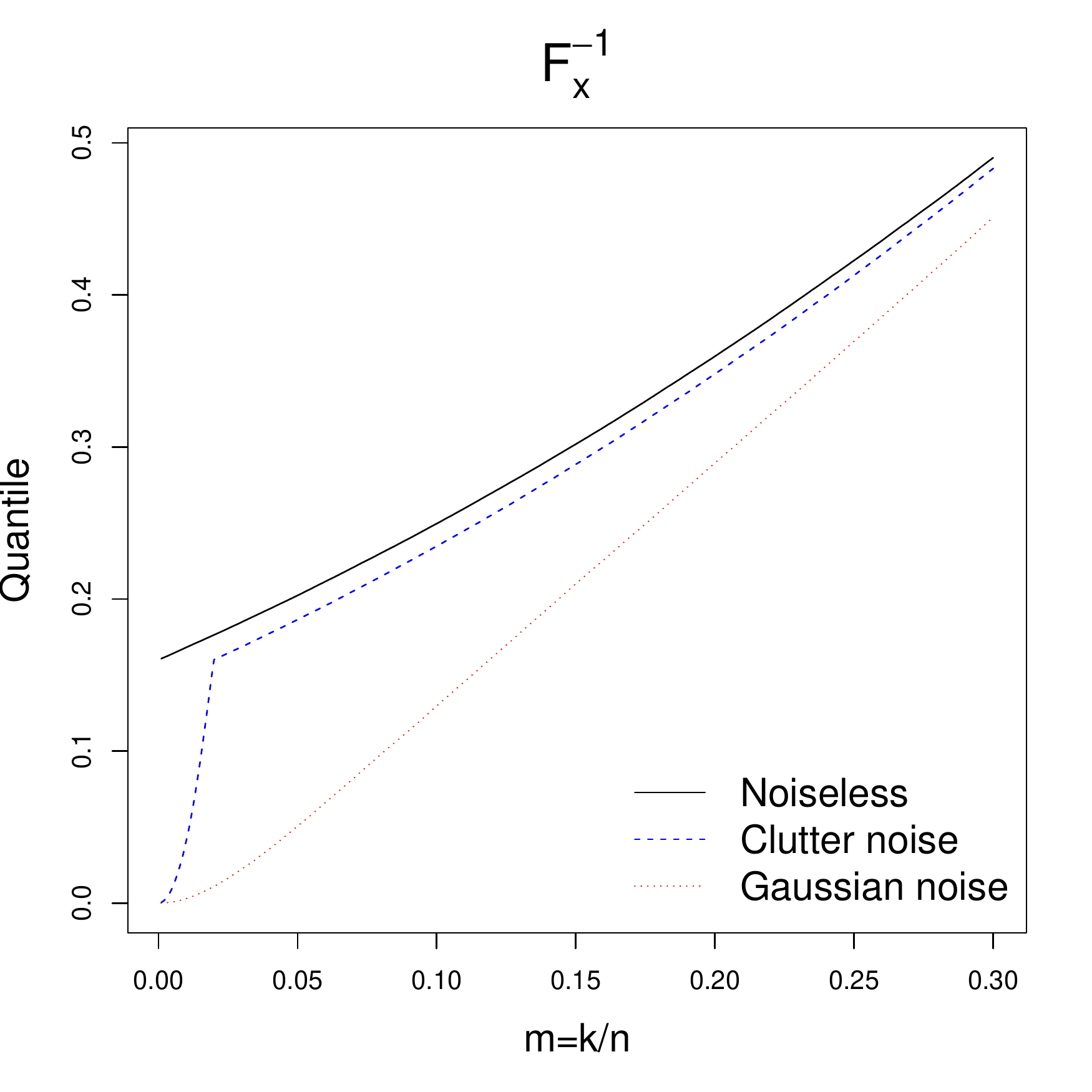} 
   \end{minipage}
   \begin{minipage}[b]{0.5\linewidth}
      \centering \includegraphics[scale=0.42]{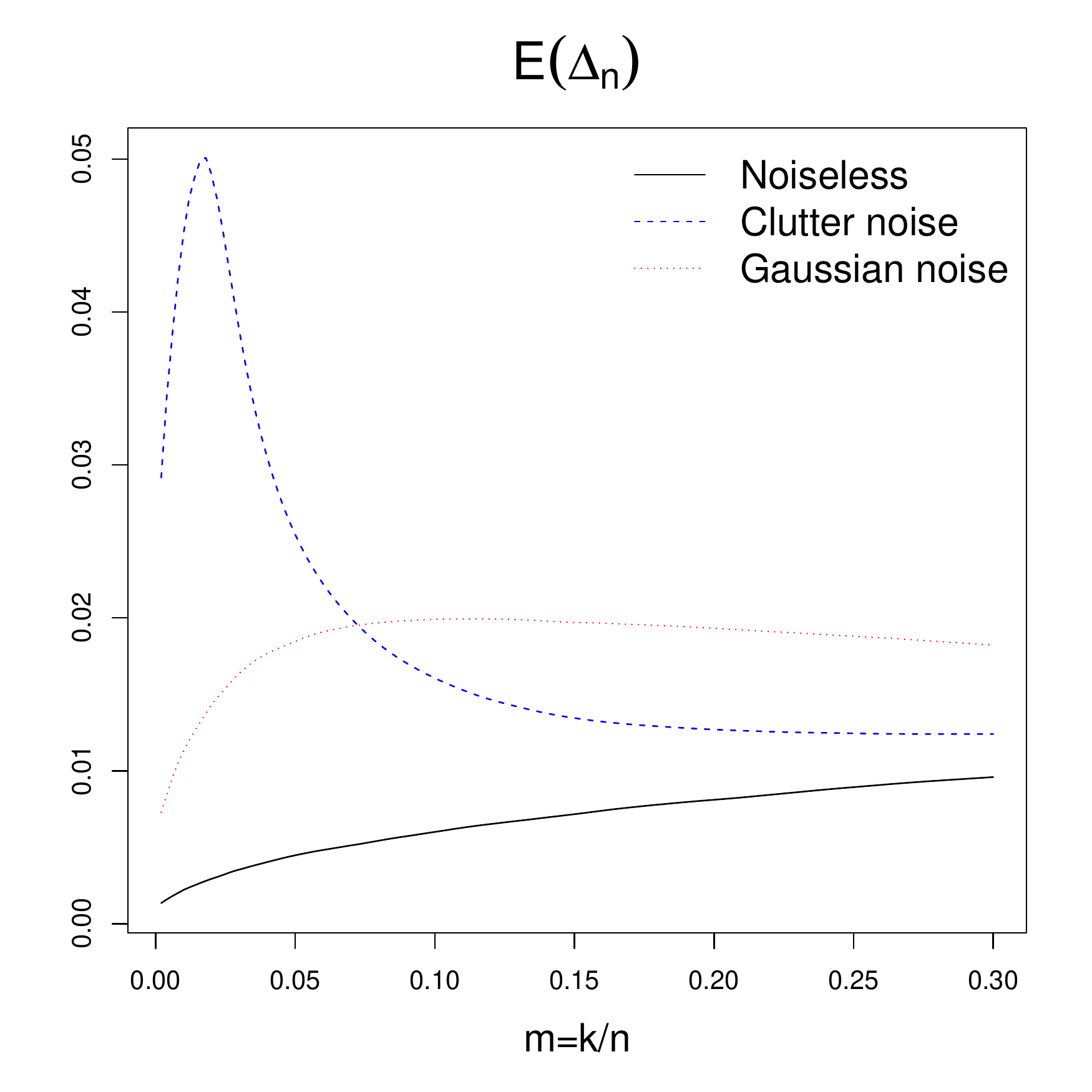} 
   \end{minipage}\hfill
   \begin{minipage}[b]{0.5\linewidth}   
      \centering \includegraphics[scale=0.42]{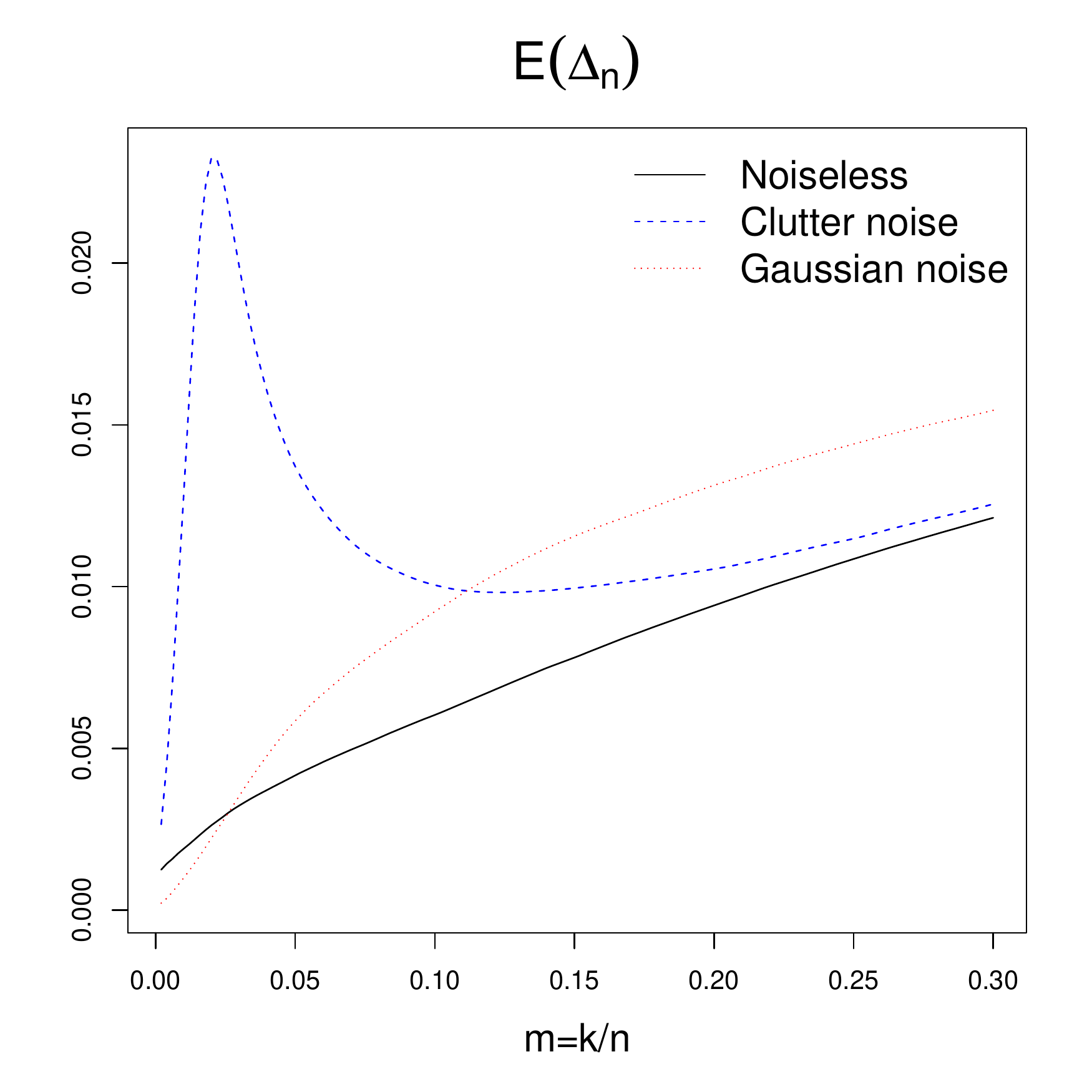} 
   \end{minipage}
   \begin{minipage}[b]{0.5\linewidth}
      \centering \includegraphics[scale=0.42]{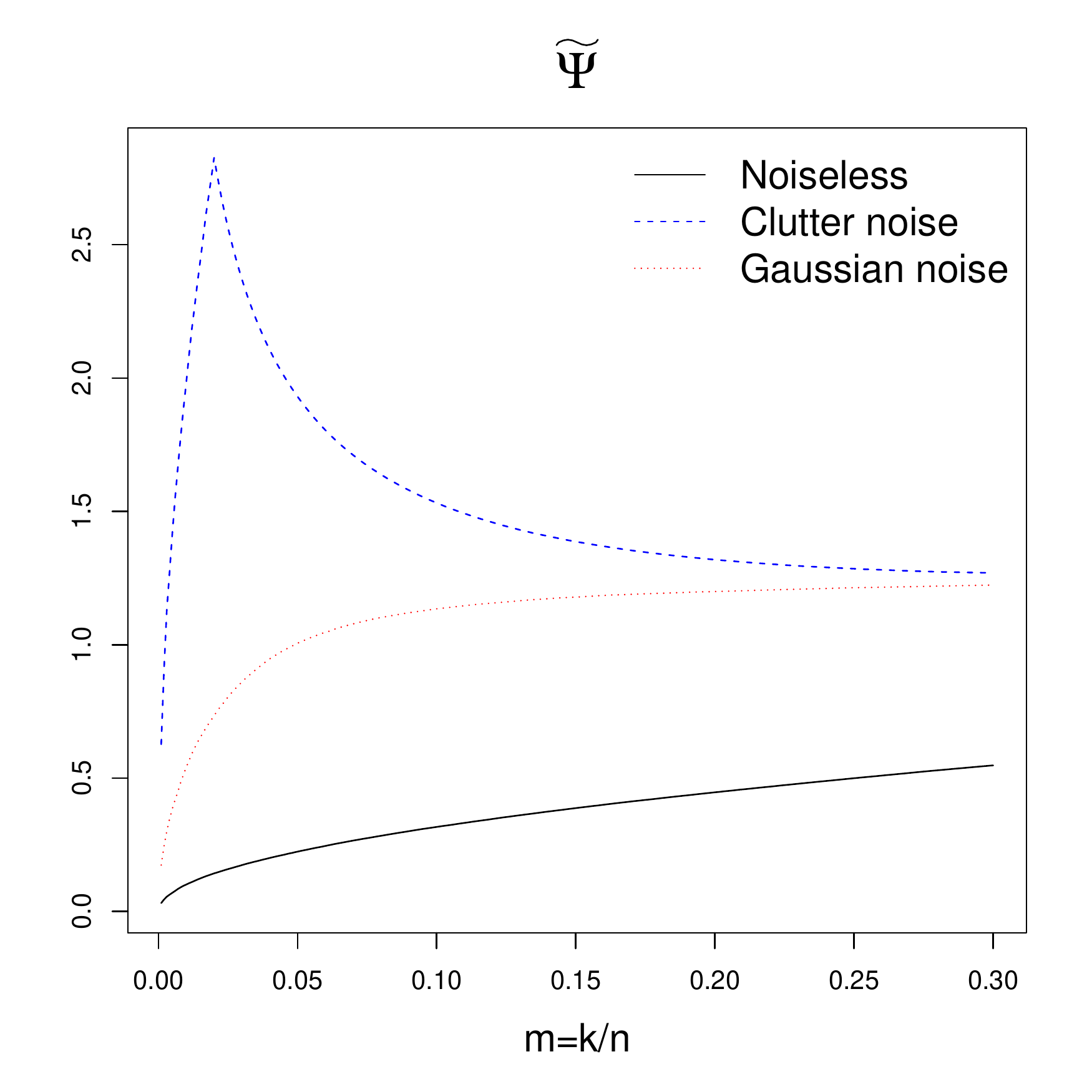} 	
   \end{minipage}\hfill
   \begin{minipage}[b]{0.5\linewidth}   
      \centering \includegraphics[scale=0.42]{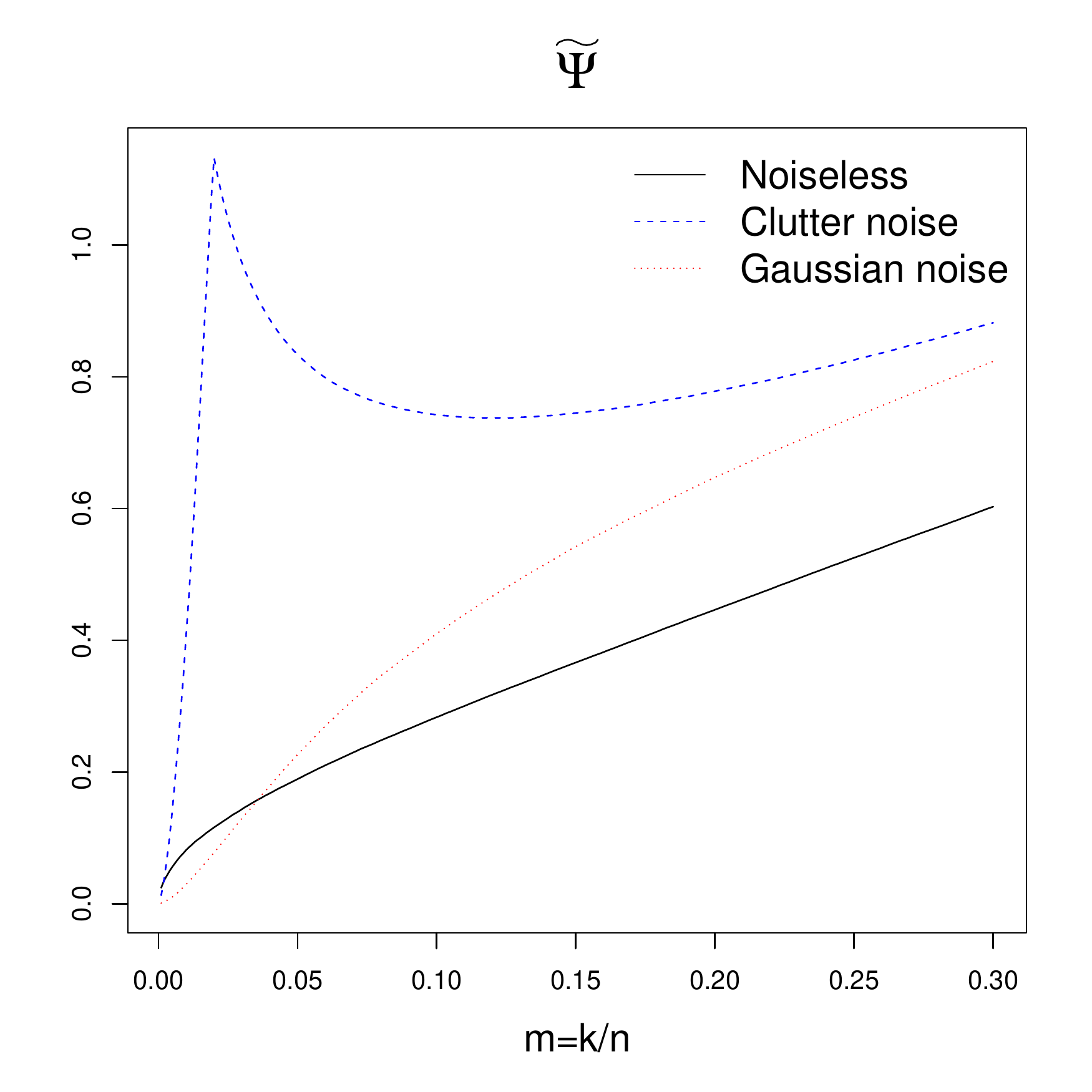} 
   \end{minipage}
   \caption{Quantiles functions $F_{x,r}^{-1}$ (top), expected error $\E \Delta_{n,\frac kn,r}(x) $ (middle) and
theoretical upper bounds $\tilde \Psi$ (bottom) with powers $r=1$ (left) and $r=2$ (right), for the Segment Experiment.}
   \label{fig:ErrorSegment}
\end{figure}

\begin{figure}[pH]
   \begin{minipage}[b]{0.5\linewidth}
      \centering \includegraphics[scale=0.42]{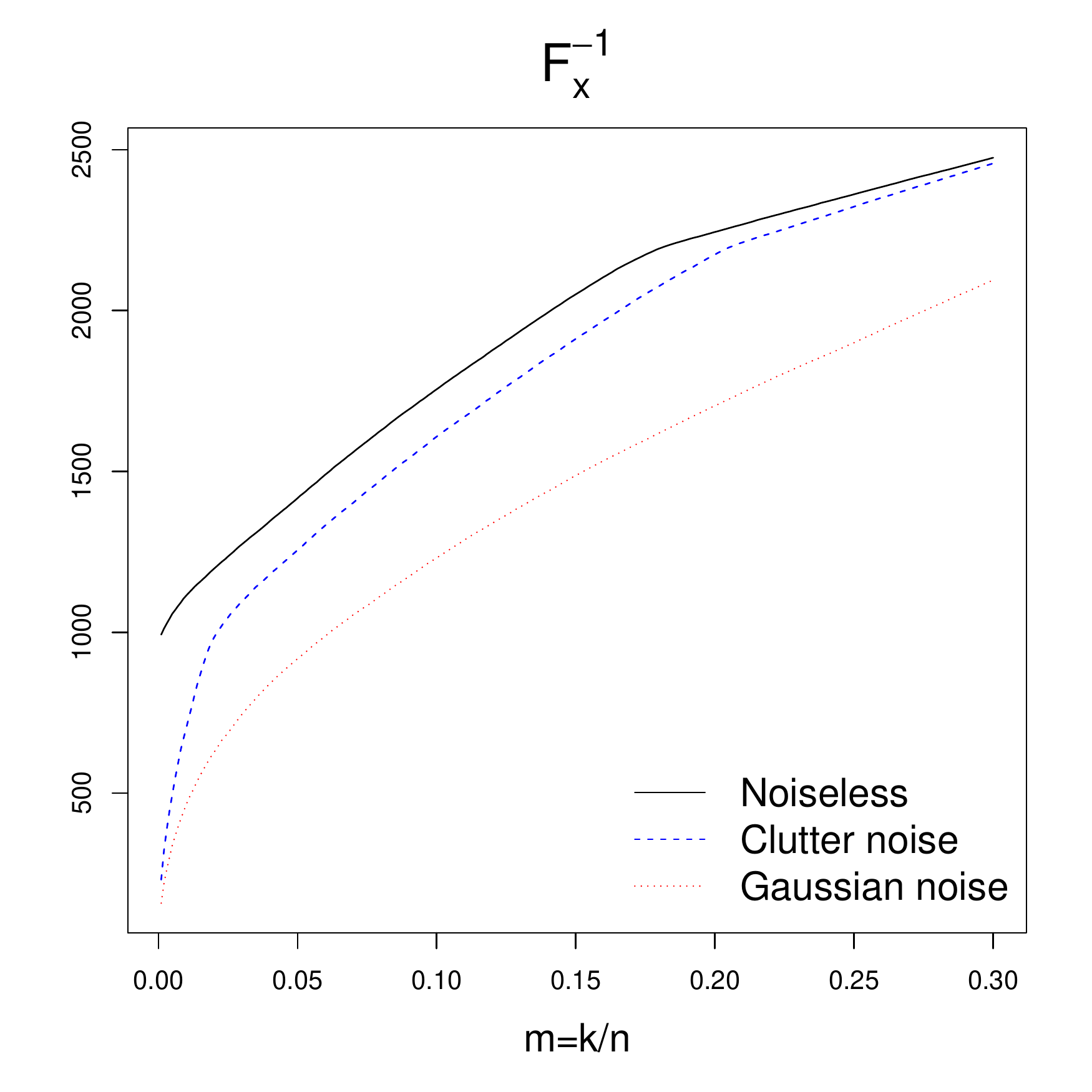} 
   \end{minipage}\hfill
   \begin{minipage}[b]{0.5\linewidth}   
      \centering \includegraphics[scale=0.42]{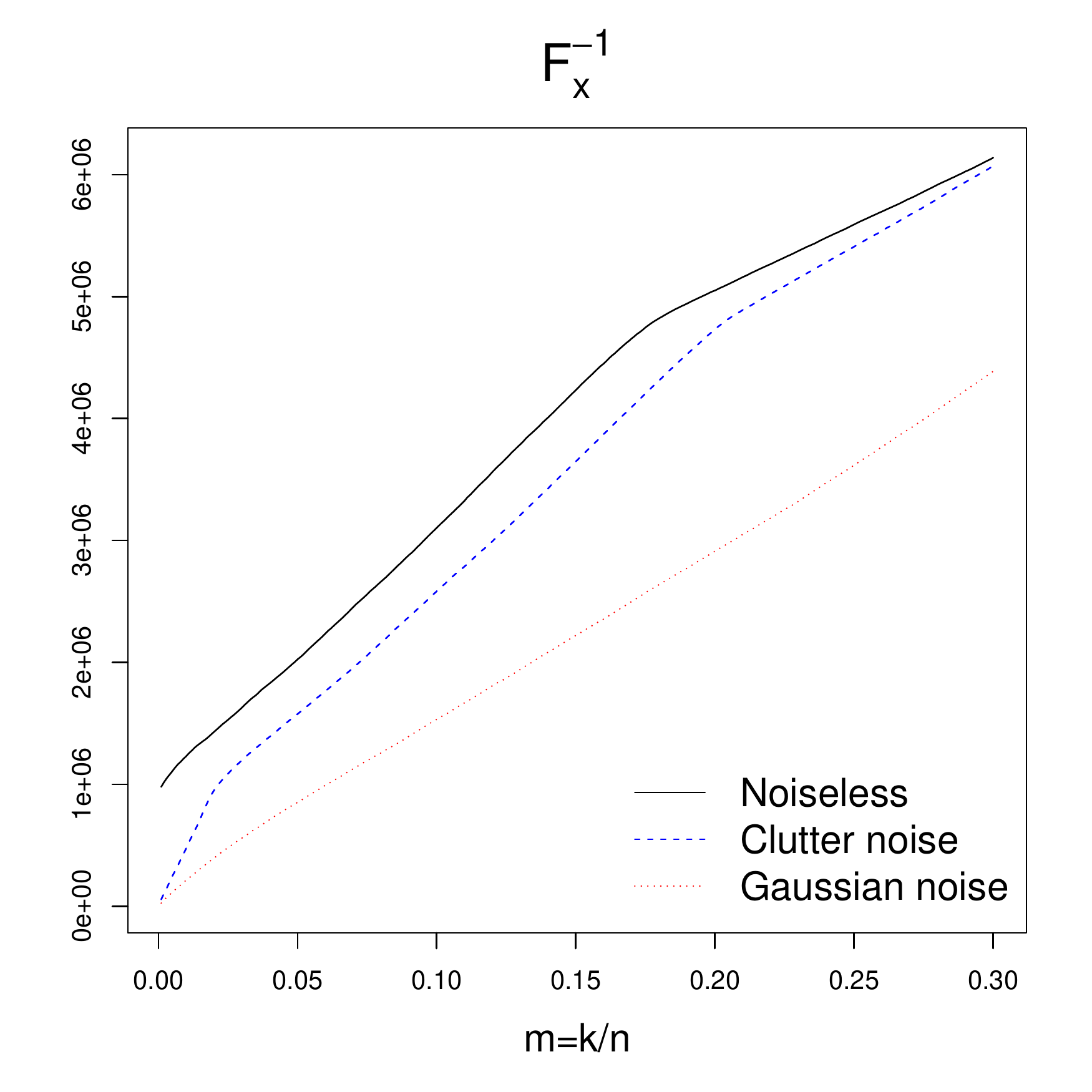} 
   \end{minipage}
   \begin{minipage}[b]{0.5\linewidth}
      \centering \includegraphics[scale=0.42]{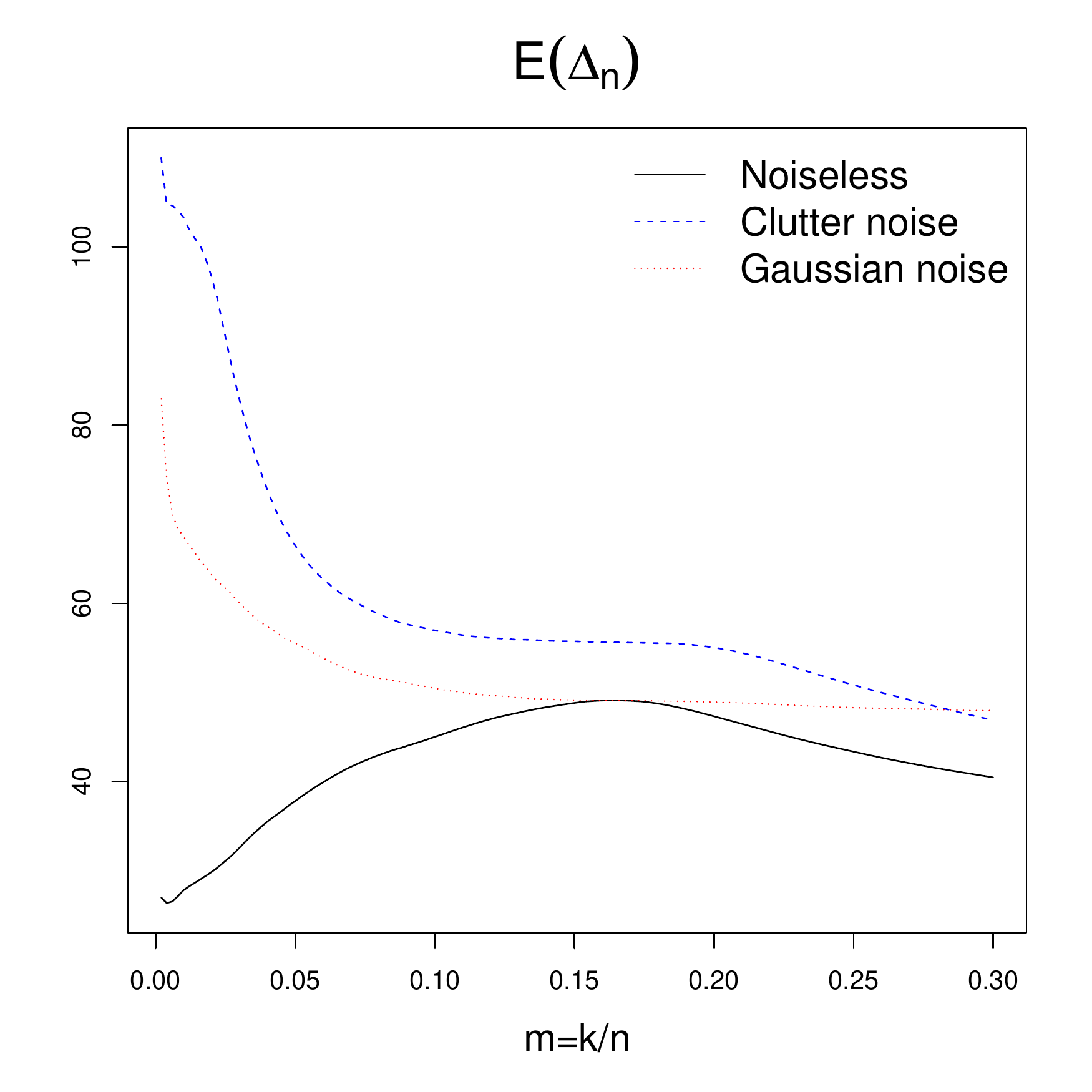} 
   \end{minipage}\hfill
   \begin{minipage}[b]{0.5\linewidth}   
      \centering \includegraphics[scale=0.42]{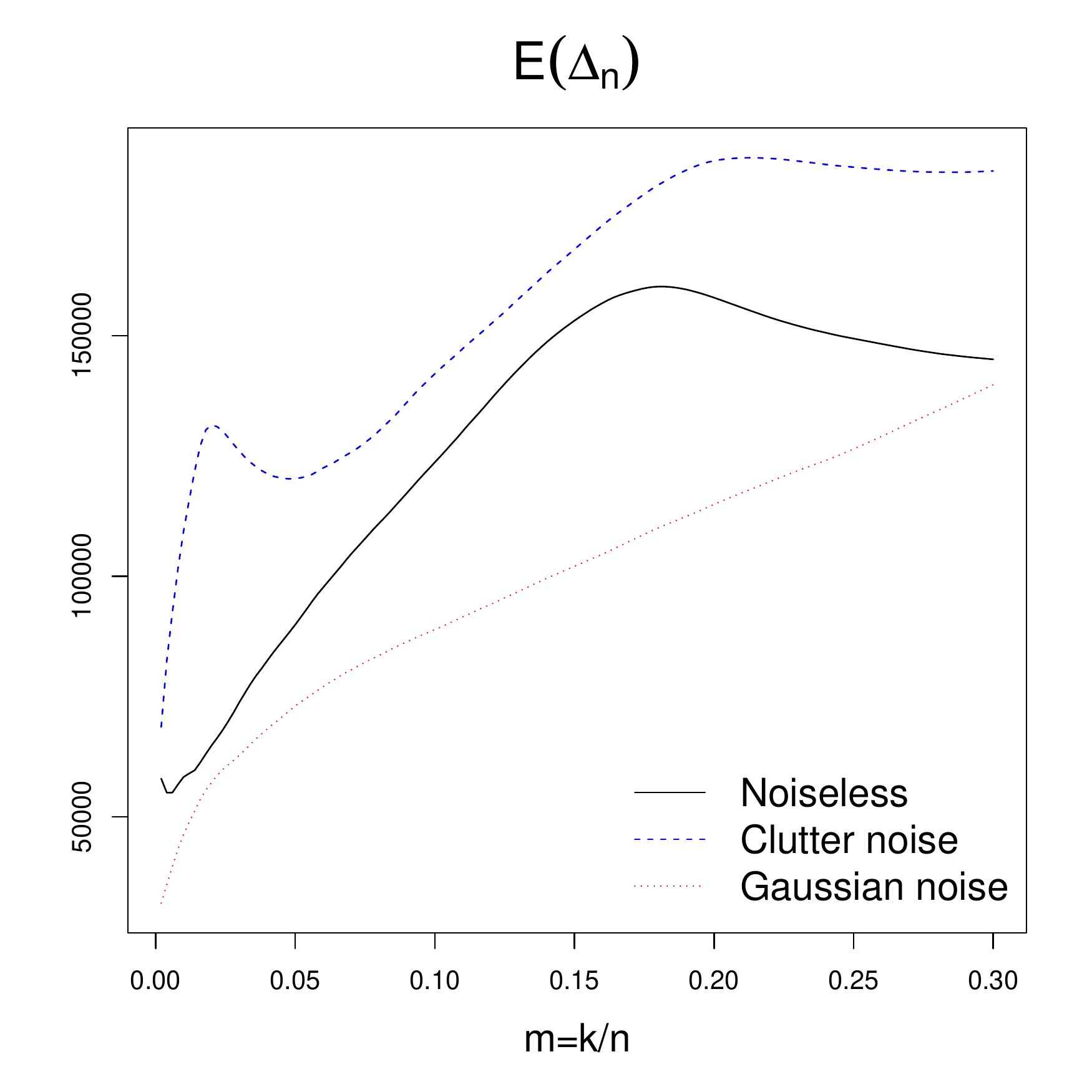} 
   \end{minipage}
   \begin{minipage}[b]{0.5\linewidth}
      \centering \includegraphics[scale=0.42]{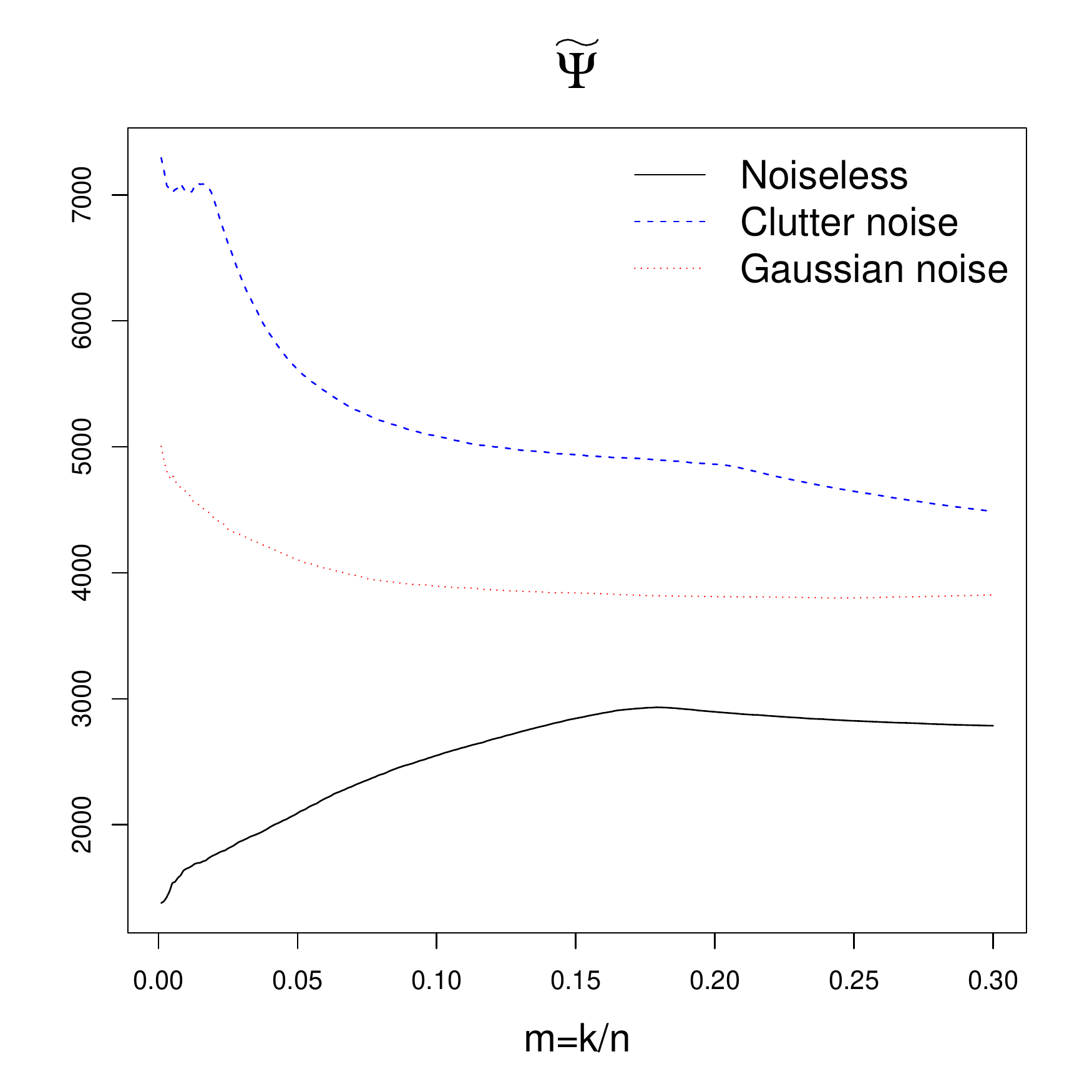} 	
   \end{minipage}\hfill
   \begin{minipage}[b]{0.5\linewidth}   
      \centering \includegraphics[scale=0.42]{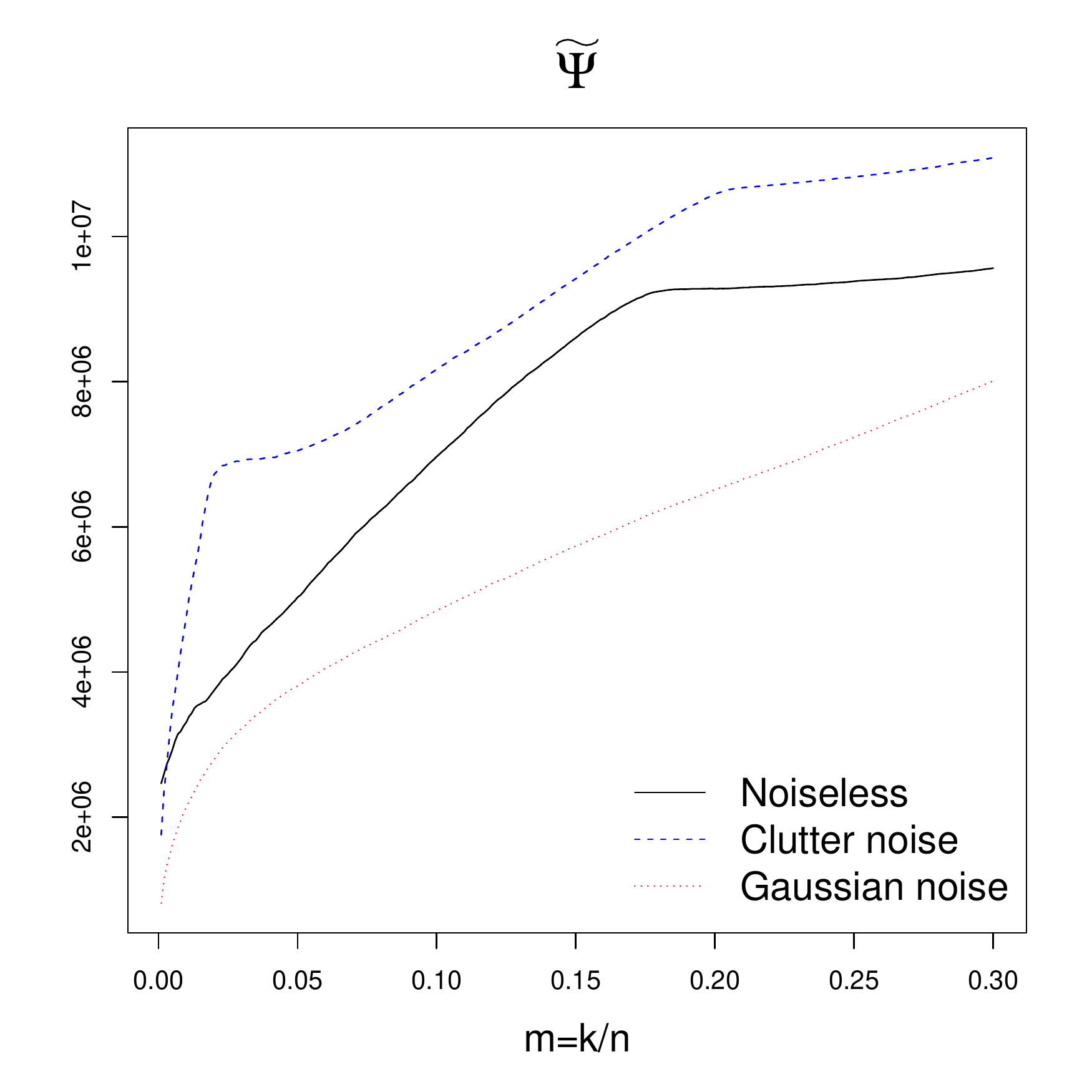} 
   \end{minipage}
   \caption{Quantiles functions $F_{x,r}^{-1}$ (top), expected error  $\E \Delta_{n,\frac kn,r}(x) $ (middle) and
theoretical upper bounds $\tilde \Psi$ (bottom) with powers $r=1$ (left) and $r=2$ (right), for the 2-d Shape Experiment.}
      \label{fig:Error2dshape}
\end{figure}

\begin{figure}[pH]
   \begin{minipage}[b]{0.5\linewidth}
      \centering \includegraphics[scale=0.42]{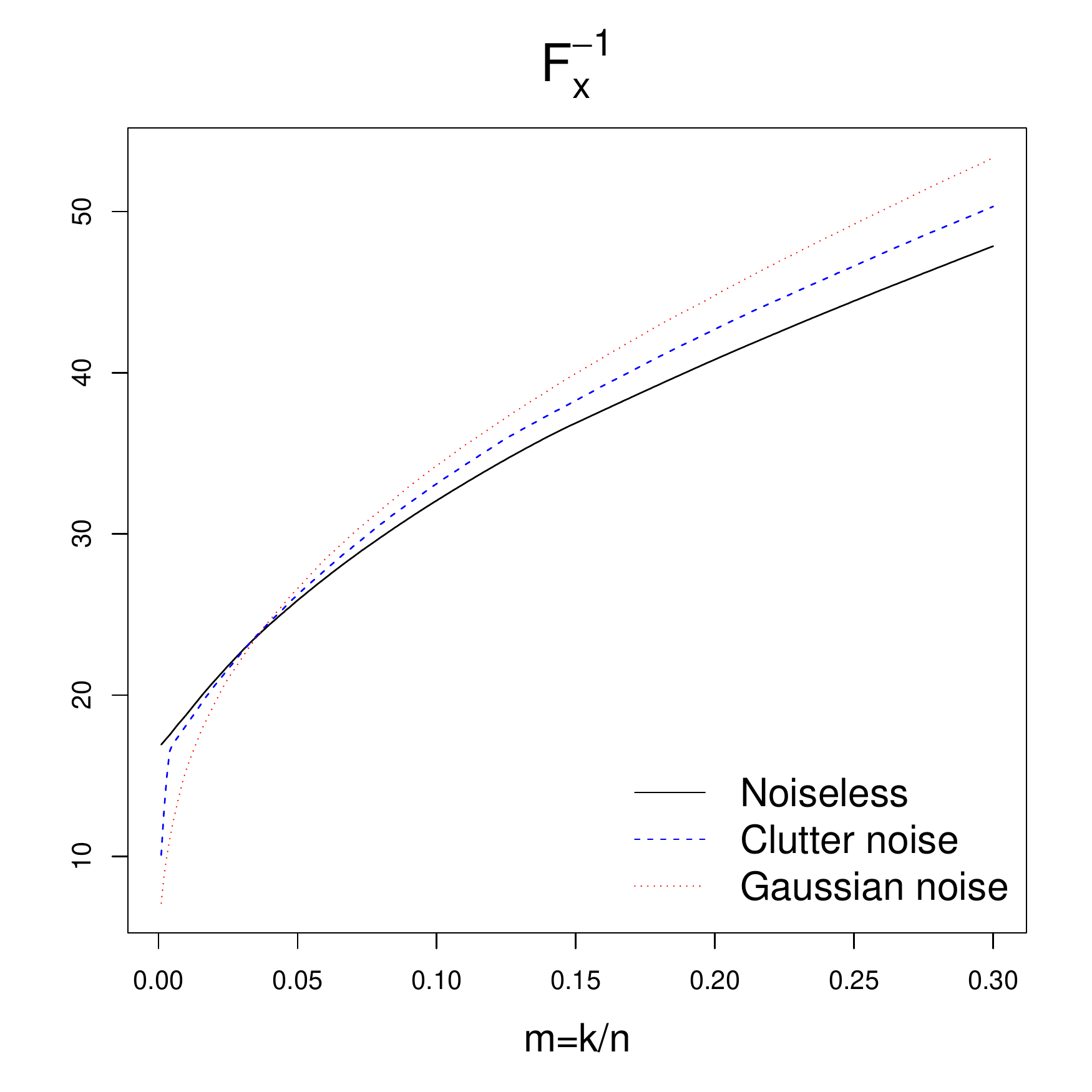} 
   \end{minipage}\hfill
   \begin{minipage}[b]{0.5\linewidth}   
      \centering \includegraphics[scale=0.42]{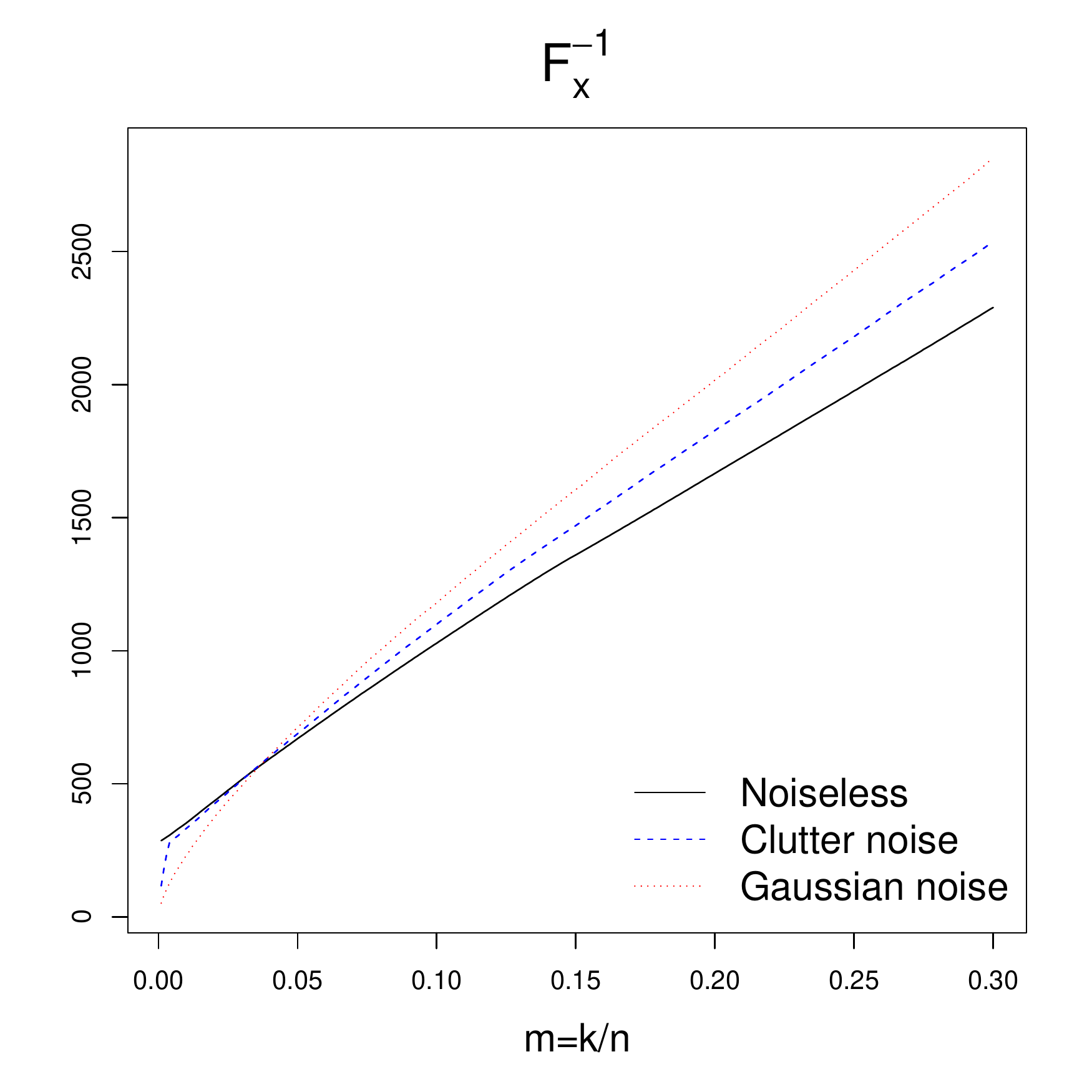} 
   \end{minipage}
   \begin{minipage}[b]{0.5\linewidth}
      \centering \includegraphics[scale=0.42]{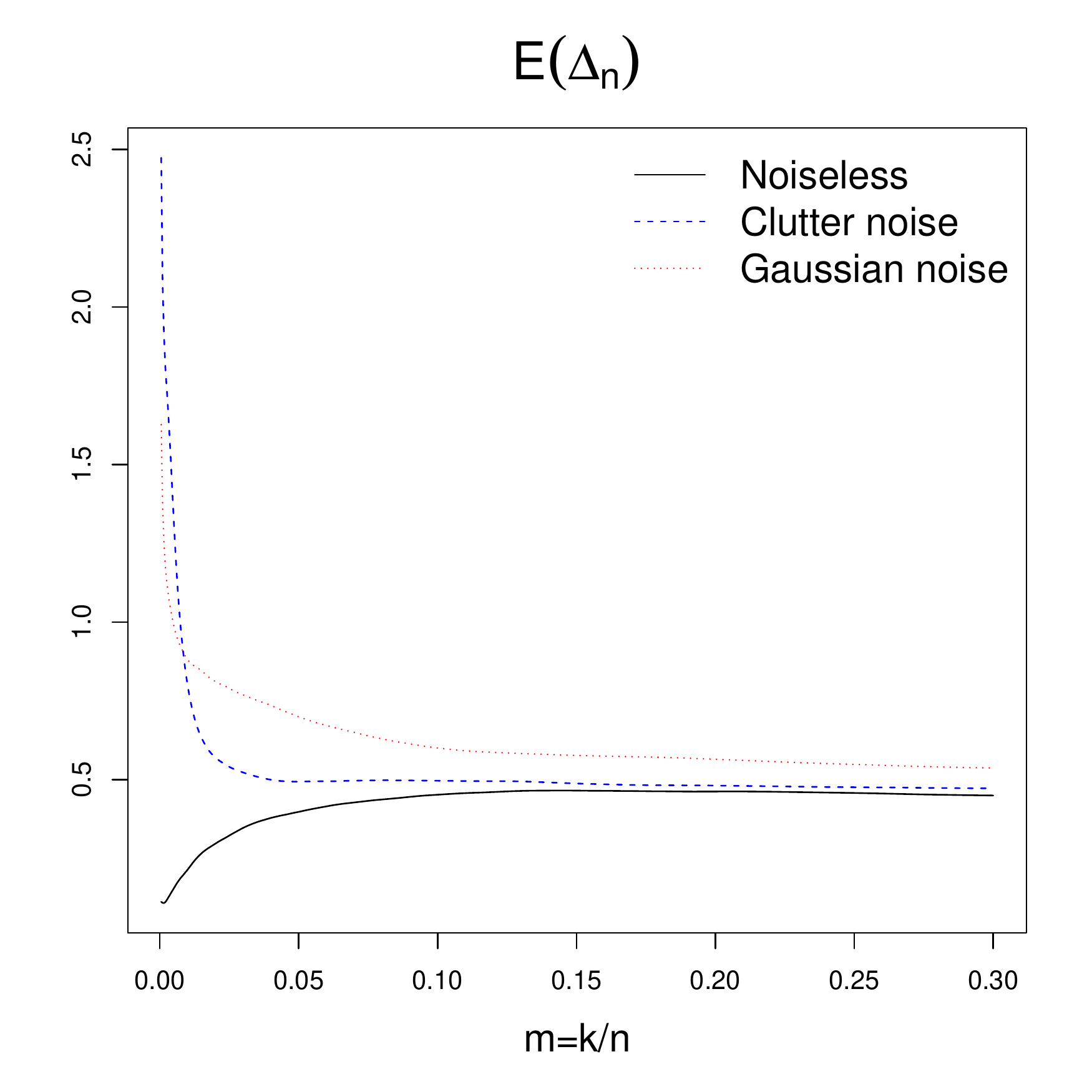} 
   \end{minipage}\hfill
   \begin{minipage}[b]{0.5\linewidth}   
      \centering \includegraphics[scale=0.42]{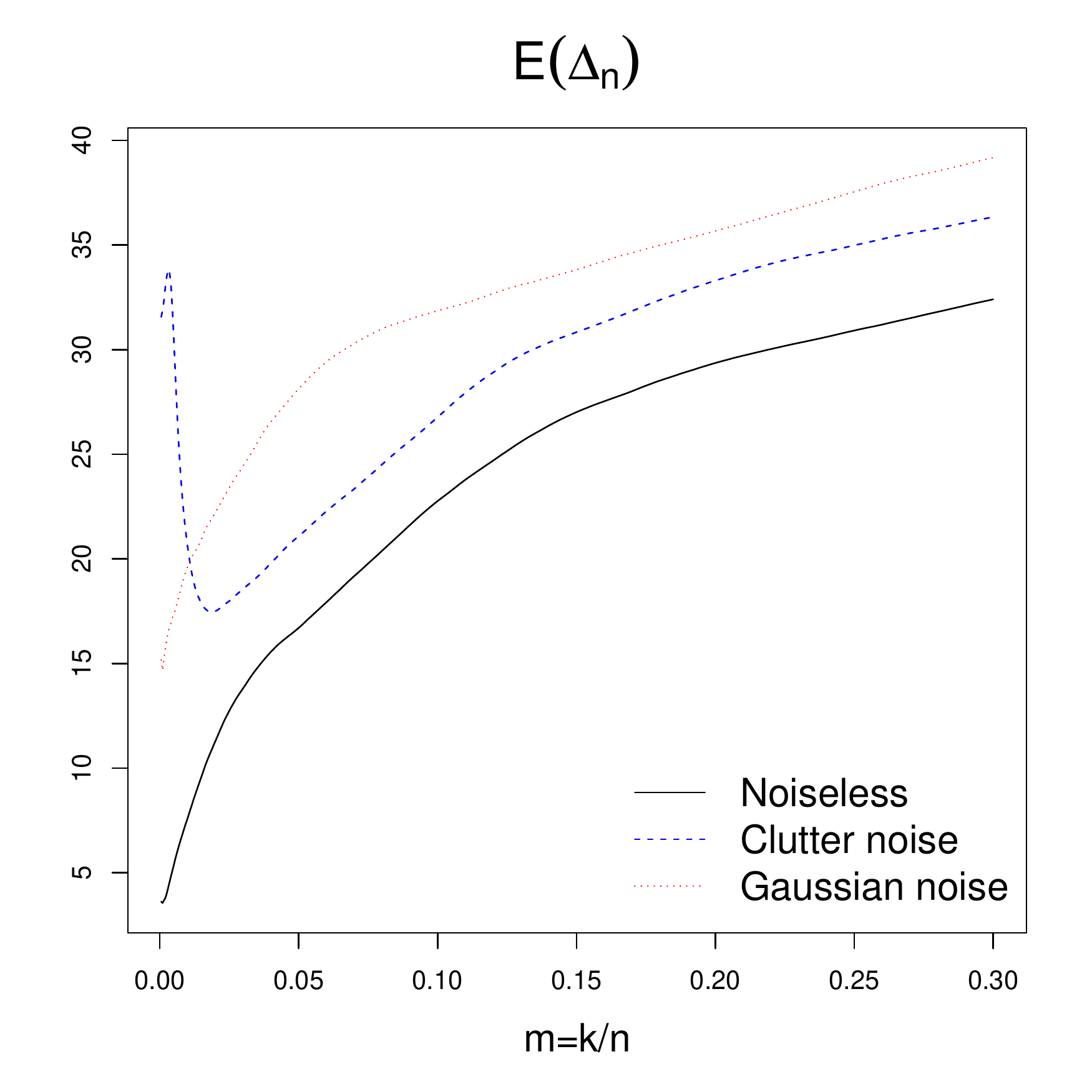} 
   \end{minipage}
   \begin{minipage}[b]{0.5\linewidth}
      \centering \includegraphics[scale=0.42]{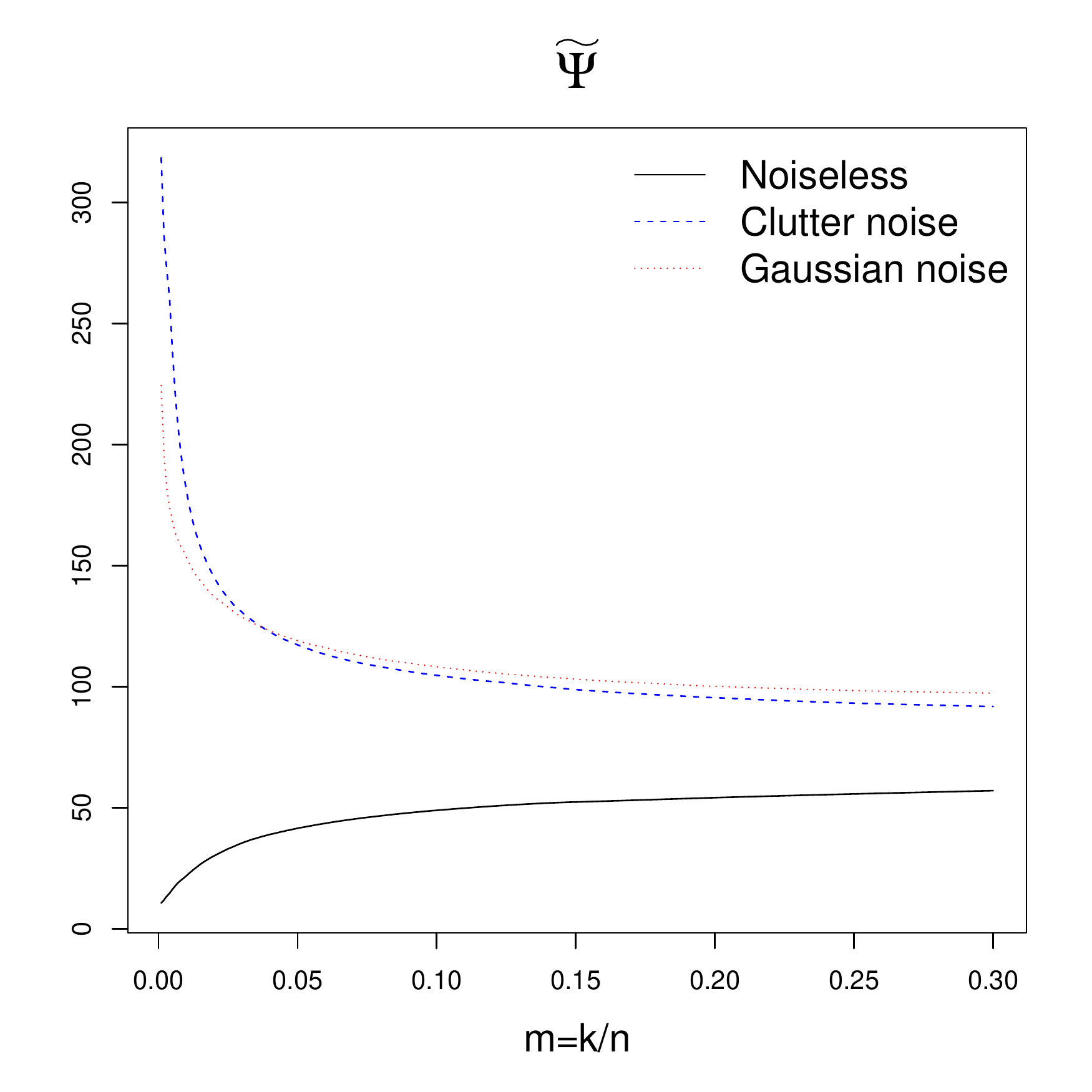} 	
   \end{minipage}\hfill
   \begin{minipage}[b]{0.5\linewidth}   
      \centering \includegraphics[scale=0.42]{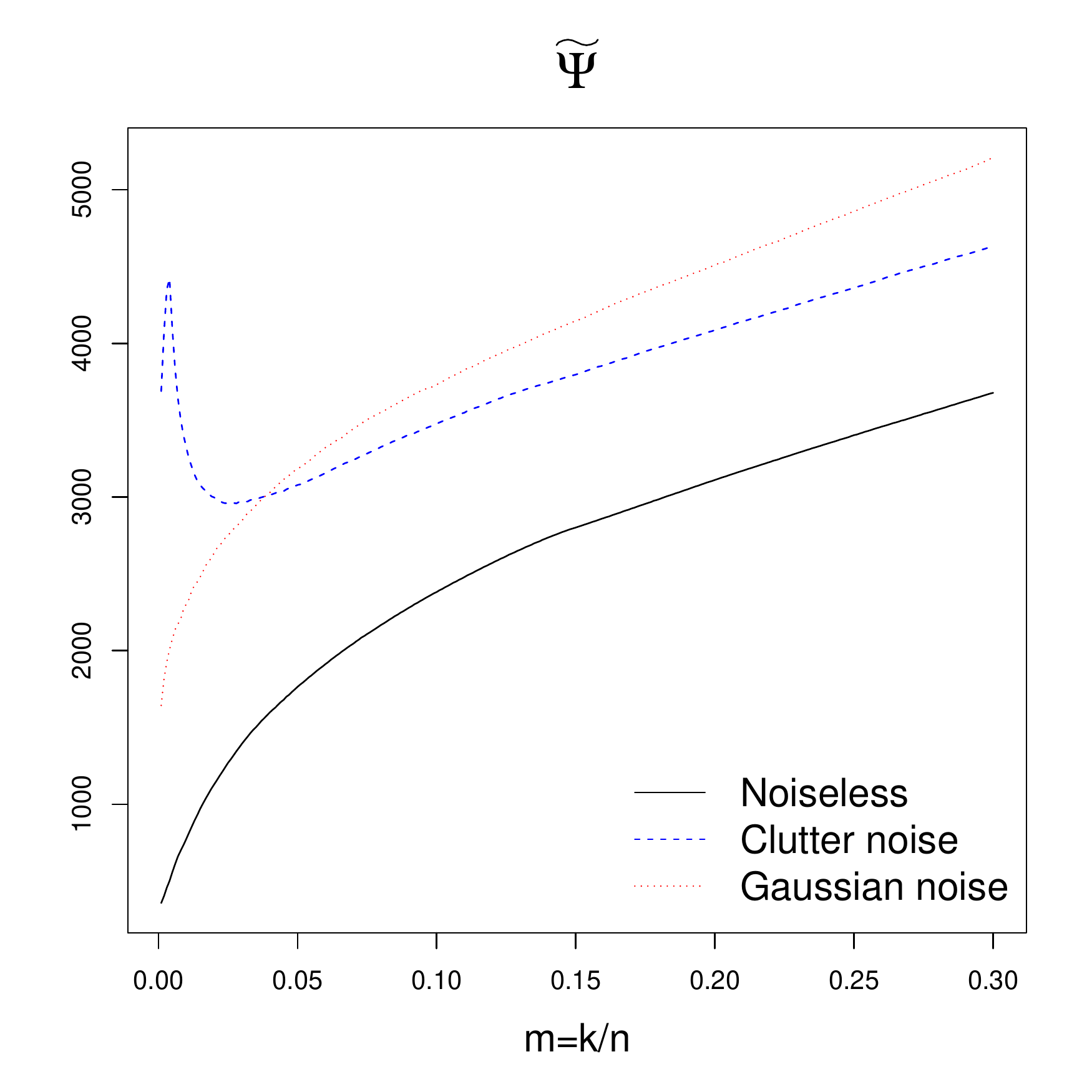} 
   \end{minipage}
   \caption{Quantiles functions $F_{x,r}^{-1}$ (top), expected error  $\E \Delta_{n,\frac kn,r}(x) $ (middle) and
theoretical upper bounds $\tilde \Psi$ (bottom) with powers $r=1$ (left) and $r=2$ (right), for the Fish Experiment.}
      \label{fig:ErrorFish}
\end{figure}

\begin{figure}[pH]
   \begin{minipage}[b]{0.32\linewidth}
      \centering \includegraphics[scale=0.35]{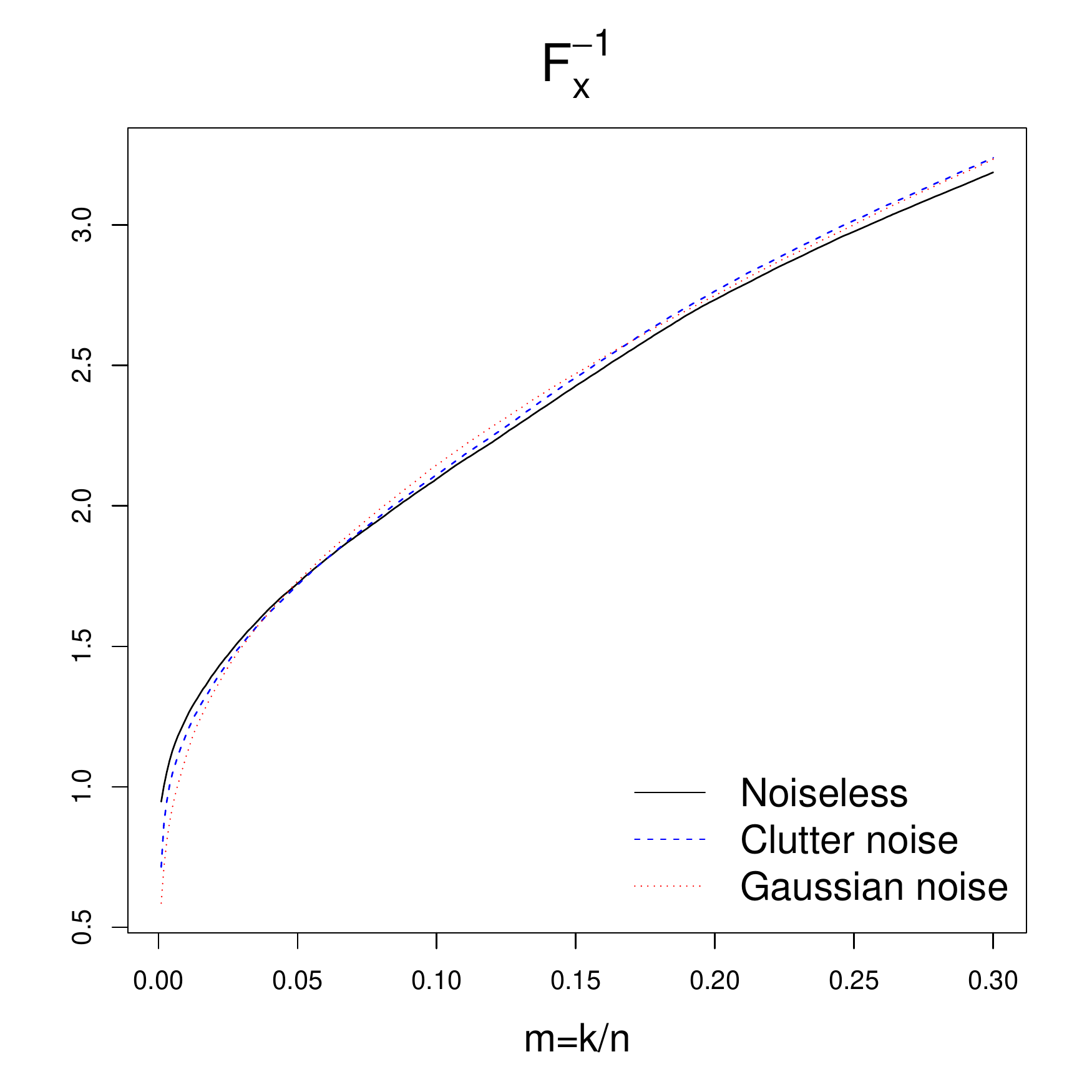} 
   \end{minipage}\hfill
   \begin{minipage}[b]{0.32\linewidth}   
      \centering \includegraphics[scale=0.35]{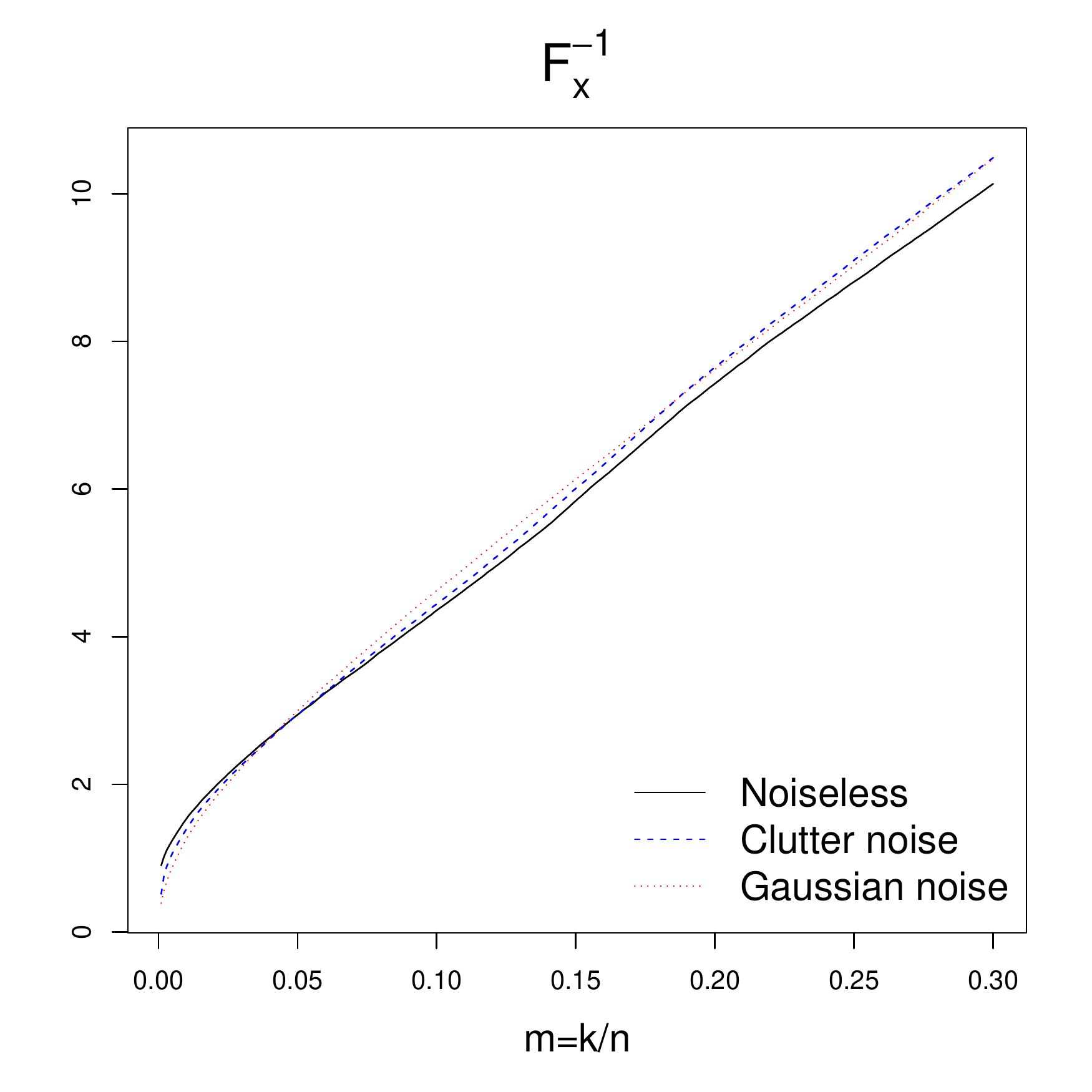} 
   \end{minipage}
   \begin{minipage}[b]{0.32\linewidth}   
      \centering \includegraphics[scale=0.35]{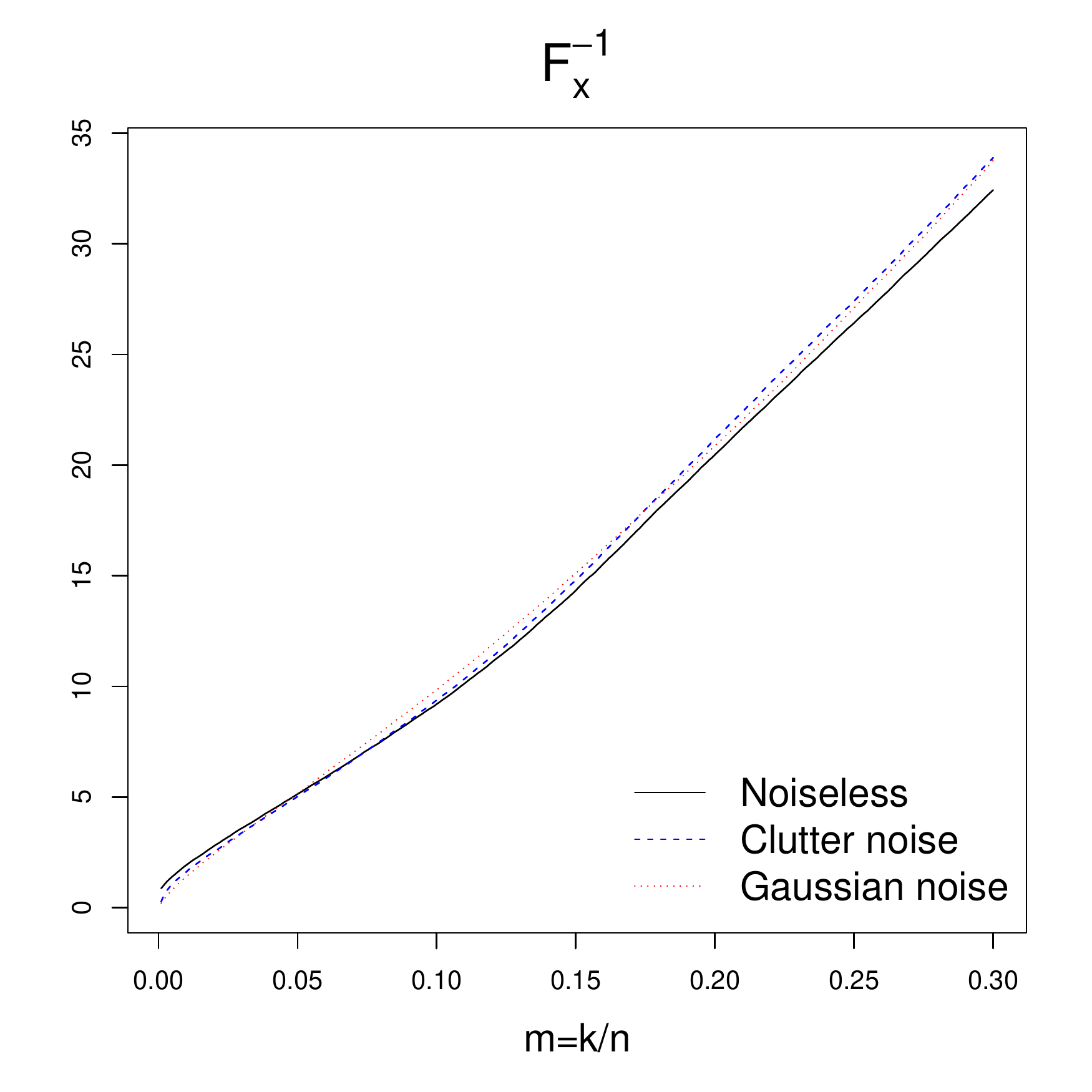} 
   \end{minipage}   
   \begin{minipage}[b]{0.32\linewidth}
      \centering \includegraphics[scale=0.35]{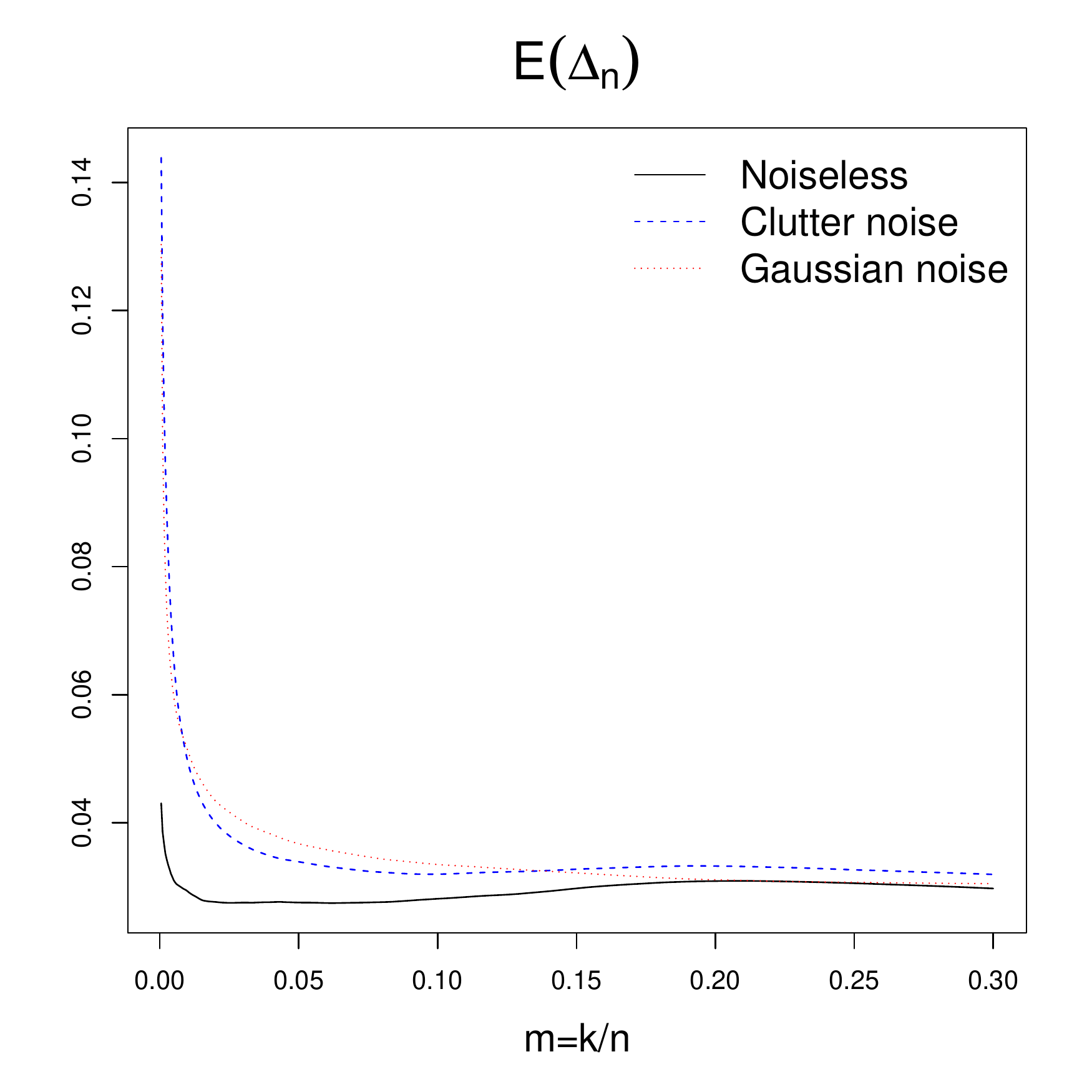} 
   \end{minipage}\hfill
   \begin{minipage}[b]{0.32\linewidth}   
      \centering \includegraphics[scale=0.35]{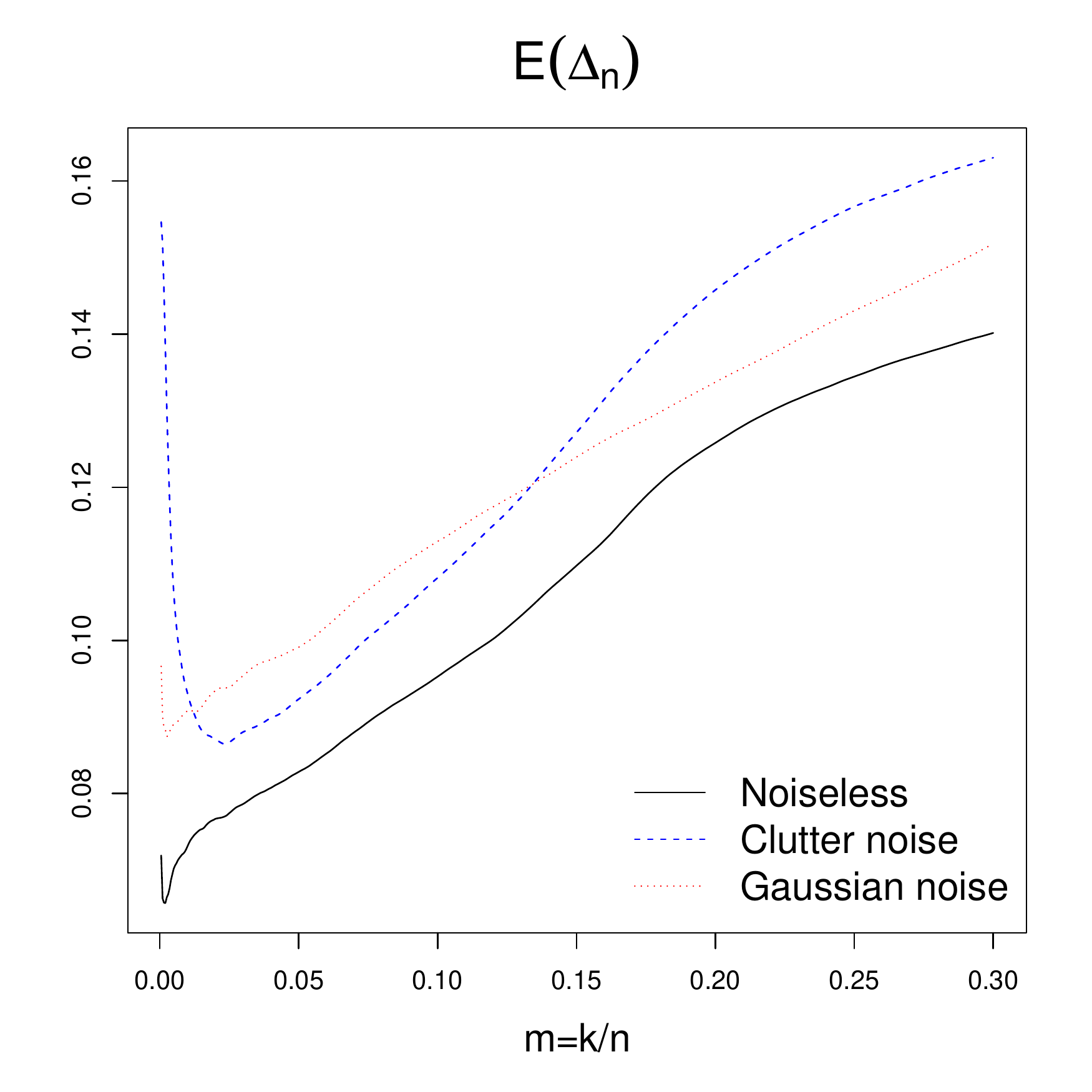} 
   \end{minipage}
    \begin{minipage}[b]{0.32\linewidth}   
      \centering \includegraphics[scale=0.35]{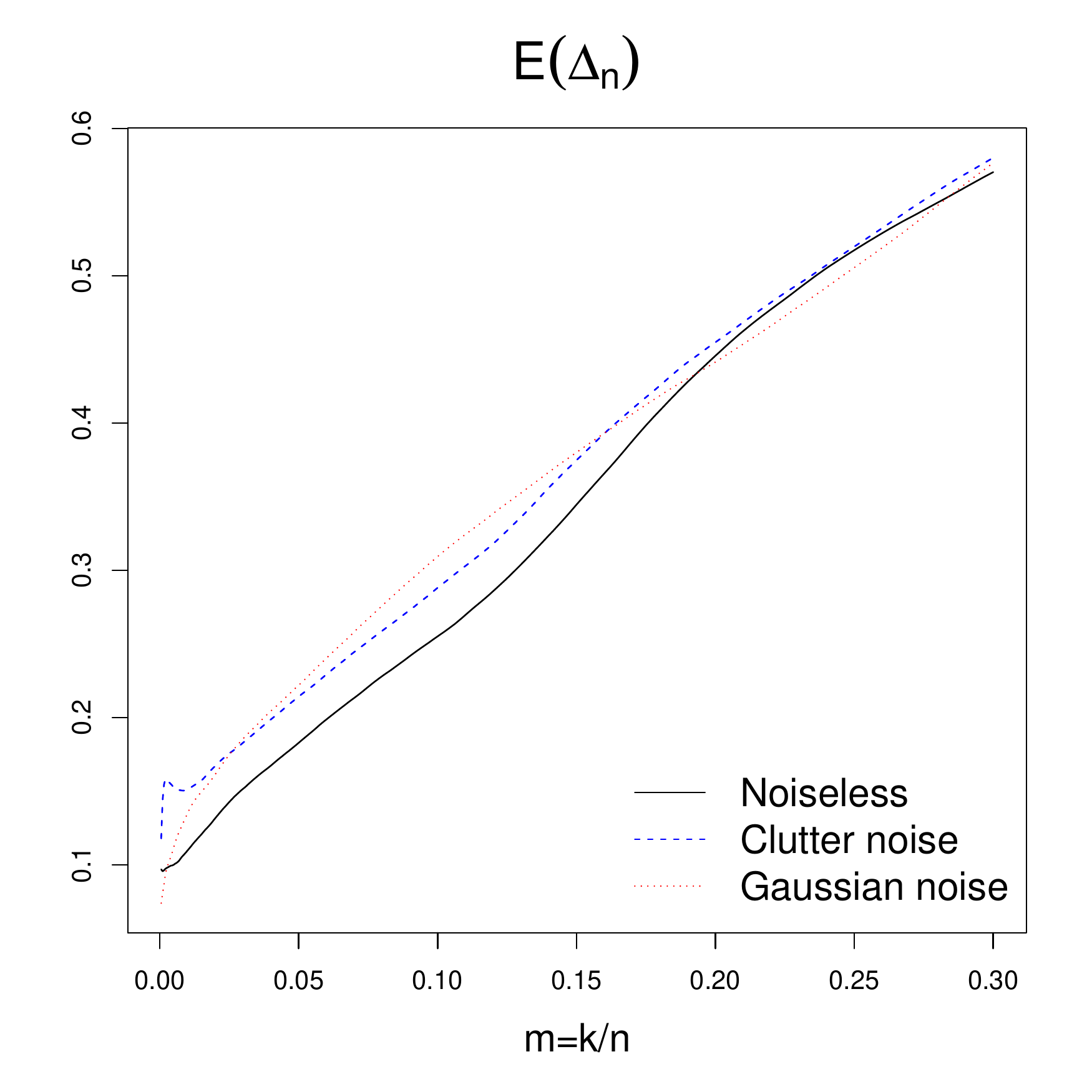} 
   \end{minipage}  
   \begin{minipage}[b]{0.32\linewidth}
      \centering \includegraphics[scale=0.35]{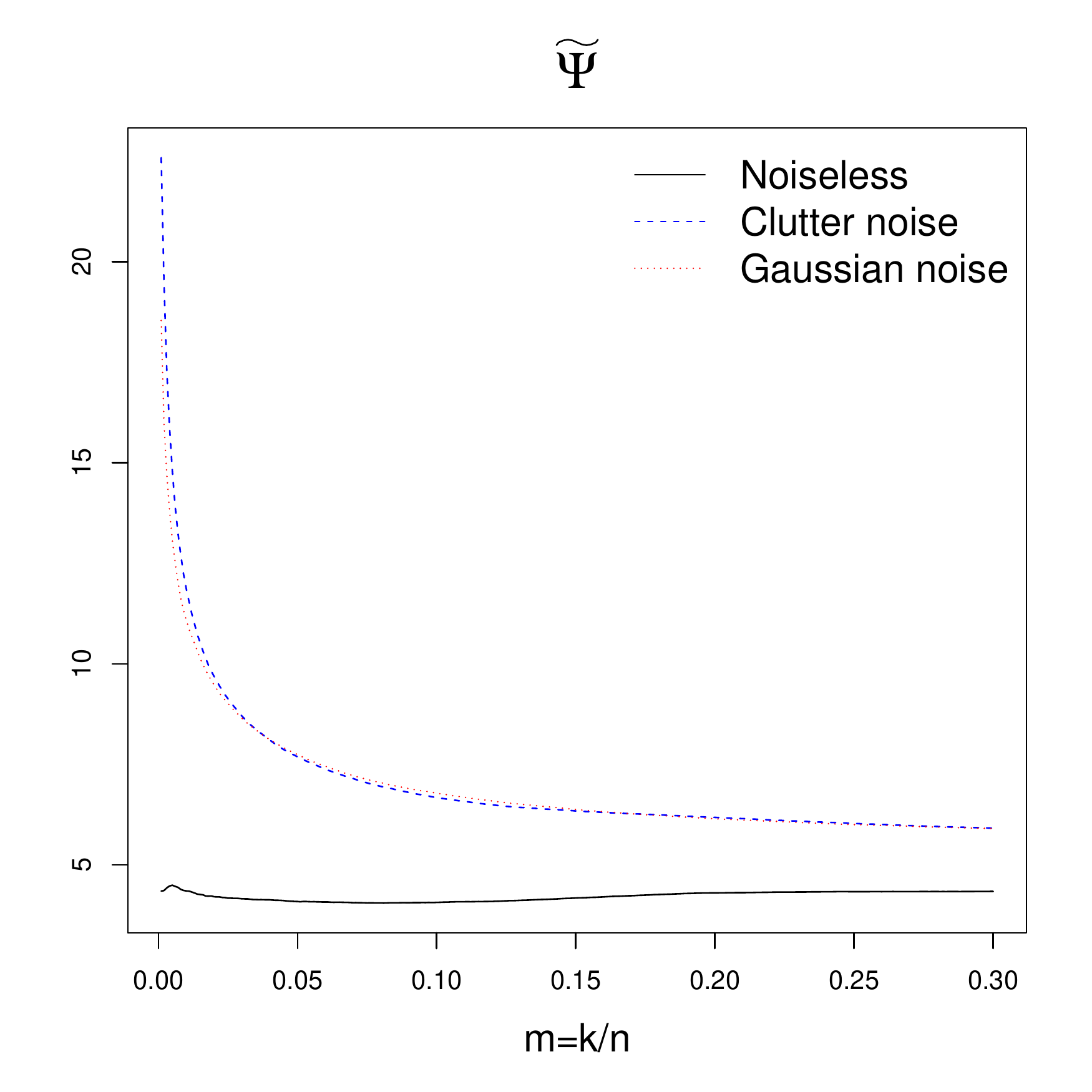} 
   \end{minipage}\hfill
   \begin{minipage}[b]{0.32\linewidth}   
      \centering \includegraphics[scale=0.35]{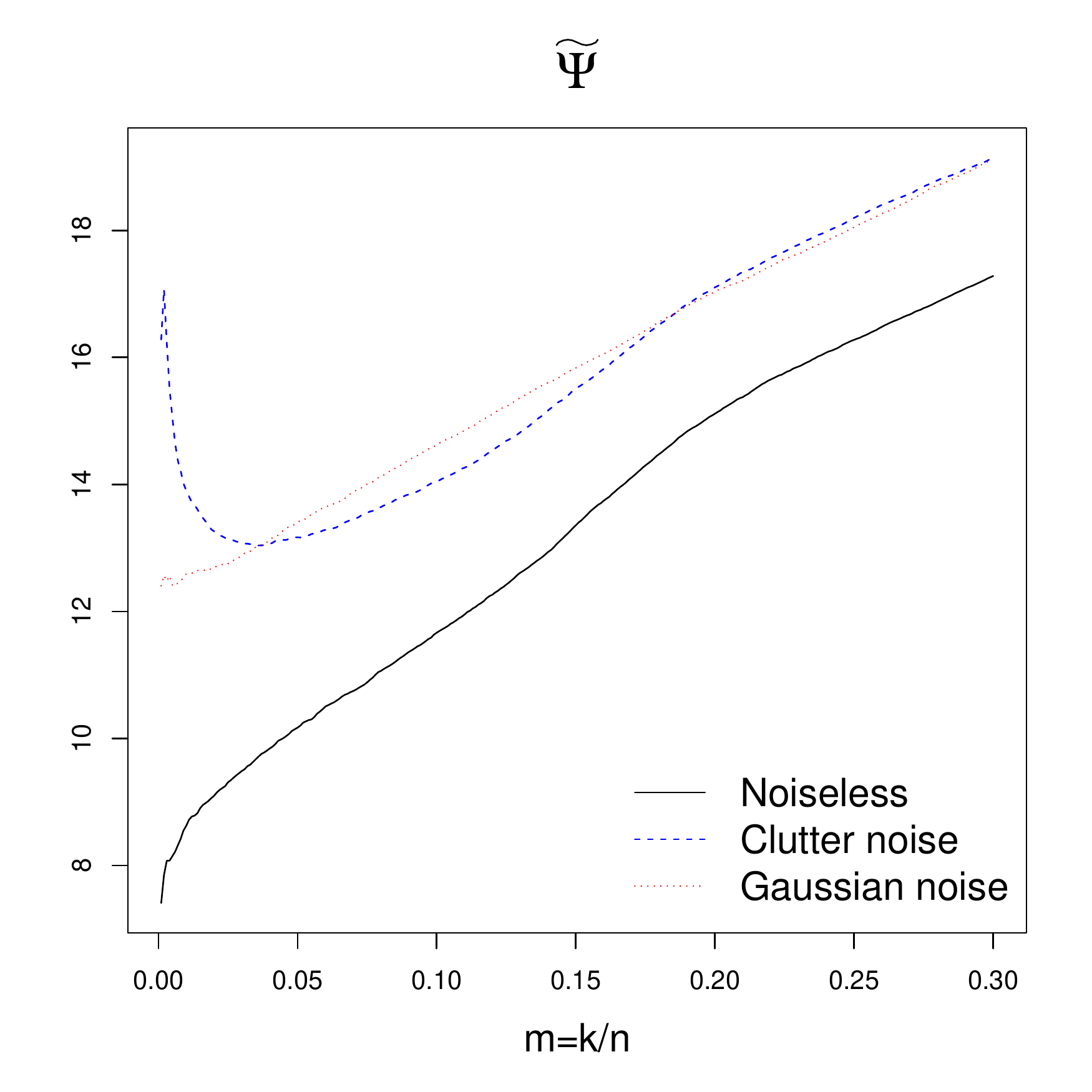} 
   \end{minipage}
   \begin{minipage}[b]{0.32\linewidth}   
      \centering \includegraphics[scale=0.35]{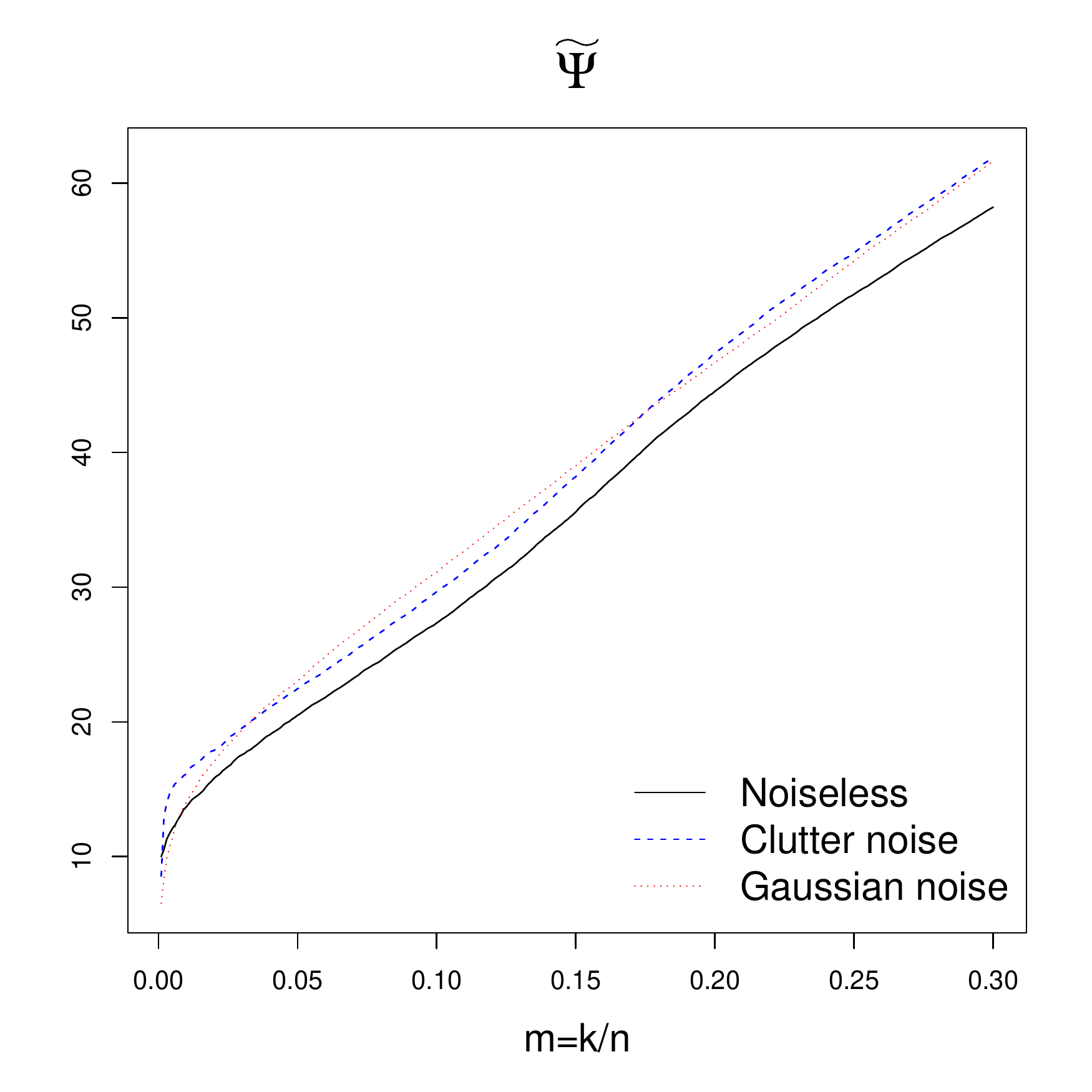} 
   \end{minipage}   
   \caption{Quantiles functions $F_{x,r}^{-1}$ (top), expected error  $\E \Delta_{n,\frac kn,r}(x) $ (middle) and
theoretical upper bounds $\tilde \Psi$ (bottom) with powers $r=1$ (left) and $r=2$ (right), for the Tangle Cube Experiment.}
        \label{fig:ErrorTangle} 
\end{figure}

\section{Conclusion}

When the data is corrupted by noise, the distance to measure is one clue for performing robust geometric inference. For instance it can be used for  support  estimation and for topological data analysis using persistence diagrams, as proposed in \cite{chazal2014robust}. In practice,  a ``plug-in" approach is adopted by replacing the measure by its empirical counterpart in the definition of the DTM. The main result  of this paper is providing sharp non asymptotic bounds on the deviations of the DTEM. 

The DTM has been recently extended to the context of metric spaces in~\cite{buchet2015efficient}. For the sake of simplicity, we have  assumed  that $P$ is a probability measure in $\R^d$. However, all the results of the paper can be easily adapted to 
more general metric spaces by considering the push forward distribution of $P$ by $d(x,\cdot)^r$ where $d$ is the metric in the sampling space.

This  paper  is a  step toward  a complete theory about robust geometric inference. Our results  give preliminary insights about how tuning  the parameter  $m$ in the DTEM, which is a difficult question. The experiments proposed in Section~\ref{subs:expe} show  that the term  $\E   \Delta_{n,m,r}(x) $ does not have a typical monotonic behavior with regard to $m$ and thus classical model selection methods can be hardly applied to this problem. We intend to study this non standard model selection problem in future works.

\appendix 


\section{Rates of convergence  derived from the DTM stability} \label{sec:RatesStability}

\medskip

The DTM satisfies several stability properties for the Wasserstein metrics. In this section, rates of convergence of the DTEM are derived from stability results of the DTM together with known results about the convergence of the empirical measure under Wasserstein metrics. We check that the results derived in this way are not as tight as the results given in Section~\ref{sec:mainresults}.

 Let us first  remind the definition of the Wasserstein metrics in $\R^d$. For $r \geq 1$, the Wasserstein distance $W_r$ between two probability
measures $P$ and $\tilde P$ on $\R^d$  is given by
$$W_r(P,\tilde P)=\inf _{\pi \in \Pi(P,\tilde P)}\left(\int_{\R^d \times \R^d} \|x-y\|^r \pi (dx,dy)\right)^{\frac
1r},$$
where $\Pi(P,\tilde P)$ is the set of probability measures on ${ \R^d} \times {\R^d}$ with marginal distributions $P$ and $\tilde P$, see for instance \cite{RR98} or \cite{Vil08}. 

The stability of the DTM with respect to the Wasserstein distance  $W_r$ is given by the following theorem. 
 \begin{Theo} [\cite{chazal2011geometric}] \label{Theo:StabWp}  
 Let $P$ and $\tilde P$ be two probability measures on $\R^d$. For any $r\geq 1$ and any  $m \in (0,1)$ we have
\begin{equation*} \label{eq:DTMW}
 \|d_{P,m,r} - d_{\tilde P,m,r} \|_\infty \leq  m^{- \frac 1 r} W_r (P,\tilde P).
\end{equation*}
\end{Theo}

Notice that ~\cite{chazal2011geometric} prove this theorem for $r=2$, but the proof for any $r \geq 1$ is exactly the same. 

We now give the pointwise stability of the DTM with respect to the Kantorovich distance $W_1$  between push forward measures on $\R$. This result  easily derives from  the expression \eqref{eq:DTMQuantile} of the DTM given in Introduction, a rigorous proof is given in Appendix\ref{subsec:proofStab}.
\begin{Prop} \label{DTM-stabW1}
For some point $x$  in $\R^d$ and some real number $r \geq 1$, let $d F_{x,r}$ and $d \tilde F_{x,r}$ be the push-forward measures by the function $y \mapsto \|x -y\| ^r$ of two probability measures $P$ and $\tilde P$ defined on $\R^d$. Then, for any $x \in \R^d$:
\begin{equation*} \label{eq:DTMW1}
 \left|d_{P,m,r} ^r (x)  - d_{\tilde P,m,r}^r (x) \right|  \leq  \frac 1 m  W_1 (d F_{x,r} , d \tilde F_{x,r}) \, .
\end{equation*}
\end{Prop} 

Convergence results for $\Delta_{n,m,r}$ can be directly derived from the stability results given in
Theorem~\ref{Theo:StabWp} and Proposition~\ref{DTM-stabW1}. For instance, it can be easily checked  that, for any $x\in \R^d$, $W_1 (d
F_{x,r} , d F_{n,x})$ tends to zero almost surely  \citep[see for instance the Introduction Section of][]{BGM99}. This
together
with Proposition~\ref{DTM-stabW1} gives the almost surely pointwise convergence to zero of $\Delta_{n,m,r}(x)$.

Regarding the convergence in expectation, using Theorem~\ref{Theo:StabWp} in $\R^d$ for $d > r/2$, we deduce from \cite{FG} or from \cite{DSS13} that  
\begin{eqnarray*}
\E  \| \Delta_{n,\frac kn,r} \|_\infty   &\leq&  \left( \frac kn\right)  ^{-1/r} \E  W_r(P,P_n)   \\ 
& \leq  &   \left( \frac kn\right) ^{-1/r} \left[   \E W_r ^r (P,P_n) \right]^{1/r} \\
& \leq  &  C  \left( \frac kn\right)  ^{-1/r} n ^{-1/d}.
\end{eqnarray*}
Nevertheless this upper bound is not sharp: assume that $k_n:= m n$ for some fixed constant $m \in (0,1)$ then the rate is of the order of $n ^{-1/d}$. We show below that the parametric rate $1/ \sqrt n$ can be obtained by considering the alternative stability result given in Proposition~\ref{DTM-stabW1}. In the one-dimensional case,  a direct application of Fubini's theorem gives that \citep[see for instance Theorem 3.2 in][]{bobkov2014one} 
\begin{equation}
\label{eq:majorKantor}
  \sqrt n  \E \left[ W_1(dF_{x,r} , dF_{x,r,n} ) \right]  \leq   \int_0^{\infty} \sqrt{ F_{x,r}(t) (1 - F_{x,r}(t))  }  dt  \:  =: J_1(dF_{x,r} ) ,
  \end{equation}
where $dF_{x,r}$ and $dF_{x,r,n}$ are the push forward  probability measures of $P$ and $P_n$ by the function $\|x-\cdot \| ^r$.  Note that \cite{bobkov2014one} have completely characterized the  convergence of $ \E W_1( \mu , \mu_n )$ in the one-dimensional case, in term  of $ J_1( \mu )$ for $\mu$ a probability measure on the real line and   $\mu_n$  its empirical counterpart.
From Proposition~\ref{DTM-stabW1} and the upper bound~\eqref{eq:majorKantor} we derive that
\begin{equation}
\label{upperW1}
\E  \left| \Delta_{n,\frac kn,r} (x)   \right|   \leq       \frac n k  \frac { J_1(dF_{x,r} )} {\sqrt n} .
\end{equation}
The integral $ J_1(dF_{x,r} )$ is finite if  $\E \| X-x\| ^{2r + \delta} < \infty $ for some $\delta >0$.  We thus obtain a pointwise rate of convergence of $1/\sqrt n$ under reasonable moment conditions, if we take $k_n:= m n$ for some fixed constant $m \in (0,1)$. However, the upper bound \eqref{upperW1} does not allow us to describe correctly how the rate  depends on the parameter $m= \frac kn$.  For instance, if $\frac kn$ is very small, the bound blows up in all cases while it  should not be the case for instance with discrete measures. The reason is that the stability results are too global to provide a sharp expectation bound for small values of $\frac kn$.

\section{Proofs} \label{app:proofs}

\subsection{Preliminary results for the DTM} \label{subsec:proofStab}

\subsubsection*{Rewritting the DTM in terms of quantile function}

Let $P$ a probability distribution in $\R^d$, $x \in \R^d$ and $r \geq 1$. Let $F_{x,r} $ be the distribution function of the random variable $\|x-X \| ^r$, where the distribution of the random variable $X$ is $P$. The preliminary distance function to $P$ 
\begin{equation*}
\delta_{P, u} : x \in \R^d  \mapsto \inf\{ t > 0 \, ; \,  P(\bar B (x,t) ) \geq  u  \} 
\end{equation*}
can be rewritten in terms of the
quantile function $F_{x,r}$:
\begin{Lemma} 
\label{Lem:DTMquantile}
 For any  $ u \in (0,1)$, we have  $ \delta^r_{P,u}(x)=  F_{x,r} ^{-1}(u)$. In particular, $ \delta _{P,u}(x)=  F_{x,1}
^{-1}(u)$.
\end{Lemma} 
\begin{proof}
Note that  for any $t \in \R^+$, $F_{x,r} (t)  =  P ( B(x,{t}^{1/r}))$. Next, 
$$ \{ t \geq  0 \, ; \,    P ( B(x,{t}^{1/r})) \geq   \ell   \}   =  \{ s^r \, ; s \geq  0 \, , \,   P ( B(x,s)) \geq  
\ell  \} $$
and we deduce that
\begin{eqnarray*}
F_{x,r} ^{-1}(\ell)  & = &  \inf\{ s^r \, ; s \geq  0 \, , \,   P ( B(x,s))  \geq   \ell   \} \\
& =  &  \delta^r_{P,\ell}(x).
\end{eqnarray*}
where we have used the continuity of $s \mapsto s^r$ for the last equality.
\end{proof}
From Lemma~\ref{Lem:DTMquantile} we directly derive the expression of the DTM in terms of the quantile function $F_{x,r}^{-1}$, as given by Equation~\eqref{eq:DTMQuantile} in the Introduction Section:
\begin{equation*}      
 d^r_{P,m,r}(x) = \frac 1m \int_0^m    F_{x,r}^{-1}(u) d u .
\end{equation*}

\subsubsection*{Proof of Proposition~\ref{Theo:StabWp}.}

Let  $F$ and $\tilde F$ be the cdfs of two probability measures $dF$ and $d \tilde F$ on $\R$. Recall that, for any $r\geq 1$, and any measure $\mu$ and $\tilde \mu$ in $\R$:
\begin{equation} \label{WpFinv}
W_r^r(dF , d\tilde F)=\int_0^1|\tilde F^{-1}(u)-F^{-1}(u)|^r du \, ,
\end{equation}
see for instance see for instance \cite{vallender1974calculation}. Thus
\begin{eqnarray*}      
\int_0^m  | \tilde F^{-1}(u)  - F^{-1}(u)  | du \leq W_1^r(F , \tilde F)
\end{eqnarray*} 
and the proof follows using Equation \eqref{eq:DTMQuantile}.

\subsubsection*{A decomposition of $\Delta_{n,\frac kn,r}$.}

For any $x\in \R^d$, any $r\geq 1$ we have $ F_{x,n,r} ^{-1}(0) \geq F_{x,r}^{-1}(0) \geq 0 $
since  $F_{x,r}$ is the cdf of the random distance $\| x-X\|^r$ whose support is
 included in $\R^+$. From Equation~\eqref{eq:DeltaW1} and geometric
considerations (see figure~\ref{fig:LemmeJerome}) we can rewrite $\Delta_{n,m,r}$ as given in the following Lemma.
\begin{Lemma} \label{DecomposDeltaJerome} The quantity $\Delta_{n,\frac kn,r}$ can be rewritten as follows:
\begin{eqnarray*} 
\Delta_{n,\frac kn,r}(x)  &:=& 
 \frac n k \int_{0}^{\frac kn} \left\{      F_{x,n,r}^{-1}(u)  - F_{x,r}^{-1}(u) \right\} du \\
&=&  \frac n k \int_{F_{x,r}^{-1}(0)}^{F_{x,r}^{-1}(\frac k n ) \vee F_{x,n,r}^{-1}(\frac k n )} 
\left\{  F_{x,r}(t) \wedge \frac k n -  F_{x,n,r}(t) \wedge \frac k n \right\} dt.
\end{eqnarray*}
\end{Lemma}

\begin{figure}[H]
\begin{center}
	\includegraphics[scale=0.8]{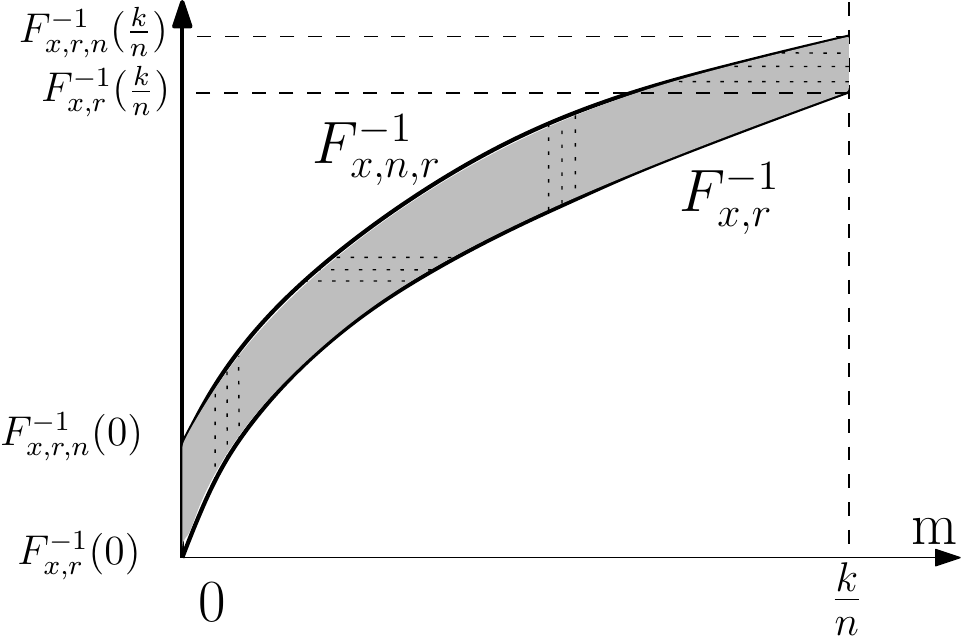}
\caption{Calculation of $ \Delta_{n,\frac kn,r}(x)$ by integrating the grey domain horizontally or vertically.}
\label{fig:LemmeJerome}
\end{center}
\end{figure}

\subsection{Proof of Theorem~\ref{Theo:ExpoDTMBounded}} 

We recall that we use the notation $F$  for $F_{x,r} $ and $F_n$ for  $F_{x,n,r} $ in the
proof.  

\subsubsection*{Upper bound on the fluctuations of $\Delta_{n,\frac kn}(x)$}

We first check that  $P\left(  | \Delta_{n,\frac kn}(x) | \geq \lambda \right) = 0$ for $\lambda \geq \omega_{x}(1)$.
Note that $\omega_{x}(1) < \infty$ because the support of $P$ is compact. Let $\mathbb G_n$ and $\mathbb G^{-1}_n$ be
the empirical uniform distribution function and the empirical uniform quantile function (see Appendix~\ref{sec:Gn}).
Starting from the definition \eqref{eq:DeltaW1} of the DTM and using Proposition~\ref{Prop:ReducLaw} in Appendix~\ref{sec:Gn}, we obtain that for $ \lambda \geq 0 $ and $k  \leq n$:
\begin{eqnarray*}  
P\left(  | \Delta_{n,\frac kn}(x) | \geq \lambda \right)   &\leq & P \left(   \sup_{u \in
[0,\frac k n] }  \left|  F^{-1} \left(\mathbb G^{-1}_n (u)   \right)   -  
F ^{-1}(u)\right|     \geq \lambda \right)   \\
&\leq & P \left(   \sup_{u \in [0,\frac k n] }   \omega_{x}  \left( \left|  \mathbb G^{-1}_n (u)    
  -   u \right|  \right)    \geq  \lambda  \right) \\
&\leq & P \left(   \sup_{u \in [0,\frac k n] }  \left|  \mathbb G^{-1}_n (u)    
  -   u \right|     \geq  \omega_{x}^{-1}(\lambda) \right),
\end{eqnarray*}
and this probability is obviously zero for any $\lambda \geq \omega_{x}(1)$. 

We now prove the deviation bounds starting from Lemma~\ref{DecomposDeltaJerome}. If
$F^{-1}(\frac k n )  \leq F_n^{-1}(\frac k n )$, then 
\begin{equation*}
\Delta_{n,\frac kn}(x)  =  \frac n k \int_{F^{-1}(0)}^{F^{-1}(\frac k n)} \left\{
F (t)   -  F_n(t)   \right\} dt 
+ \frac n k \int_{F^{-1}(\frac k n)}^{F_n^{-1}(\frac k n)} \left\{ \frac k
n   -  F_n(t)   \right\} dt  
\end{equation*}
and thus 
\begin{eqnarray}
\left|\Delta_{n,\frac kn}(x)\right|& \leq &   
 \underbrace{
 \frac n k
\int_{F^{-1}(0)}^{F^{-1}(\frac k n)}  \left|F (t)   -  F_n(t) \right| dt
}_{A}
+  
\underbrace{
\frac n k   \int_{F^{-1}(\frac k n)}^{F_n^{-1}(\frac k n)} \left\{ \frac kn   -  F_n(t)   \right\} dt   \:  \indic{F_n^{-1}\left(\frac k n\right) \geq   F^{-1} \left(\frac k n \right) } 
}_{B} 
\label{decompExpoBound}
\end{eqnarray}
If $F^{-1}(\frac k n )  > F_n^{-1}(\frac k n )$, then
\begin{equation*}
\Delta_{n,\frac kn}(x)  =  \frac n k \int_{F^{-1}(0)}^{F_n^{-1}(\frac k n)} \left\{
F (t)   -  F_n(t)   \right\} dt 
+ \frac n k \int_{F_n^{-1}(\frac k n)}^{F^{-1}(\frac k n)} \left\{  F(t) - \frac
k
n     \right\} dt
\end{equation*}
and thus 
\begin{eqnarray*}
\left|\Delta_{n,\frac kn}(x)\right| &\leq & \frac n k
\int_{F^{-1}(0)}^{F_n^{-1}(\frac k n)}  \left|F (t)   -  F_n(t) \right| dt 
+ \frac n k \int_{F_n^{-1}(\frac k n)}^{F^{-1}(\frac k n)} \left\{  \frac
k n  -  F(t)   \right\} dt   \\
 &\leq & \frac n k \int_{F^{-1}(0)}^{F_n^{-1}(\frac k n)}  \left|F (t)   - 
F_n(t) \right| dt 
+ \frac n k \int_{F_n^{-1}(\frac k n)}^{F^{-1}(\frac k n)}  \left\{  F_n(t) - F(t) \right\} dt  \\
&\leq & \frac n k \int_{F^{-1}(0)}^{F^{-1}(\frac k n)}  \left|F (t)   - 
F_n(t) \right| dt .
\end{eqnarray*}
In all cases, Inequality (\ref{decompExpoBound})  is thus satisfied.

\paragraph{$\bullet$ Local analysis : deviation bound of $\Delta_{n,\frac kn}(x)$ for $\frac kn $ close to zero.}
We now prove the deviation bound for $\frac kn  < \frac 12$. We  first upper bound the term A in
(\ref{decompExpoBound}).
According to
Proposition~\ref{prop:devEmpProcunif} in Appendix~\ref{sec:Gn}, for any $u_0 \in (0,\frac 1 2)$
and any $\lambda >0$:
\begin{equation} \label{Gndev}
P \left( \sup_{u \in [0,u_0]}  \left|\mathbb G_n (u)  - u  
 \right| \geq  \lambda  \right) \leq 2 \exp \left(  - \frac{n \lambda^2
(1-u_0)^2}{2u_0} \frac 1 {1 + \frac{(1-u_0) \lambda} {3u_0}} \right).
\end{equation}
For $u_0 \leq    \frac 12  $  and $\lambda >0$ it yields
\begin{eqnarray*}
P \left( \sup_{t \in \left[ F^{-1}(0), F^{-1}(u_0) \right)}  \left| F_n(t) -  F(t)  \right| \geq  \lambda  \right)
& =& P \left( \sup_{u \in [ 0,  u_0 )}  \left| \mathbb G_n (u) -  u  \right| \geq  \lambda  \right) \\
  & \leq & 2 \exp \left(  - \frac{n \lambda^2 (1-u_0)^2}{2u_0} \frac 1 {1 + \frac{(1-u_0) \lambda} {3u_0}} \right)
\\ 
& \leq & 2 \exp \left(  - \frac{n \lambda^2 (1-u_0)^2}{4u_0} \right)  + 2 \exp
\left(  - \frac{3 n \lambda (1-u_0)}{4}\right) ,
\end{eqnarray*}
where we have used Proposition~\ref{Prop:ReducLaw} in Appendix~\ref{sec:Gn} for the first equality,  \eqref{Gndev}   for the second inequality, and that for any $u,v >0$, $\exp(- u/(1 +v) \leq \exp(- u/2) + \exp(-u/(2v))$.
The term $A$ can be upper bounded by controlling the supremum of $\left| F_n   -  F  \right|$ over  $  \left[ F^{-1}(0)
,F^{-1} \left( \frac kn\right)\right)$.
If $ \frac k n  < \frac 12 $, it yields
\begin{multline}
P \left( A \geq  \lambda  \right)  
 \leq 
P \left(  
\frac n k \left[ F^{-1}\left(\frac k n\right) - F^{-1}(0) \right]  
\sup_{t \in  \left[F^{-1}(0), F^{-1}\left( \frac kn \right) \right) }  \left|
F_n(t) -  F(t)  \right| \geq  \lambda  \right)    \\
\leq  
2 \exp \left( - \frac{n  \lambda^2 (1-\frac k n)^2}{4  \frac k n}  \left[
\frac{\frac k n}{ 
F^{-1}\left(\frac k n\right) - F^{-1}(0)} \right]^{2}    \right) + 2 \exp \left(
 - \frac {3 n\lambda  (1- \frac k n )}  4   \left[
\frac{\frac k n}{ F^{-1}\left(\frac k n\right) - F^{-1}(0)} \right]\right)  
\\
\leq 2 \exp \left( - \frac n{16} \frac kn  \left[ \frac {\lambda} {F^{-1}\left(\frac k n\right) - F^{-1}(0)} 
\right]^{2} 
\right) 
+ 2 \exp \left(  - \frac {3 n}8 \frac kn   \frac{\lambda }  { F^{-1}\left(\frac k n\right) - F^{-1}(0)} 
\right) \label{InegA}
\end{multline}

\medskip

We now upper bound $B$.  We have 
\begin{equation}
\label{ub:restart}
 B    \leq       \frac n k      \left[ F_n^{-1}\left(\frac k n\right)  - F^{-1}\left(\frac k n\right) 
\right]  
    \left[ \frac k n   -   F_n \left( F^{-1}\left(\frac k n \right)   \right)   \right] 
 \: \indic{F_n^{-1}\left(\frac k n\right) \geq   F^{-1} \left(\frac k n \right) } .
\end{equation}
Thus, according to Proposition~\ref{Prop:ReducLaw} in Appendix~\ref{sec:Gn},
\begin{eqnarray*}
P( B \geq \lambda)  &\leq & P\left(  \frac n k    \left\{     \left[ F^{-1}  \left( \mathbb G_n^{-1} 
\left(\frac k n \right) \right) - F^{-1}\left(\frac k n\right)  \right] \right\}  \, 
 \left\{        \left[ \frac k n   -   \mathbb G_n \left(\frac k n \right)   \right] 
\right\} 
 \: \indic{\mathbb G_n^{-1} \left(\frac k n\right) \geq    \frac k n   }  \geq \lambda \right) \\
 &\leq  &  P( B_0   \geq \lambda)
\end{eqnarray*}
where 
\begin{equation*}
B_0 := \left\{  \sqrt{ \frac n k }     \omega_x\left( \, \mathbb G_n^{-1}  \left(\frac k n \right) -
\frac k n
\right)  \right\}  \, 
 \left\{  \sqrt{ \frac n k }     \left[ \frac k n   -   \mathbb G_n \left(\frac k n \right)   \right] 
\right\} .
 \end{equation*}
Let $\theta \in (0,1)$ to be chosen
further. We have
\begin{eqnarray*}
 2 B_0  \  & \leq &
  \underbrace{  \left\{    \theta   \sqrt{ \frac n k }      \omega_x\left( \, \mathbb G_n^{-1} 
\left(\frac k n \right) - \frac k n
\right)  \right\} ^2   }_{B_1^2}
\,  +  \, 
 \underbrace{ \left\{\frac 1   \theta  \sqrt{ \frac n k }    \left[ \frac k n   -   \mathbb G_n
\left(\frac k n \right)   \right]  \right\} ^2}_{B_2^2} \\
\end{eqnarray*}
Then we can write
\begin{eqnarray}
 P (B \geq \lambda) & \leq & P \left( B_1 \geq \sqrt \lambda  \right) + P \left( B_2 \geq \sqrt \lambda  \right)  \notag
\\
&\leq &   
 P \left(   \left|  \mathbb G_n^{-1} \left(\frac k n \right)  - \frac k n  \right| \geq \omega_x^{-1} \left(  \frac{
\sqrt \lambda}{\theta}   \sqrt{ \frac kn}  \right)  \right) 
+ P \left(   \left|  \mathbb G_n  \left( \frac k n \right)  - \frac k n  \right| \geq   \theta \sqrt \lambda    \sqrt{
\frac kn}  \right)  \notag .
\end{eqnarray}
Thus,
\begin{eqnarray}
 P (B \geq \lambda) &\leq & 2 \exp \left(  - \frac{   n  \left\{ \omega_x^{-1} \left(  \frac {\sqrt \lambda}{\theta}  
 \sqrt{ \frac kn} \right) \right\} ^2 }{4 k / n}
\right) + 2   \exp \left(  -  \frac{   3 n   \omega_x^{-1} \left(  \frac {\sqrt \lambda}{\theta}    \sqrt{ \frac kn} 
\right)   }{8}
\right)  \notag  \\
&  & +  2   \exp \left(  -  \frac{ n^2  \theta ^2\left[ \frac k n \right]     \lambda   }{ 4 k}  \right)  
+  2   \exp \left(  -  \frac {3 \theta n}4     \sqrt{ \frac kn}  \sqrt \lambda  \right)  \label{InegB}
\end{eqnarray}
where we have used Propositions ~\ref{prop:devEmpProcpointwise} and \ref{prop:devquantilepointwise}. According to
(\ref{decompExpoBound}), we have $ P \left( \left|\Delta_{n,\frac kn}(x)\right| \geq  \lambda \right) \leq P (A \geq
\frac \lambda 2) + P (B \geq \frac \lambda 2) $. We then obtain the following deviation bound  from  Inequalities
(\ref{InegA}) and (\ref{InegB}) for any $\frac kn < \frac 12$ and any $\lambda >0$:
\begin{multline}
\frac { P\left( | \Delta_{n,\frac kn}(x) | \geq  \lambda \right)}{2}    \leq 
  \exp \left( - \frac{n} {64   \left[ F^{-1}\left(\frac k n\right) - F^{-1}(0) \right]^{2} }\frac kn   \lambda^2
\right)  
+   \exp \left(  - \frac {3 n}{16}  \frac kn \frac{\lambda }   { F^{-1}\left(\frac k n\right) - F^{-1}(0)}  \right)  \\
+ \exp \left(  - \frac{   n^2 }{4 k} \left\{ \omega_x^{-1} \left(  \frac 1 \theta  {\sqrt \frac \lambda 8}  \sqrt{ \frac
kn}  \right) \right\} ^2  \right)  
 +      \exp \left(  -  \frac{   3 n  }{8} \omega_x^{-1} \left(  \frac 1 \theta  {\sqrt \frac \lambda 2}     \sqrt{
\frac kn}\right)   
\right)   \\
 +     \exp \left(  -  \frac{ n   \theta ^2  \lambda  }{8}       \right)  
+     \exp \left(  -  \frac {3 \theta n  }4   \sqrt { \frac \lambda 2}    \sqrt{ \frac kn}   \right)  \label{eq:devzero} ,
\end{multline}
where $\theta$ will be chosen further in the proof.

\paragraph{$\bullet$ Deviation bound of $\Delta_{n,\frac kn}(x)$ for $\frac kn \geq \frac 12$.} For controlling $A$, we
now use the DKW Inequality (see Theorem~\ref{DKW}), it gives that
\begin{eqnarray*}
P \left( A \geq  \lambda  \right)  
& \leq &
P \left(  
\frac n k \left[ F^{-1}\left(\frac k n\right) - F^{-1}(0) \right]  
\sup_{t \in  [0,1] }  \left|
F_n(t) -  F(t)  \right| \geq  \lambda  \right)    \\
&\leq  &
2 \exp \left( - 2  n  \lambda^2  \left[
\frac{\frac k n}{ 
F^{-1}\left(\frac k n\right) - F^{-1}(0)} \right]^{2}    \right) 
\end{eqnarray*}
We decompose $B$ into $B_1$ and $B_2$ as before. We use DKW again for $B_1$ and $B_2$. For the quantile term $B_1$, note
that 
$$ \left\{ \sup_{u \in [0,\frac kn]} \left| \mathbb G_n ^{-1} (u) - u  \right|  > \lambda   \right\}
  \subset \left\{ \sup_{u \in [0,1]} \left| \mathbb G_n ^{-1} (u) - u  \right|  > \lambda \right\} 
   =  \left\{ \sup_{t \in [0,1] } \left| \mathbb G_n (t) - t  \right|  > \lambda \right\}. $$
We find that for any $\tilde \theta>0$ :
\begin{multline}
\frac { P\left( | \Delta_{n,\frac kn}(x) | \geq  \lambda \right)}{2}    \leq   
 \exp \left( - 2  n  \lambda^2  \left[ \frac{\frac k n}{  F^{-1}\left(\frac k n\right) - F^{-1}(0)} \right]^{2}   
\right) 
+ \exp \left(  - 2 n \left\{ \omega_x^{-1} \left(  \frac 1 {\tilde \theta}  {\sqrt \frac \lambda 2} \sqrt{ \frac
 kn}   \right) \right\} ^2  \right)    \\
  +     \exp \left(  -  n   \tilde \theta ^2  \lambda      \frac k n    \right)  \label{deviationNonLocal}  .
\end{multline}
where $\tilde \theta$ will be chosen further in the proof.

%
\subsubsection*{Upper bound on the expectation of $\Delta_{n,\frac kn}(x)$}

\paragraph{$\bullet$ Case $\frac kn \leq \frac 12$.}
By integrating the probability in \eqref{eq:devzero}, we obtain
\begin{multline}
 \label{eq:expectDelta}
\frac{\E\left| | \Delta_{n,\frac kn}(x) \right|} 2  \leq
16 \sqrt {\pi}  \frac 1{\sqrt n} \left(\frac kn \right)^{-\frac 12} \left[ F^{-1}\left(\frac k n\right)
 - F^{-1}(0)\right]
+ \frac{16}{3n}  \left(\frac kn \right)^{-1} \left[ F^{-1}\left(\frac k n\right) - F^{-1}(0) \right]   \\
 +   \underbrace{\int_{ \lambda >0} \exp \left(  - \frac{   n^2 }{4 k} \left\{ \omega_x^{-1} \left(  \frac 1 \theta 
{\sqrt \frac \lambda 2}   \sqrt{ \frac kn}  \right) \right\} ^2  \right)    \, d \lambda }_{I_1}
 +    \underbrace{ \int_{ \lambda >0}   \exp \left(  -  \frac{   3 n  }{8} \omega_x^{-1} \left(  \frac 1{\theta}  \sqrt{
\frac \lambda 2}    \sqrt{ \frac kn} \right) \right)  \, d \lambda    }_{I_2}   \\
 +    \frac {8 }  {\theta^2 n }     
 + \frac {32}9  \frac 1 {  \theta ^2  n^2 }    \sqrt{ \frac nk} 
 \end{multline}
Since $\omega_x(u) /u$ is a non increasing function, we have that $\omega_x^{-1}(t) /t$ is a non decreasing function.
Then, for any positive constants $\lambda_1$ and $\lambda_2$:
\begin{multline*}
I_1 + I_2 \leq \lambda_1  + \int_{ \lambda > \lambda_1} \exp \left(  - \frac{   n^2 }{4 k} \left\{ \frac 1{\lambda_1}
\omega_x^{-1} \left( \frac 1 \theta 
{\sqrt \frac {\lambda_1} 2}   \sqrt{ \frac kn}  \right) \right\} ^2  \lambda^2 \right)    \, d \lambda  \\
  + \lambda_2 +    \int_{ \lambda >\lambda_2}   \exp \left(  -  \frac{   3 n  }{8 \lambda_2} \omega_x^{-1} \left(  \frac
1{\theta} \sqrt{
\frac {\lambda_2} 2}   \sqrt{ \frac kn}  \right) \lambda  \right)  \, d \lambda    \\
= \left[\lambda_1  + \frac{2 \sqrt k }n \frac{\lambda_1}{\omega_x^{-1} \left( \frac 1 \theta {\sqrt \frac {\lambda_1} 2}
 \sqrt{ \frac kn}  \right)} \right] + \left[ \lambda_2 +  \frac{8 \lambda_2}{   3 n    \omega_x^{-1}
\left(  \frac 1{\theta} \sqrt{ \frac {\lambda_2} 2}   \sqrt{ \frac kn}  \right)} \right]
\end{multline*}
We then take $ \lambda_1 =  2 \left\{ \theta    \sqrt{ \frac nk}  \  \omega_x 
  \left( \frac{2\sqrt{k}}{n}  \right)  \right\} ^2$ and $ \lambda_2 = 2 \left\{\theta     \sqrt{ \frac nk}   \omega_x  
\left( \frac{8}{3n}  \right)  \right\} ^2 $ and   we find that
$$ I_1  + I_2 \leq  4   \left\{  \theta    \sqrt{ \frac nk}  \omega_x 
\left(
\frac{2\sqrt{k}}{n}  \right)     \right\}    ^2  +   4   \left\{  \theta     \sqrt{ \frac nk}\omega_x   \left(
\frac{8}{3n}  \right)     \right\}    ^2 . $$
We then choose 
$$\theta ^2  =  \frac 1 {\sqrt n   \omega_x  \left( \frac{2\sqrt{k}}{n}  \right)  } 
 \sqrt{ \frac kn}    $$ 
to balance the terms $I_1$  and  $\frac 8 {\theta^2 n } $ in~\eqref{eq:expectDelta}. The deviation bound
given in the theorem corresponds to this choice for $\theta$. 

Finally,  note that  $ \omega_x  \left( \frac{\sqrt{2k}}{n}  \right) \leq \sqrt 2 \omega_x   \left( \frac{\sqrt{k}}{n}  \right)$ because $\omega_x(u) /u$ is a non increasing function and we obtain that
there exists an absolute constant $C$ such that  
\begin{equation} \label{ExpectationLocal} 
\E\left|  \Delta_{n,\frac kn}(x) \right|  \leq
 \frac C{\sqrt n} \left[ \frac k n \right] ^  {- 1/2}    \left\{  \left[ F^{-1}\left(\frac k n\right)
 - F^{-1}(0)\right]
 +       \omega_x  \left( \frac{\sqrt{k}}{n}  \right)  \right\}. 
\end{equation}

\paragraph{$\bullet$ Case $\frac kn \geq \frac 12$.}
We integrate the deviations \eqref{deviationNonLocal}  and  we obtain that 
\begin{equation} \label{ExpectationNonLocal} 
 \E\left|  \Delta_{n,\frac kn}(x) \right|   \leq C \left[ 
  \frac 1{\sqrt n}  \left[ \frac{  F^{-1}\left(\frac k n\right) - F^{-1}(0)}{\frac k n} \right] 
 + \left\{  \tilde \theta   \sqrt{ \frac nk}  \omega_x  \left( \frac 1{\sqrt{n}} \right)      \right\}    ^2
 +    \frac {1}  {\tilde \theta^2 n }    \frac nk   \right] .
\end{equation}
We then choose 
$$\tilde \theta ^2  =  \frac 1 {\sqrt n   \omega_x  \left( \frac 1{\sqrt{n}} \right)  } .$$ 
The deviation bound given in the Theorem correspond to this choice for $\tilde \theta$. Since  $ \sqrt{\frac nk} \leq \sqrt 2  \leq 2 $, we see that the expectation bound \eqref{ExpectationNonLocal} for this choice of $\tilde \theta$ can
be rewritten as the expectation bound \eqref{ExpectationLocal}. This concludes the proof of
Theorem~\ref{Theo:ExpoDTMBounded}.

\subsection{Proof of Proposition \ref{prop:lowerboundDTM}}

\begin{figure}[H]
\begin{center}
	\includegraphics[scale=0.8]{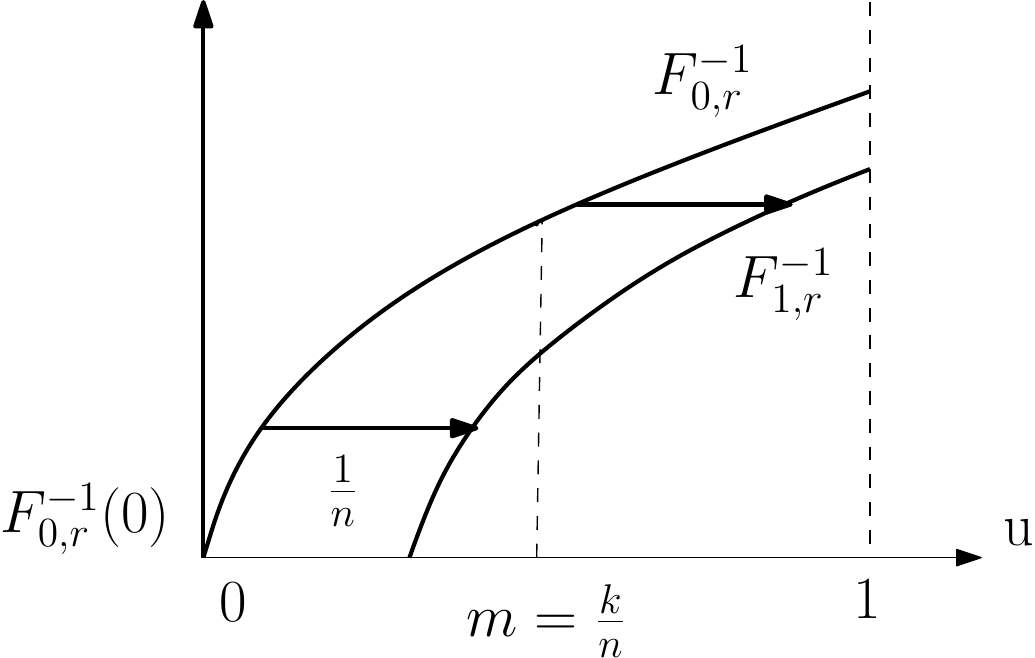}
\caption{The two quantile functions $F_{0}^{-1}$ and $F_{1}^{-1}$.}
\label{fig:quantileLowerbound}
\end{center}
\end{figure}

We first consider the case $d=1$. For applying Le Cam's Lemma (Lemma \ref{Lem:Lecam}), we need to find two probabilities
$P_0$ and $P_1$ which distances
to measure are sufficiently far from each other. 
Without loss of generality we can assume that  $x = 0$. Let  $\bar P  \in \mathcal P _\omega$ which satisfies 
\eqref{OmegaEstUnQuantile}. We can assume that $\bar P$ is supported on $\R^+$ since the push forward measure of $\bar
P$ by the norm is in $ \mathcal P _\omega$ and also satisfies  \eqref{OmegaEstUnQuantile}.
Let $\bar F^{-1}$ be the quantile function of $\bar P$. For some $n \geq 1$, let $P_{0} := \bar P$ and let $P_{1} :=
\frac 1 n \delta_{0} +  \left. \bar P \right|_{[0,\bar
F^{-1}(1- 1/n)]} $, where $\delta_{0}$ is a Dirac distribution at zero and where $ \left. \bar P \right|_{[a,b]}$ is the
restriction of the measure $\bar P$ to the set $[a,b]$. For $i = 0,1$, let  $ P_{i,r}$  be
the push-forward measure of $P_{i}$  by the power function $t \mapsto t^r$ on $\R^+$. Let also
$F_{i,r}$ and $F_{i,r}^{-1}$ be the distribution function and the quantile function  of $ P_{i,r}$, see
Figure~\ref{fig:quantileLowerbound} for an illustration. Note that
 that $ P_{1,r}   = \frac 1 n \delta_{0} + \left.P_{0,r} \right|_{[0,F_{0,r}^{-1}(1-
1/n)]} $. Thus $P_{1}$ is in $\mathcal P _\omega$ because 
$$
F_{1,r}^{-1}(u)  = \left\{  
\begin{array}{ll}
      F_{0,r}^{-1}(0)  & \textrm{ if } u \leq  \frac 1n ,  \\
      F_{0,r}^{-1}(u-1/n) &   \textrm{ otherwise.}
 \end{array} \right.
$$ 
The probability measures  $P_{0} $  and $P_{1} $ are absolutely continuous with respect to the measure  $\mu := 
\delta_{0} + \bar P$. The density of $P_{0}$ with respect to $\mu$  is $p_{0} :=   \indic{(0,+\infty)}$ whereas
the density of $P_{1}$ with respect to $\mu$  is $p_{1} =  \frac 1n \indic{\{0\}} + \indic{(0,\bar F^{-1}(1-
1/n)]}$. Thus,
\begin{eqnarray*}
TV(P_{0},P_{1}) & = & \int_{\R^+} | p_{1}(t) -  p_{0}(t) | \; d  \mu(t)  \\
& = &  \frac 1 {n}  +     \bar P \left( \left( \bar F^{-1}\left(1-\frac1 n\right), \infty  \right)  \right) \\ 
& = &  \frac 2 {n}.
\end{eqnarray*}
The next, $ \left[ 1 - TV(P_{0},P_{1}) \right] ^{2 n }  = (1-\frac 2 n ) ^{2n}  \rightarrow e^{-4}$  as $n$ tends to
infinity. 
Moreover,
\begin{eqnarray*} 
\left|d_{P_{0},r}^r(x) -  d_{P_{1},r}^r(x)  \right| &=&  \frac n k \int_{0}^{\frac kn} \left\{  F_{0,r}^{-1} (u) - 
F_{1,r}^{-1} (u)   \right\} du \\
&\geq &     \frac n k  \int_{F_{1,r}^{-1}(u-1/n)}^{F_{1,r}^{-1} (\frac kn)} \left\{  F_{1,r} (t) -  F_{0,r}  (t)  
\right\} dt
\end{eqnarray*} 
according to basic  geometric calculations. Thus, 
\begin{eqnarray*} 
\left|d_{P_{0},r}^r(x) -  d_{P_{1},r}^r(x)  \right|
&\geq &     \frac n k   \frac 1n \left[ F_{1,r}^{-1} \left(\frac kn \right)  - F_{1,r}^{-1}(0) \right]\\
&\geq &     \frac n k   \frac 1n \left[ F_{0,r}^{-1} \left(\frac kn \right)  - F_{0,r}^{-1}(0) \right]\\
&\geq &     \frac n k   \frac c n  \omega \left(\frac {k} n \right) .
\end{eqnarray*}
where we have used Assumption\eqref{OmegaEstUnQuantile} for the last inequality. We conclude using  Le Cam's Lemma.

\medskip 
We now consider the case $d \geq 2$. Let $\bar P  \in \mathcal P _\omega$ which satisfies  \eqref{OmegaEstUnQuantile}.
By considering the push-forward measure of $\bar P$ by the function  
\begin{displaymath}
 \begin{array}{rcl}
    \mathbb R^d & \longrightarrow & \R^+ \times \{0\}^{d-1} \\
    y & \longmapsto &  ( \|y\|, 0 , \dots,0) \\
  \end{array} \, ,
\end{displaymath}
we see that it is aways possible to assume that there exist a probability $\bar P $  supported on $\R^+ \times \{0\}
^{d-1}$ which satisfies \eqref{OmegaEstUnQuantile}.  Now, it is then possible to define  $P_{0}$  and $P_{1}$ as in the
case $d=1$ except that their support is now in $\R^+ \times \{0\}^{d-1}$.  Following the same construction,  the
quantities $TV(P_{0},P_{1}) $ and $d_{P_{0},r}^r(x) -  d_{P_{1},r}^r(x) $ take the same values as in the case $d=1$. We
thus obtain the same lower bound as in the case $d=1$.

\subsection{Proof  of Theorem~\ref{Theo:ExpoDTMUnBounded}}

Inequality \eqref{decompExpoBound} in the proof of Theorem~\ref{Theo:ExpoDTMUnBounded} is still valid. We can also use the deviation bound \eqref{InegA} on $A$ for the case  $\frac kn \leq \frac 12$. Regarding the deviation bound on $B$,  we restart from Inequality~\eqref{ub:restart} and we note that
\begin{multline*} 
    \frac n k    \left\{     \left[ F^{-1}  \left( \mathbb G_n^{-1} 
\left(\frac k n \right) \right) - F^{-1}\left(\frac k n\right)  \right] \right\}  \, 
 \left\{        \left[ \frac k n   -   \mathbb G_n \left(\frac k n \right)   \right] 
\right\} 
 \: \indic{\mathbb G_n^{-1} \left(\frac k n\right) \geq    \frac k n   }   \\
 \leq       \frac n k    \left\{     \left[ F^{-1}  \left( \mathbb G_n^{-1} 
\left(\frac k n \right) \right) - F^{-1}\left(\frac k n\right)  \right] \right\}  \, 
 \left\{        \left[ \frac k n   -   \mathbb G_n \left(\frac k n \right)   \right] 
\right\} 
 \: \indic{   \frac k n  \leq \mathbb G_n^{-1} \left(\frac k n\right)  \leq \bar m  }   
 \\ 
  +   \frac n k    \left\{     \left[ F^{-1}  \left( \mathbb G_n^{-1} 
\left(\frac k n \right) \right) - F^{-1}\left(\frac k n\right)  \right] \right\}  \, 
 \left\{        \left[ \frac k n   -   \mathbb G_n \left(\frac k n \right)   \right] 
\right\} 
 \: \indic{  \mathbb G_n^{-1} \left(\frac k n\right)  >  \bar m  }     \\
  \leq     
\underbrace{  
    \frac n k    \left\{   \omega_x^{-1}  \left( \mathbb G_n^{-1} \left(\frac k n \right)
- \frac k n \right)   \right\}  \, 
 \left\{        \left[ \frac k n   -   \mathbb G_n \left(\frac k n \right)   \right] 
\right\}  
 \: \indic{  \mathbb G_n^{-1} \left(\frac k n\right)  \leq   \bar m  } 
}_{\tilde B}
    \\
  +  
 \underbrace{ \frac { n} k    \left\{      F^{-1}  \left( \mathbb G_n^{-1} \left(\frac k n \right)  \right) \right\}  \, 
 \left\{        \left[ \frac k n   -   \mathbb G_n \left(\frac k n \right)   \right] 
\right\} 
 \: \indic{  \mathbb G_n^{-1} \left(\frac k n\right)  >  \bar m  } 
 }_{B_3} 
\end{multline*}
By definition of $B$, $\tilde B $ and $B_3$, and using~Proposition~\ref{Prop:ReducLaw} in Appendix~\ref{sec:Gn}, we obtain that
\begin{eqnarray}
P\left( | \Delta_{n,\frac kn}(x) | \geq  \lambda \right) &\leq &P \left(  A \geq \frac \lambda 2 \right) + P \left( 
B \geq
\frac \lambda 2 \right) \notag \\
&\leq & P \left(  A \geq \frac \lambda 2 \right) + P \left(  \tilde B \geq \frac \lambda 2  \right)  +  P \left( B_3
\geq \frac \lambda 2  \right)  \label{ABB3}
\end{eqnarray}
where $ P \left(  A \geq \frac \lambda 2\right ) + P \left(  \tilde B \geq \frac \lambda 2  \right)$ has already been
upper bounded in the Proof of Theorem~\ref{Theo:ExpoDTMUnBounded}. We now upper bound the deviations of $B_3 $. For any
$\theta_3 \in (0,1)$ to be chosen further, we have:
$$
2 B_3   \leq   \underbrace{ \left\{  \theta_3   F^{-1}  \left( \mathbb G_n^{-1} \left(\frac k n \right)  \right) \right\} ^2 
 \:
\indic{  \mathbb G_n^{-1} \left(\frac k n\right)  >  \bar m  }
}_{B_4^2}
\,  +  \,  \underbrace{ 
\left\{ \frac 1 {\theta_3}  \frac n k \left[   \frac k n   -   \mathbb G_n \left(\frac k n \right)   \right]  \right\} ^2 
}_{B_5^2} .
$$
We have $P(B_3 \geq \frac \lambda 2) \leq  P \left(B_4 \geq \sqrt {\frac \lambda 2} \right)  + P \left(B_5 \geq \sqrt
{\frac\lambda 2} \right)$ 
where 
\begin{equation}
 P \left(B_5 \geq \sqrt {\frac \lambda 2} \right)  \leq  2   \exp \left(  -  \frac{ k  \theta_3 ^2 \lambda   }{ 8 } 
\right) 
+  2   \exp \left(  -  \frac {3 \theta_3 k}8       \sqrt{ \frac \lambda 2} \right)    \label{B5}.
\end{equation}
The probability $P \left(B_4 \geq \sqrt {\frac \lambda 4} \right)$ can be upper bounded in two different ways: one using a concentration argument et one based on the Beta distribution of $\mathbb G_n^{-1}$. According to
Proposition~\ref{prop:devquantilepointwise}, we have
\begin{eqnarray}
 P \left(B_4 \geq \sqrt {\frac \lambda 2} \right) 
&\leq&  P \left( \mathbb G_n^{-1} \left(\frac k n \right) - \frac k n \geq  F  \left( \sqrt{ \frac{ \lambda}{ 2}} \frac
1{\theta_3}\right) \vee \bar m  - \frac k n\right)  \label{B4} \notag \\
 & \leq  & 4 \exp \left[- \frac{n^2}{2k} \left(   \bar m - \frac kn \right)^2   \right] \label{B4V1} .
\end{eqnarray}
Next, it is well known \citep[see for instance p.97 in Chapter~3 of][]{shorack2009empirical} that $\mathbb G_n^{-1}\left(\frac kn\right)$ has a Beta$(k,n-k+1)$ distribution with
density on $[0,1]$:
$$t \mapsto \frac{n!}{(k-1)!(n-k)!} t^{k-1} (1-t)^{n-k} .$$
Thus, for any $ t \in (0,1)$, 
$
  P \left( \mathbb G_n^{-1}\left(\frac kn\right) \geq 1-t \right)  \leq   {n \choose k-1}  t ^{n-k+1}
$. Thus 
\begin{eqnarray}
 P \left( B_4 \geq \sqrt {\frac \lambda 2} \right) 
 & \leq &  P \left( B_4 >  \sqrt {\frac \lambda 3} \right)  \notag \\
&\leq&  P \left( \mathbb G_n^{-1} \left(\frac k n \right)  >  F  \left( \sqrt{ \frac{ \lambda}{3}} \frac 1{\theta_3}
\right) \vee \bar m  \right) \notag \\
 &\leq& {n \choose k-1}  \left[1 - F  \left( \sqrt{ \frac{ \lambda}{3}} \frac 1{\theta_3} \right) \vee \bar m  \right]
^{n-k+1} \label{B4V2} 
\end{eqnarray}
where the first inequality allows us to deal with  a strict comparaison, which is necessary to rewrite the probability in terms of the cdf. Note that a similar bound can be obtained using Bennett's inequality for $B_4$.

We now upper bound   $ \E|\Delta_{n,\frac kn}(x) | $. We only need to control the deviations of $B_3$. Since
$P$ has a moment of order $r$, for any $t >0$:
$$  t  \left(1- F(t) \right)  =   t P \left( {\| x-X \|}^r   > t \right)   \leq \E {\| x-X \|}^r = : C_{x,r}  .$$
Then for any $\bar \lambda >0$ (and $n$ larger than $3$):
\begin{eqnarray*}
 \int_0 ^{\infty}  P \left (B_4 \geq \sqrt {\frac \lambda 2} \right)  d \lambda  &\leq  &
   4 \int_0 ^{ \bar \lambda}  \exp \left[- \frac{n^2}{2k} \left(   \bar m - \frac kn \right)^2   \right]  d \lambda 
+  {n \choose k-1}   \int_{ \bar \lambda}^{\infty}  \left[1 - F  \left( \sqrt{ \frac{ \lambda}{ 3}} \frac 1{\theta_3} 
\right)  \right] ^{n-k+1}   d \lambda \\
&\leq  &
4  \bar \lambda   \exp \left[- \frac{n^2}{2k} \left(   \bar m - \frac kn \right)^2   \right] 
+   2^n  \left[ \sqrt 3 C_{x,r} \theta_3     \right] ^{n-k+1}
\int_{ \bar \lambda}^{\infty}    \lambda ^{\frac{-n+k-1}2}   d \lambda \\
&\leq  &
4  \bar \lambda   \exp \left[- \frac{n^2}{2k} \left(   \bar m - \frac kn \right)^2   \right] 
+   \frac{8 \bar \lambda  }{n}  \left[ \frac{ 4 \sqrt 3 C_{x,r} \theta_3}{ \sqrt {\bar \lambda}}  \right] ^{n-k+1} \\
&\leq  &
4 \bar \lambda  \left\{   \exp \left[- \frac{n^2}{2k} \left(   \bar m - \frac kn \right)^2   \right]   
  +  \exp  \left[ -(n-k+1) \log \left( \frac{ 4 \sqrt 3 C_{x,r} \theta_3}{ \sqrt {\bar \lambda}}  \right)\right]
\right\}
\end{eqnarray*}
We choose $\bar \lambda$ to balance the two terms inside the brackets: 
$$ \bar \lambda = \left\{ 4 \sqrt 3 C_{x,r} \theta_3  
\log \left[ \frac {n^2}{2k (n-k+1)} \left(\bar m- \frac kn\right) ^2 \right]
\right\}^2
$$
and then 
\begin{eqnarray*}
\int_0 ^{\infty}  P \left( B_3 \geq  \frac \lambda 2   \right)  d \lambda  &\leq  & 
\int_0 ^{\infty}  P \left(B_4 \geq \sqrt{\frac \lambda 2} \right)  d \lambda + \int_0 ^{\infty}    P \left(B_5 \geq
\sqrt{\frac \lambda 2} \right)  d \lambda \\
&\leq  &    C_{x,r,\bar m} \left\{ \frac {1}{k \theta_3 ^2} + \theta_3^2    \exp \left[- \frac{n^2}{2k} \left(   \bar m
- \frac kn \right)^2   \right] 
  \right\}
\end{eqnarray*}
where $C_{x,r,\bar m}$ only depends on $C_{x,r}$ and $\bar m$.
We thus take $\theta_3 ^2 =   \frac{\exp \left[ \frac{n^2}{4k} \left( \bar m - \frac kn \right)^2   \right]}{\sqrt k}
$ and we obtain that
$$ \int_0 ^{\infty}  P \left( B_3 \geq  \frac \lambda 2   \right)  d \lambda \leq  C_{x,r,\bar m} \sqrt k \exp \left[-
\frac{n^2}{4k} \left( \bar m - \frac kn \right)^2   \right] .$$
The deviation bound given in the Theorem derives from \eqref{ABB3}, \eqref{B5}, \eqref{B4V1} and \eqref{B4V2} with this value for $\theta_3$.

\subsection{Proof of Theorem~\ref{Theo:ExpoDTMBoundedUniform}}

We first recall the following Lemma  from \cite{chazal2011geometric}.
\begin{Lemma}[\cite{chazal2011geometric}]  \label{DTM-Lip}
For any $(x,y) \in (\R^d) ^2$ and any $m \in (0,1)$:
$$    |d_{P,m,r}(x) -  d_{P,m,r}(y) | \leq    \| x - y \|  .   $$
\end{Lemma} 
Next Lemma directly derives from Lemma~\ref{DTM-Lip}.
\begin{Lemma} \label{lem:DeltaLip} For $r=1$, the function $x \mapsto \Delta_{n,m,1}(x)$ is $1$- Lipschitz on $\R^d$.
For $r>1$, the function $x \mapsto \Delta_{n,m,r}(x)$ is $C_{\mathcal D,r}$ -Lipschitz on  the compact domain $\mathcal
D$ where $C_{\mathcal D,r}$  depends on $r$ and on the Hausdorff distance between $\cal D$ and the support of $P$.
\end{Lemma}

We give the proof of the Theorem for $r=1$. The calculations are also valid  $r >1$ by replacing $\lambda$ by $\lambda C_{\mathcal D,r}$ in  the probability bounds.  The deviation bound of the Theorem can be proved with a simple union
bound strategy. Up to enlarging the constant $c$, we can  write
$$N(\mathcal D,\lambda) \leq c \lambda ^{-\nu} \quad \textrm{for any } \lambda \leq \omega_{\cal D}(1). $$
Now, for a given $\lambda  \leq \omega_{\cal D}(1)$, there exists an integer $N \leq c \lambda^{-\nu}$ and  $N$ points
$(x_1,\dots,x_N)$ laying in $\mathcal D$ such that  $ \bigcup_{i = 1 \dots N} B(x_i,\lambda) \supseteq  \mathcal D$.
For any point $x \in \mathcal D$, there exists  a point  $\pi_\lambda(x)$ of $\{x_1,\dots,x_N\}$ such that $\| x-
\pi_\lambda(x)\|  \leq  \frac{\lambda}{ 2}$. According to Lemma~\ref{lem:DeltaLip}, we have
\begin{equation}
 | \Delta_{n,\frac kn,1}(x) - \Delta_{n,\frac kn,1} (\pi_\lambda(x)) |\leq  \frac \lambda 2 .
\label{eq:LipDelta}
 \end{equation}
 According to
Theorem~\ref{Theo:ExpoDTMBounded}, we have for any  $k <  \frac n 2$ and any $\lambda >0$:
\begin{equation}
 \label{UpperDiscr}
P\left(\sup_{i=1\dots N} | \Delta_{n,\frac kn,1}(x_i) | \geq  \frac{\lambda}2 \right)   \leq 
  \left\lbrace
\begin{array}{lll}
1 \wedge 2  c \lambda^{-\nu}  \square(\lambda)    & \textrm {if }  \lambda  \leq  \omega_{\cal D}(1), \\
 0  & \textrm {if }  \lambda  >  2\omega_{\cal D}(1),
\end{array}
\right.   
\end{equation}
where
\begin{multline*}
\square(\lambda) = 
\exp \left( - \frac 1 {64 }   \frac { k \lambda ^2 } {\left[ F_x^{-1}\left(\frac k n\right)
 - F_x^{-1}(0)\right]^{2}}    
\right) 
+ \exp \left(  - \frac 3{16} \frac  {k \lambda} { F_x^{-1}\left(\frac k n\right)
 - F_x^{-1}(0)}       \right) \\ 
+ \exp \left(  - \frac{   n^2}{4 k }  \left\{ 
\omega_{\mathcal D}^{-1}   \left(  k ^{1/4}    \sqrt{ \frac \lambda 8  \omega_{\mathcal D} \left( 2\frac{\sqrt{k}} n
\right)}     \right)  \right\} ^2  \right) 
+   \exp \left(  -  \frac{   3 n}8   \omega_{\mathcal D}^{-1}   \left(  k^{1/4}    \sqrt{ \frac \lambda 2
\omega_{\mathcal D}\left( \frac{2\sqrt{k}} n \right)}     \right) \right)     \\
 +   \exp \left(  -  \frac{\sqrt k}8     \frac  \lambda  { \omega_{\mathcal D} \left(\frac{2\sqrt{k}} n \right)}      \right)  
+   \exp \left(  -  \frac {3   k^{3/4}} 4     \sqrt{ \frac  \lambda  {2 \omega_{\mathcal D} \left(\frac{2 \sqrt{k}} n \right)}  } \right) \\
=: \square_1(\lambda) + \square_2(\lambda) +
\square_3(\lambda) + \square_4(\lambda) + \square_5(\lambda) + \square_6(\lambda).
\end{multline*}
Using (\ref{eq:LipDelta}) and
(\ref{UpperDiscr}), we find that $P\left( \sup_{x \in \mathcal D} | \Delta_{n,\frac kn} (x)| \geq  \lambda \right)$ is
also
upper bounded by the right hand term of (\ref{UpperDiscr}).

We now integrate each term in $ \lambda^{-\nu} \square(\lambda)$. For the first one, let $\alpha_{k,n} = \frac{1} {64}
 \frac k{\left[ F_x^{-1}\left(\frac k n\right)
 - F_x^{-1}(0)\right]^{2}}$, then for any $\lambda_{k,n}>0$:
\begin{eqnarray*}
\int_{0 } ^{\infty}  1 \wedge 2c \lambda^{-\nu}  \square_1(\lambda)    d \lambda   
& \leq  &  \lambda_{k,n} +  2c\lambda_{k,n}  ^{-\nu}  \int_{  \lambda_{k,n} } ^{\infty}     \exp \left( -
\alpha_{k,n} \lambda^2 \right) d \lambda   \\
& \leq  & \lambda_{k,n}+ 2c\frac { \lambda_{k,n}  ^{-\nu-1}}{  \alpha_{k,n}}  \exp(-\alpha_{k,n} \lambda_{k,n} ^2) .
\end{eqnarray*}
We balance these two terms by taking $\lambda_n =  \sqrt{ \frac { \log^+  \left(\left[ \alpha_{k,n} \right]^{\nu +5 } 
\right) }{ 2\alpha_{k,n}}}$, it gives:
\begin{equation}
\label{carre1}
\int_{0 } ^{\infty}  1 \wedge 2c \lambda^{-\nu}  \square_1(\lambda)    d \lambda  \lesssim \sqrt{ \frac { \log^+ 
\left(\left[ \alpha_{k,n} \right]^{ +5 }  \right) }{ \alpha_{k,n}}} .
\end{equation}
The upper bound for the second term can be obtained in the same way. For the third term, let 
$\beta_{k,n} :=  k^{1/4}    \sqrt{ \frac 18  \omega_{\mathcal D} \left(
\frac{2\sqrt{k}} n
\right)}  
$. Since $\omega_{\mathcal D}^{-1} (t) / t$ is non
decreasing, for any $\lambda_{k,n}>0$:
\begin{eqnarray*}
\int_{0 } ^{\infty}  1 \wedge 2c \lambda^{-\nu}  \square_3(\lambda)    d \lambda   
& \leq  &  \lambda_{k,n} +  2c\lambda_{k,n}  ^{-\nu}  \int_{  \lambda_{k,n} } ^{\infty}     \exp \left(  - \frac{  
n^2  \left\{ \omega_{\mathcal D}^{-1}   \left(\beta_{k,n}  \sqrt \lambda  \right) 
\right\} ^2 }{4 k} \right)  d \lambda   \\
& \leq  & \lambda_{k,n} +  2c\lambda_{k,n}  ^{-\nu}   \int_{   \lambda_{k,n} } ^{\infty}     \exp \left(  - \frac{  n^2
}{4 k} \left\{   \frac{ \omega_{\mathcal D}^{-1}   \left(\beta_{k,n}  \sqrt{ \lambda_{k,n}} \right)}{\sqrt{
\lambda_{k,n}}}   \right\} ^2   \lambda  \right)  d \lambda   \\
& \leq  & \lambda_{k,n}+ 4c \frac{k}{n^2} \frac { \lambda_{k,n}  ^{-\nu}}{     
 \left\{     \omega_{\mathcal D}^{-1}   \left(\beta_{k,n}  \sqrt{ \lambda_{k,n}} \right) \right\} ^2}  \exp \left( -
\frac{  n^2 }{4k} \left\{    \omega_{\mathcal D}^{-1}   \left(\beta_{k,n}  \sqrt{ \lambda_{k,n}} \right) 
\right\} ^2  \right) .
\end{eqnarray*}
We balance the two terms in the upper bounds by taking 
$$\lambda_{k,n} =  \left\{
	\frac{1}{ \beta_{k,n}}  
	\omega_{\mathcal D}  \left( 
		\frac{  2\sqrt{k}}{n}  
		\sqrt{   \log^+ \left(  \left[ \frac{\beta_{k,n}} { \omega_{\mathcal D} \left(  \frac{ 2\sqrt { k} } n
\right)} \right] ^ {2 (\nu-1)}   \right)  }
	    \right)   \right\}^2 .$$
Indeed, we then obtain that:
\begin{multline*}
\int_{0 } ^{\infty}  1 \wedge 2c \lambda^{-\nu}  \square_3(\lambda)    d \lambda   
 \lesssim   \lambda_{k,n}   +  \\
    \frac{ \lambda_{k,n}  }{  \log^+ \left(  \left[ \frac{\beta_{k,n}} { \omega_{\mathcal D} \left(  \frac{ 2\sqrt {k} }
n  \right)} \right] ^ {2 (\nu-1)}   \right)  }  
\left\{ 	\frac{1}{ \beta_{k,n}}   \omega_{\mathcal D}  \left(  \frac{ 2\sqrt{ k}}{n}   \sqrt{  
\log^+ \left(  \left[ \frac{\beta_{k,n}} { \omega_{\mathcal D} \left(  \frac{ 2\sqrt {k} } n \right)} \right] ^ {2
(\nu-1)}   \right)  }    \right)   \right\}^{-2 (\nu-1)}  \left[ \frac{ \omega_{\mathcal D} \left(  \frac{ 2\sqrt {k} }
n \right)} {\beta_{k,n}} \right] ^ {2 (\nu-1)}      \\
 \lesssim   \lambda_{k,n}     +   
    \lambda_{k,n}    
\left\{   \omega_{\mathcal D}  \left(  \frac{ 2\sqrt{ k}}{n}  \sqrt{   \log^+ \left(  \left[ \frac{\beta_{k,n}} {
\omega_{\mathcal D} \left(  \frac{ 2\sqrt {k} } n \right)} \right] ^ {2 (\nu-1)}   \right)  }    \right)   \right\}^{-2 (\nu-1)}  \times \\
 \left[  \omega_{\mathcal D} \left(  \frac{ 2\sqrt {k} } n   \sqrt{   \log^+ \left(  \left[ \frac{\beta_{k,n}}
{ \omega_{\mathcal D} \left(  \frac{ 2\sqrt {k} } n \right)} \right] ^ {2 (\nu-1)}   \right)  }  \right) \right] ^ {2
(\nu-1)}   \\
 \lesssim   \lambda_{k,n}  
 \end{multline*}
 where we have used $\log^+ \geq 1 $ and the fact that $\omega_{\mathcal D}(u)   $ is non decreasing  for the second inequality.  Since
$\omega_{\mathcal D}(u) / u $ is non increasing and $\log^+ \geq 1 $, we find that
  \begin{multline}
  \lambda_{k,n}  
  \lesssim k ^{-1/2}  \left[  \omega_{\mathcal D} \left( \frac{\sqrt{k}} n \right) \right]^{-1}  
\left\{	 \omega_{\mathcal D}  \left( \frac{ \sqrt{ k}}{n}  \right)
  \sqrt{   \log^+ \left( \left[     \frac{ \sqrt k} { \omega_{\mathcal D} \left(  \frac{ \sqrt {k} } n \right)}  
\right] ^ { \nu-1}   \right)   }  \right\} ^2   \\
   \lesssim \frac{   \omega_{\mathcal D} \left( \frac{\sqrt{k}} n \right) }{ \sqrt k} 
  \log^+ \left( \left[     \frac{ \sqrt k} { \omega_{\mathcal D} \left(  \frac{ \sqrt {k} } n \right)}   \right] ^ {
\nu-1}   \right) \label{carre3}  .
\end{multline}
We proceed in the same way to show  that  the upper bound \eqref{carre3} is also valid for $\square_5$  for $\square_6$.
The   bound in expectation given in the Theorem is of the order of the sum of the upper bounds  \eqref{carre1} and
\eqref{carre3}.

\section{Uniform empirical and quantile processes} \label{sec:Gn}

This section brings together known exponential inequalities for  the uniform empirical process and of the uniform quantile process. These results can be found for instance in Chapter 11 of~\cite{shorack2009empirical}. 

Let $\xi_1, \dots,\xi_n$ be $n$ i.i.id. uniform random variables.  The uniform
empirical distribution function is defined by
$$ \mathbb G_n (t) = \frac 1n \sum_{i=1}^n \mathbf{1}_{\xi_i \leq t} \quad \textrm{for }  0 \leq t \leq 1.$$
The inverse uniform empirical distribution function is the function
$$ \mathbb G^{-1}_n (u) =  \inf \{ t \, | \, \mathbb G _n (t) > u \} \quad \textrm{for }  0 \leq u \leq 1.$$
\begin{Prop} \label{Prop:ReducLaw}
For any $x \in \R^d$ and any $n \in \N^*$:
\begin{equation*}
\label{eq:ReducF}
F_{x,n} - F_x     \stackrel{\mathcal D}{=}  \mathbb G_n \left( 
F_x\right) -  F_x 
\end{equation*}
and 
\begin{equation*}
\label{eq:ReducQ}
F_{x,n}^{-1} - F_x ^{-1}    \stackrel{\mathcal D}{=} F_x^{-1}
\left(\mathbb G^{-1}_n   \right) -  F_x^{-1}  ,
\end{equation*}
where $F_x$ and $F_{x,n}$ are defined in the Introduction Section.
\end{Prop} 
 
\subsection{Exponential inequalities for the uniform empirical process}

Let the function $\Phi$ defined on $\R$ by 
\begin{equation*}
\Phi(\lambda) := 
\left\lbrace
\begin{array}{lll}
 \frac {2 (\lambda + 1) \left[\log(1+ \lambda) - 1 \right]} {\lambda^2} &  \textrm {if }\lambda > -1, \\
 + \infty &  \textrm {otherwise.} 
\end{array}
\right.
\end{equation*}
Next result is a point-wise exponential inequality for the deviations of the uniform empirical process $\left(\sqrt   n \left[\mathbb G_n
(t) - t \right]\right)_{t \geq 0}$. 
\begin{Prop}[Inequality~1 and Proposition~1  in \cite{shorack2009empirical}]\label{prop:devEmpProcpointwise}
For any $0 \leq t \leq t_0 \leq \frac  1 2$ and any $\lambda >0$, we have
\begin{eqnarray*}
P \left(  \sqrt   n \left|\mathbb G_n (t)  - t  \right| \geq  \lambda   \right) 
& \leq&  2 \exp \left\{  -\frac{  \lambda^2}{2t  }  \Phi\left( \frac{\lambda} {t  \sqrt n}  \right) \right\} \\
& \leq&  2 \exp \left\{  - \frac {\lambda^2}{2t_0 }  \Phi\left( \frac{\lambda} {t_0  \sqrt n}  \right) \right\}\\
& \leq&  2 \exp \left(  - \frac{\lambda^2}{2t_0} \frac 1 {1 + \frac{\lambda} {3t_0 \sqrt n}}  \right).
\end{eqnarray*}
\end{Prop}
The first Inequality comes from Bennett's Inequality and from the fact that  $n \mathbb G_n (t)$ follows a Binomial$(n,t)$ distribution. The second Inequality derives from the fact that $\lambda \mapsto \lambda \Phi(\lambda)$  is a
non decreasing function, see Point (9) of Proposition 1 p.441 in \cite{shorack2009empirical}. The last inequality is Bernstein's Inequality, it can be derived   by upper bounding Bennett's Inequality with the following result, see Point (10) of Proposition~1 p.441 in \cite{shorack2009empirical}:
\begin{equation} \label{lbPsi}
\Phi(\lambda) \geq \frac{1}{1+ \frac \lambda 3} \quad \textrm{for any } \lambda \in \R.
\end{equation}

The famous DKW inequality \cite{dvoretzky1956asymptotic} gives an universal exponential inequality for empirical
processes. The tight constant comes from \cite{Massart1990}:
\begin{Theo} \label{DKW} For any $\lambda >0$:
$$ P \left( \sup_{t \in [0,1]}   \sqrt n    \left|  \mathbb G_n (t)  - t     \right| \geq  \lambda  \right)  \leq   2
\exp \left(  -   2 \lambda^2  \right) 
 . $$ 
\end{Theo}
However, in  the neighborhood of the origin, a tighter  uniform exponential inequality  can be given.
\begin{Prop}[Inequality~2 p. 444 in \cite{shorack2009empirical}] \label{prop:devEmpProcunif} Let $t_0 \in  (0,\frac 1
2)$. Then, for any
$\lambda >0$,
\begin{eqnarray*}
P \left( \sup_{t \in [0,t_0]}   \sqrt   n \left| \frac{ \mathbb G_n (t)  - t   }{1-t} \right| \geq \frac{\lambda}{1 -
t_0}  \right)  & \leq&  2 \exp \left\{  - \frac{ \lambda^2}{2t_0}  \Phi\left( \frac{\lambda} {t_0 \sqrt n}  \right) 
\right\} \\
& \leq&  2 \exp \left(  - \frac{\lambda^2}{2t_0} \frac 1 {1 + \frac{\lambda} {3t_0 \sqrt n}}  \right).
\end{eqnarray*}
\end{Prop}
This local result directly derives from   the fact that $ \left( \frac{\mathbb G_n (t) - t}{1-t}  \right)_{0\leq
t <1}$ is a martingale (Proposition~1 p.133 in \cite{shorack2009empirical}). 
The second inequality directly  derives from the previous one together with Inequality (\ref{lbPsi}). 

\subsection{Exponential inequalities for the uniform quantile process}

The general strategy followed in~\cite{shorack2009empirical} to prove exponential inequalities for the uniform quantile
process consists in rewriting inequalities on $\mathbb G_n^{-1}$ into inequalities on $\mathbb G_n$. For more details
see for instance the proof of Inequality 2 p.415, or Lemma~1 p. 457 in~\cite{shorack2009empirical}. We introduce the function $\tilde \Phi$
defined on $\R$ by 
\begin{equation*}
\tilde \Phi(\lambda) := \frac 1 {1+\lambda} \Phi\left(- \frac \lambda {1+\lambda} \right).
\end{equation*}
We give below  a point-wise exponential bound for the uniform quantile process. 
\begin{Prop}[Inequality~1 p. 453 in \cite{shorack2009empirical}] \label{prop:devquantilepointwise} For all  $\lambda >0$
and all $ 0 < u \leq u_0 < 1$, we have
\begin{eqnarray*} 
P \left( \sqrt   n \left|\mathbb G_n^{-1} (u)  - u \right| \geq \lambda \right)
&\leq& 2 \exp \left\{  - \frac{\lambda^2}{2 u} \tilde \Phi\left(\frac \lambda{u \sqrt n} \right) \right\} \\
&\leq& 2 \exp \left\{  - \frac{\lambda^2}{2 u_0} \tilde \Phi\left(\frac \lambda{u_0 \sqrt n} \right) \right\} \\
& \leq &  2 \exp \left(  - \frac{\lambda^2}{2 u_0}  \frac 1 {1 + \frac {2\lambda}{3 u_0 \sqrt n} }   \right) 
\end{eqnarray*} 
\end{Prop} \label{prop:devquantileunif}
 The second Inequality derives from the property that $\lambda\mapsto \lambda \tilde  \Phi(\lambda)$  is a nondecreasing function, see point (10) of Proposition~1 p.455 in \cite{shorack2009empirical}. The last inequality comes
from the following lower bound, see Point (12) of Proposition~1 in \cite{shorack2009empirical}:
\begin{equation} \label{lbPsitilde}
\tilde \Phi (\lambda) \geq \frac{1}{1+ \frac {2\lambda} 3} \quad \textrm{for any } \lambda \in \R^+.
\end{equation}

The following result is an uniform exponential inequality for the quantile process  in the
neighborhood of the origin.
\begin{Prop}[Inequality~2 p. 457 in \cite{shorack2009empirical}] \label{prop:Quantileunif} 
Let $u_0 \in  (0,\frac 1 2)$ and $ n \geq 1$.
Then, for any $\lambda >0$ such that 
\begin{equation} \label{condition-u0-Quantile}
\lambda  \leq \sqrt n  \left( \frac 1 2 - u_0 \right),
\end{equation}
we have
\begin{eqnarray*}
P \left( \sup_{t \in [0, u_0]}  \frac{ \sqrt n \left|\mathbb G_n^{-1} (u)  - u \right| }{1-u} \geq  \frac{\lambda}{1 -
u_0}  \right)
&\leq& 2 \exp \left\{ - \frac{\lambda^2}{2u_0} \tilde \Phi\left(  \frac {\lambda}{u_0 \sqrt n}  \right) \right\}\\
&\leq& 2 \exp \left(  - \frac{\lambda^2}{2 u_0}  \frac 1 {1 + \frac {2\lambda}{3 u_0 \sqrt n} }   \right) .
\end{eqnarray*} 
\end{Prop}
This first Inequality comes from Proposition~\ref{prop:devEmpProcunif}, the second Inequality is deduced from
the first one using (\ref{lbPsitilde}). 

\subsection{Le Cam's Lemma}

The  version of Le Cam's Lemma given below is from \cite{Yu97}. Recall that  the total variation
distance between two distributions $P_0$ and $P_1$ on a measured space $(\mathcal X, \mathcal B) $ is defined by 
$$ \TV(P_0,P_1)  = \sup_{B \in  \mathcal B} | P_0(B) - P_1(B) |.$$
Moreover, if $P_0$ and $P_1$ have densities $p_0$ and $p_1$ with respect to the same measure $\lambda$ on $\mathcal X$, then 
$$ \TV(P_0,P_1)    = \frac 1 2 \ell_1(p_0,p_1) := \int_{\mathcal X} |p_0-p_1|  d \lambda. $$

\begin{Lemma} \label{Lem:Lecam} Let $\mathcal P $ be a set of distributions. For $P \in  \mathcal P$, let $\theta(P)$
take values in a metric space $(\X,\rho)$. Let
$P_0$ and $P_1$ in $\mathcal P$ be any pair of distributions. Let $X_1,\dots,X_n$ be drawn i.i.d. from some $P \in 
\mathcal P$. Let $\hat 
\theta = \hat \theta(X_1,\dots,X_n) $ be any estimator of $\theta(P)$, then
\begin{equation*}
 \sup_{ P \in \mathcal P} \E _{P^n} \rho( \theta , \hat \theta  )  \geq   \frac 1 8    \rho \left( \theta(P_0),
\theta(P_1) \right)   
\left[1 -  \TV(P_0,P_1) \right] ^{2 n } .
\end{equation*}
\end{Lemma}

\paragraph{Acknowledgments:} The authors are grateful to J\'erome Dedecker for pointing out the key decomposition
Lemma~\ref{DecomposDeltaJerome} of the DTM.  The authors were supported by the ANR project TopData ANR-13-BS01-0008.
\bibliographystyle{apalike}
\bibliography{DTM}

\end{document}